\documentclass[a4paper,11pt]{amsart}
\usepackage{mathtools}
\usepackage{amsthm, amsfonts,amsmath,amscd,amssymb, amsrefs}
\usepackage{bbm}
\usepackage{quiver}
\usepackage{mdframed}
\usepackage{yhmath}
\usepackage[hidelinks]{hyperref}
\usepackage[margin=1.0in]{geometry}
\usepackage{xcolor}

\makeatletter
\@namedef{subjclassname@2020}{\textup{2020} Mathematics Subject Classification}
\makeatother

\title[Equivariant Vector Bundles with Connection on Drinfeld Symmetric Spaces]{Equivariant Vector Bundles with Connection \\ on Drinfeld Symmetric Spaces}
 \date{\today}
\author{James Taylor}
\email{james.taylor@math.unipd.it}
\address{\scriptsize Dipartimento di Matematica, Universit\`{a} degli Studi di Padova, Via Trieste 63, 35131 Padova, Italy}
\subjclass[2020]{14F10, 14G22, 14G32, 11S37.}
\keywords{Galois covering, $\sD$-modules, Drinfeld tower, finite bundles.}

\newcommand{\HH}{\mathrm{H}}
\newcommand{\OO}{\mathcal{O}}

\DeclareMathOperator{\Sym}{Sym}
\DeclareMathOperator{\Hom}{Hom}
\DeclareMathOperator{\Mod}{\bold{Mod}}
\DeclareMathOperator{\Vect}{\bold{Vect}}
\DeclareMathOperator{\Gal}{Gal}
\DeclareMathOperator{\Nrd}{Nrd}
\DeclareMathOperator{\GM}{GM}
\DeclareMathOperator{\Ind}{Ind}
\DeclareMathOperator{\Pic}{Pic}
\DeclareMathOperator{\cts}{cts}
\DeclareMathOperator{\op}{op}

\DeclareMathOperator{\fd}{\bold{fd}}
\DeclareMathOperator{\sm}{sm}
\DeclareMathOperator{\charfield}{char}
\DeclareMathOperator{\PicCon}{PicCon}
\DeclareMathOperator{\VectCon}{\bold{VectCon}}
\DeclareMathOperator{\Coh}{\bold{Coh}}
\DeclareMathOperator{\Con}{Con}
\DeclareMathOperator{\End}{End}
\DeclareMathOperator{\Rep}{\bold{Rep}}

\DeclareMathOperator{\rank}{rank}
\DeclareMathOperator{\equaliser}{eq}
\DeclareMathOperator{\ext}{ext}
\DeclareMathOperator{\tors}{tors}
\DeclareMathOperator{\fin}{fin}
\DeclareMathOperator{\ab}{ab}

\DeclareMathOperator{\id}{id}
\DeclareMathOperator{\la}{la}

\DeclareMathOperator{\coh}{c}
\DeclareMathOperator{\gr}{gr}

\DeclareMathOperator{\Irr}{Irr}
\DeclareMathOperator{\GL}{GL}

\DeclareMathOperator{\SL}{SL}
\DeclareMathOperator{\Aut}{Aut}
\DeclareMathOperator{\Stab}{Stab}
\DeclareMathOperator{\Sp}{Sp}

\DeclareMathOperator{\Alg}{\textbf{Alg}}
\DeclareMathOperator{\LR}{\textbf{LR}}
\DeclareMathOperator{\sLR}{\mathcal{L}\mathcal{R}}
\DeclareMathOperator{\Spec}{Spec}

\DeclareMathOperator{\Der}{Der}

\newcommand{\bC}{{\mathbb C}}
\newcommand{\bD}{{\mathbb D}}
\newcommand{\bF}{{\mathbb F}}

\newcommand{\bN}{{\mathbb N}}
\newcommand{\bP}{{\mathbb P}}
\newcommand{\bQ}{{\mathbb Q}}

\newcommand{\bZ}{{\mathbb Z}}

\newcommand{\sA}{{\mathcal A}}
\newcommand{\sB}{{\mathcal B}}
\newcommand{\sC}{{\mathcal C}}
\newcommand{\sD}{{\mathcal D}}

\newcommand{\sF}{{\mathcal F}}
\newcommand{\sG}{{\mathcal G}}
\newcommand{\sH}{{\mathcal H}}

\newcommand{\sL}{{\mathcal L}}
\newcommand{\sM}{{\mathcal M}}
\newcommand{\sN}{{\mathcal N}}

\newcommand{\sR}{{\mathcal R}}
\newcommand{\sS}{{\mathcal S}}
\newcommand{\sT}{{\mathcal T}}
\newcommand{\sU}{{\mathcal U}}
\newcommand{\sV}{{\mathcal V}}
\newcommand{\sW}{{\mathcal W}}

\newtheorem{thm}{Theorem}
\numberwithin{thm}{section}
\newtheorem{lemma}[thm]{Lemma} 
\newtheorem{prop}[thm]{Proposition} 
\newtheorem{cor}[thm]{Corollary} 
 
\newtheorem*{thm*}{Theorem}
\newtheorem*{theoremA}{Theorem A}
\newtheorem*{theoremB}{Theorem B}
\newtheorem*{theoremC}{Theorem C}
\newtheorem*{theoremD}{Theorem D}

\theoremstyle{definition}
\newtheorem{defn}[thm]{Definition}
\newtheorem{eg}[thm]{Example}
\newtheorem{remark}[thm]{Remark}
\newtheorem*{remark*}{Remark}

\DeclareRobustCommand{\SkipTocEntry}[5]{}

\begin{document}

\maketitle
\vspace{-0.5cm}
\begin{abstract}
For a finite extension $F$ of $\bQ_p$ and $n \geq 1$, let $D$ be the division algebra over $F$ of invariant $1/n$ and let $G^0$ be the subgroup of $\GL_n(F)$ of elements with norm $1$ determinant.

We show that the action of $D^\times$ on the Drinfeld tower induces an equivalence of categories from finite dimensional smooth representations of $D^\times$ to $G^0$-finite $\GL_n(F)$-equivariant vector bundles with connection on $\Omega$, the $(n-1)$-dimensional Drinfeld symmetric space.
\end{abstract}
\vspace{0.3cm}
\tableofcontents
\vspace{-0.7cm}
\section{Introduction}

Let $p$ be a prime, let $F$ be a finite extension of $\bQ_p$ and let $n \geq 1$. Let $K$ be a complete field extension of $L$, the completion of the maximal unramified extension of $F$. The \emph{Drinfeld tower} is a system of $(n-1)$-dimensional rigid analytic spaces over $K$,
\[
	\Omega \leftarrow \sM \leftarrow \sM_1 \leftarrow \sM_2 \leftarrow \cdots,
\]
for which each space is equipped with a compatible action of $D^\times \times \GL_{n}(F)$, where $D$ is the division algebra of invariant $1/n$ over $F$~\cite{DRI, BC, RZ}.

The space $\sM$ is a fundamental example of a Rapoport-Zink space, and the covering spaces $\sM, \sM_1, \sM_2, ...$ are indexed by the compact open subgroups $\OO_D^\times, 1 + \Pi \OO_D, 1 + \Pi^2 \OO_D^2, ...$ of $D^\times$, where $\Pi$ is a uniformiser of $\OO_D$ \cite[\S 4.25]{RAPVIE}. The action of $D^\times$ on each space comes from the action of $D^\times$ on the tower via Hecke correspondences \cite[\S 5.34]{RZ}, which stabilises each covering space because these compact open subgroups are all normal in $D^\times$.

These spaces play an important role in the representation theory of both $\GL_{n}(F)$ and $D^\times$. For example, this tower has been shown to provide a geometric realisation of both the local Langlands and Jacquet-Langlands correspondences for $\GL_{n}(F)$~\cite{CAR, HARR, BOY, HARTAY}. 

These correspondences are realised in the cohomology of the spaces $\sM_m$ for an appropriate cohomology theory. In this framework, one can also consider the coherent cohomology groups $\HH^{0}(\sM_m, \OO_{\sM_m}) = \OO(\sM_m)$, that vanish in all other degrees, for which the topological dual $\OO(\sM_m)^*$ is naturally a \emph{locally analytic} representation of $\GL_{n}(F)$. The isotypical parts of these representations lie in the image of the functor
\[
\Hom_{D^\times}(-, \OO(\sM_{\infty}))^* \colon \Rep_{\sm}^{\fd}(D^\times) \rightarrow \Rep_{\la}(\GL_n(F)),
\]
defined as the direct limit of the functors
\[
\Hom_{D^\times}(-, \OO(\sM_{m}))^* \colon \Rep^{\fd}(D^{(m)}) \rightarrow \Rep_{\la}(\GL_n(F)),
\]
where $D^{(m)} \coloneqq D^\times / (1 + \Pi^m \OO_D)$.
When $F=\bQ_p$ and $n=2$, Dospinescu and Le Bras have shown that for any irreducible (necessarily finite dimensional) smooth representation $V$ of $D^\times$ with trivial central character and dimension strictly greater than one, the corresponding locally analytic representation of $\GL_2(\bQ_p)$ is \emph{admissible} and \emph{topologically irreducible}. This is deduced as a consequence of much deeper results relating this locally analytic representation of $\GL_2(\bQ_p)$ to the Jacquet-Langlands and $p$-adic local Langlands correspondences \cite[Thm.\ 1.2]{DLB}.

One would like to deduce similar admissibility and topological irreducibility results beyond $\GL_2(\bQ_p)$, where a $p$-adic Langlands correspondence is not yet currently formulated. One natural approach is through the use of $p$-adic $\sD$-modules. The functor $\Hom_{D^\times}(-, \OO(\sM_{\infty}))^*$ above admits a natural factorisation
\[
\Rep_{\sm}^{\fd}(D^\times) \rightarrow \VectCon^{\GL_n(F)}(\Omega) \rightarrow \Rep(\GL_n(F))
\]
through the category $\VectCon^{\GL_n(F)}(\Omega)$ of $\GL_n(F)$-equivariant vector bundles with connection on $\Omega$, where the second functor is defined by taking the dual of the global sections $\Gamma(\Omega, -)^*$. In this paper we study the first functor, which we denote by
\[
\Hom_{D^\times}(-, f_* \OO_{\sM_{\infty}}) \colon \Rep_{\sm}^{\fd}(D^\times) \rightarrow \VectCon^{\GL_n(F)}(\Omega),
\]
which similarly to above is defined as the direct limit of functors
\[
\Hom_{D^\times}(-, f_{m,*} \OO_{\sM_{m}}) \colon \Rep^{\fd}(D^{(m)}) \rightarrow \VectCon^{\GL_n(F)}(\Omega),
\]
where $f_m \colon \sM_m \rightarrow \Omega$ is the Galois covering map defined by the Drinfeld tower.

This functor is also related to the Jacquet-Langlands correspondence: the composition
\[
\Rep_{\sm}^{\fd}(D^\times) \rightarrow \VectCon^{\GL_n(F)}(\Omega) \xrightarrow{\HH_{\text{dR}, c}^{n-1}(-)} \Rep(\GL_n(F))
\]
should send an irreducible representation $\rho$ of $D^\times$ with $\dim(\rho) > 1$ to the direct sum of $n$ copies of the smooth representation $\text{JL}(\rho)$. This is known for $\GL_2(F)$ \cite[Thm.\ 0.4]{CDN1}, and for certain representations of $D^\times$ corresponding to the first Drinfeld covering $\sM_1$ \cite[Thm.\ A]{JUNDR}.

Our main result is the following.
\begin{theoremA}
The functor
\[
\Hom_{D^\times}(-, f_* \OO_{\sM_{\infty}}) \colon \Rep_{\sm}^{\fd}(D^\times) \rightarrow \VectCon^{\GL_n(F)}(\Omega)
\]
is exact, monoidal, fully faithful, and the essential image is closed under sub-quotients.

The essential image is intrinsically described as the full subcategory
\[
	\VectCon^{\GL_n(F)}(\Omega)_{G^0-\fin}
\]
with objects those that are finite when viewed as $G^0$-equivariant vector bundles with connection.
\end{theoremA}
Here $G^0$ is the subgroup of elements of $\GL_n(F)$ with determinant of norm $1$, and for the definition of a finite equivariant vector bundle with connection we direct the reader to Section \ref{sect:finiteVB}. For now, let us simply remark that an equivariant \emph{line bundle} with connection is \emph{finite if and only if it is torsion}, and therefore finiteness can be viewed as a natural generalisation of the notion of a torsion line bundle to vector bundles of arbitrary rank. We also remark that the functor of Theorem A \emph{preserves irreducibility}, as the essential image is closed under sub-objects.

In order to explain the appearance of the group $G^0$ in the statement of Theorem A (which is essential - see Remark \ref{rem:GvG0}), let $\sN$ be a connected component of $\sM$, and consider
\[
	\sN \leftarrow \sN_1 \leftarrow \sN_2 \leftarrow \cdots ,
\]
the induced sub-tower of $(\sM_m)_{m \geq 1}$ defined by $\sN_m = \phi^{-1}_m(\sN)$ for $m \geq 1$. This is stabilised by $\OO_D^\times \times G^0$ (see Section \ref{sect:thedrinfeldtower}), and in the same way as above one obtains a functor
\[
\Hom_{\OO_D^\times}(-, f_* \OO_{\sN_{\infty}}) \colon \Rep_{\sm}^{\fd}(\OO_D^\times) \rightarrow \VectCon^{G^0}(\Omega)
\]
which is compatible with the functor of Theorem A in the sense that the diagram
\[\begin{tikzcd}
	{\Hom_{D^\times}(-, f_* \OO_{\sM_{\infty}})} \hspace{-1.25em} & \hspace{-1em} {\Rep_{\sm}^{\fd}(D^\times)} & {\VectCon^{\GL_n(F)}(\Omega)} \\
	{\Hom_{\OO_D^\times}(-, f_* \OO_{\sN_{\infty}})} \hspace{-1.35em} & \hspace{-1em} {\Rep_{\sm}^{\fd}(\OO_D^\times)} & {\VectCon^{G^0}(\Omega)}
	\arrow["\colon"{marking, allow upside down}, draw=none, from=1-1, to=1-2]
	\arrow[from=1-2, to=1-3]
	\arrow[from=1-2, to=2-2]
	\arrow[from=1-3, to=2-3]
	\arrow["\colon"{marking, allow upside down}, draw=none, from=2-1, to=2-2]
	\arrow[from=2-2, to=2-3]
\end{tikzcd}\]
with vertical forgetful maps commutes. We deduce the essential image statement of Theorem A above from the following analogue of Theorem A for $D^\times$ replaced with $\OO_D^\times$.
\begin{theoremB}
	The functor
	\[
	\Hom_{\OO_D^\times}(-, f_* \OO_{\sN_{\infty}}) \colon \Rep_{\sm}^{\fd}(\OO_D^\times) \rightarrow \VectCon^{G^0}(\Omega)
	\]
	is exact, monoidal, fully faithful and the essential image is closed under sub-quotients.

	The essential image is intrinsically described as the full subcategory
	\[
	\VectCon^{G^0}(\Omega)_{\fin}
	\]
	of finite $G^0$-equivariant vector bundles with connection.
\end{theoremB}
We prove the description of the essential image in Theorem B by establishing analogous results to Nori's \cite{NORI} relating finite vector bundles to Galois coverings (Section \ref{sect:finiteVB}), and applying these in conjunction with the factorisation theorem of Scholze-Weinstein \cite[Thm.\ 7.3.1]{SW}. The other properties of the functors of Theorem A and Theorem B follow from general results which we describe below regarding functors of this type associated to Galois coverings.

For now, let us describe instead how these results are related to and have the potential to lead to admissibility and topological irreducibility results. The main result of the recent work of Ardakov and Wadsley stated in our context is the following (in which $n = 2$).
\begin{thm*}[{\cite[Thm.\ A]{AW2}}]
	Suppose that $\sL \in \VectCon^{G^0}(\Omega)_{\fin}$ has rank $1$. Then the locally analytic representation $\Gamma(\Omega, \sL)^*$ of $G^0$ is admissible and has length at most $2$.
\end{thm*}
One of the main ingredients that goes into the proof of this theorem is an explicit classification of torsion $G^0$-equivariant line bundles with connection on $\Omega$: the main result of \cite{AW1} is an explicit group isomorphism
\[
	\text{PicCon}^{G^0}(\Omega)_{\tors} \xrightarrow{\sim} \Hom_{\sm}(\OO_D^\times, K^\times),
\]
where the left-hand side is the group of isomorphism classes of torsion $G^0$-equivariant line bundles with connection on $\Omega$. Therefore, in this context, Theorem A and Theorem B can be viewed as generalisations of this isomorphism in four different directions:
\begin{itemize}
	\item To representations of of arbitrary dimension and vector bundles of arbitrary rank,
	\item To Drinfeld spaces of any dimension,
	\item From the pair $(\OO_D^\times, G^0)$ to both pairs $(D^\times, \GL_n(F))$ and $(\OO_D^\times, G^0)$,
	\item From an isomorphism to a functorial correspondence.
\end{itemize}

For $\sL \in \VectCon^{G^0}(\Omega)$ which is rank $1$ and torsion, the results of \cite{AW2} regarding the length of the admissible representation $\Gamma(\Omega, \sL)^*$ are actually more precise: $\Gamma(\Omega, \sL)^*$ is also shown to be topologically irreducible if and only if $\sL$ is non-trivial when viewed as an object of $\VectCon(\Omega)$ (if $\sL \not\cong \OO_{\Omega}$ as vector bundles with connection once we forget the equivariant structure). We would like to similarly understand the restriction map 
\[
\VectCon^{G^0}(\Omega) \rightarrow \VectCon(\Omega)
\]
for objects in the image of $\Hom_{\OO_D^\times}(-, f_* \OO_{\sN_{\infty}})$.

Whenever $K$ contains $F^{\ab}$ (the maximal abelian extension of $F$) the spaces $\sM_m$ are disjoint unions of geometrically connected components, and we may fix a compatible sequence of geometrically connected components $\Sigma^m$ (see Section \ref{sect:forgettingequivstructure}) at each level to obtain a sub-tower
\[
\sN \leftarrow \Sigma^1 \leftarrow \Sigma^2 \leftarrow \cdots
\]
which similarly defines a functor
\[
\Hom_{\SL_1(D)}(-, f_* \OO_{\Sigma^{\infty}}) \colon \Rep_{\sm}^{\fd}(\SL_1(D)) \rightarrow \VectCon(\Omega)
\]
where $\SL_1(D) = \ker(\Nrd \colon D^\times \rightarrow F^\times)$. The following is the analogue of Theorem A for $\SL_1(D)$.
\begin{theoremC}
		Suppose that $K$ contains $F^{\ab}$. Then the functor
		\[
			\Hom_{\SL_1(D)}(-, f_* \OO_{\Sigma^{\infty}}) \colon \Rep_{\sm}^{\fd}(\SL_1(D)) \rightarrow \VectCon(\Omega)
		\]
		is exact, monoidal, fully faithful, the essential image is closed under sub-quotients and
\[\begin{tikzcd}
	{\Hom_{\OO_D^\times}(-, f_* \OO_{\sN_{\infty}})} \hspace{-2.3em} & \hspace{-1.5em} {\Rep_{\sm}^{\fd}(\OO_D^\times)} & {\VectCon^{G^0}(\Omega)} \\
	{\Hom_{\SL_1(D)}(-, f_* \OO_{\Sigma^{\infty}})} \hspace{-1.1em} & \hspace{-1em} {\Rep_{\sm}^{\fd}(\SL_1(D))} & {\VectCon(\Omega)}
	\arrow["\colon"{description}, draw=none, from=1-1, to=1-2]
	\arrow[from=1-2, to=1-3]
	\arrow[from=1-2, to=2-2]
	\arrow[from=1-3, to=2-3]
	\arrow["\colon"{description}, draw=none, from=2-1, to=2-2]
	\arrow[from=2-2, to=2-3]
\end{tikzcd}\]
commutes.
\end{theoremC}
Therefore, for any finite dimensional smooth representation $V$ of $D^\times$, the underlying $\sD$-module structure of $\Hom_{D^\times}(V, f_* \OO_{\sM_{\infty}})$ is completely determined by the restriction of $V$ to a representation of $\SL_1(D)$. Similarly, with the results of Ardakov and Wadsley described above, this tells us that for any smooth character $\chi$ of $\OO_D^\times$, the locally analytic representation $\Hom_{\OO_D^\times}(\chi, \OO(\sN_{\infty}))^*$ of $G^0$ is topologically irreducible precisely when $\chi|_{\SL_1(D)} \neq 1$.

We also note in passing that all the functors described above send a representation $V$ to a vector bundle of rank $\dim(V)$, and commute with duals and taking the determinant representation and determinant line bundle on either side (Remark \ref{rem:dimrank}).

\addtocontents{toc}{\SkipTocEntry}
\subsection*{Functors Associated to Galois Coverings}

We now give an overview of how the functors of Theorems A, B and C are defined, and how their properties are established. Each functor is constructed as the direct limit of certain functors attached to each of the finite level Galois coverings of each respective tower. We now describe how we define each of these functors at finite level. Suppose in what follows that $k$ is any characteristic $0$ field.

Suppose first that $X$ is a smooth scheme or rigid space over $k$, with an action of an abstract group $H$. Suppose moreover that $X$ has an additional action of some abstract group $G$, and that this action commutes with the action of $H$. In this situation, we can consider the category $\VectCon^{G \times H}(X)$, and there are canonically defined functors
\begin{align*}
	\OO_X \otimes_k - &\colon \Mod_{k[H]}^{\fd} \rightarrow \VectCon^{G \times H}(X), \\
	\Hom_{G\text{-}\sD_X}(\OO_X, -) &\colon \VectCon^{G \times H}(X) \rightarrow \Mod_{k[H]}^{\fd},
\end{align*}
The following result is the main technical ingredient. To state it, let $c_X$ denote the sheaf of \emph{constant functions} (considered in more detail in Section \ref{sect:constantsheaf}), defined by
\[
c_X = \ker(d \colon \OO_X \rightarrow \Omega^1_{X/k}).
\]
\begin{theoremD}[Theorem \ref{FunctorkGGDX}]
	Suppose that $c_X(X)^G = k$. Then:
	\begin{enumerate}
		\item The functor
	\[
		\OO_X \otimes_k - \colon \Mod_{k[H]}^{\fd} \rightarrow \VectCon^{G \times H}(X)
	\]
	is exact and fully faithful.
	\item 
	For any $\sM \in \VectCon^{G \times H}(X)$, 
	\[
		\dim_k(\Hom_{G \text{-}\sD_X}(\OO_X, \sM)) \leq \rank(\sM).
	\]
	The essential image of $\OO_X \otimes_k -$ is the full subcategory with objects $\sM$ for which this is an equality. On this full subcategory the solution functor $\Hom_{G\text{-}\sD_X}(\OO_X, -)$ is a quasi-inverse for $\OO_X \otimes_k -$.
	\item The essential image of $\OO_X \otimes_k -$ is closed under sub-quotients.
	\end{enumerate}
\end{theoremD}
\begin{remark*}
	When $G$ is trivial, the assumption of the theorem that $c_X(X)^G = k$ is equivalent to the assumption that $X$ is geometrically connected. This can easily be shown to be a necessary condition for each conclusion of the theorem to hold (Remark \ref{rmk:geomconnnec}).
\end{remark*}
Suppose now that $f \colon X \rightarrow Y$ is a $G$-equivariant finite \'{e}tale Galois morphism of smooth schemes or rigid spaces over $k$, with Galois group $H$. In this situation there is an equivalence
\[
	(-)^H \colon \VectCon^{G \times H}(X) \rightarrow \VectCon^G(Y),
\]
and the functor
\[
\Hom_{k[H]}(-, f_* \OO_X) \colon \Mod_{k[H]}^{\fd} \rightarrow \VectCon^G(Y)
\]
can be shown to be equal to the composition
\[
	(\OO_X \otimes_k (-)^*)^H \colon \Mod_{k[H]}^{\fd} \rightarrow \VectCon^G(Y)
\]
and thus inherits all the good properties of $\OO_X \otimes_k -$ from Theorem D. Furthermore, this functor sends the regular representation $k[H]$ to the pushforward $f_*\OO_X$ (Theorem \ref{mainthm1}), and in this way we obtain a complete description of $f_*\OO_X$ as a semisimple $G\text{-}\sD_Y$-module (Corollary \ref{maincor2}).

For example, we deduce the relevant properties of the functor of Theorem B from the properties of the functors at each finite level $m \geq 1$, which are deduced from the above by taking $X = \sN_m$, $H = \OO_D^\times / (1 + \Pi^m \OO_D)$ and $G = G^0$. 

\addtocontents{toc}{\SkipTocEntry}
\subsection*{Other Results}

Let us now briefly mention some other results we establish in this paper which are of independent interest.

The first is Proposition \ref{prop:continuousactionlifts}, which in particular implies that the action of $\GL_n(F)$ on each covering space $\sM_m$ of the Drinfeld tower is continuous in the sense of \cite{ARD} (Corollary \ref{cor:ctsactionGLnM}).

The second is a description of the geometrically connected components of $(\sN_m)_{m \geq 1}$ (Theorem \ref{thm:geomconncompdescDT}). This description can already be found in the unpublished work of Boutot and Zink \cite[Thm.\ 0.20]{BZ}, which uses global methods and $p$-adic uniformisation of Shimura curves. In contrast, we deduce this by completely elementary methods. Precisely, this description is a simple application of the theory we develop in Section \ref{sect:constantsheaf} regarding the sheaf $c_X$ of constant functions, coupled with a result of Kohlhaase on maximal fields contained inside the global sections of $\sN_m$ \cite[Prop.\ 2.7]{KOH}.

The third (Corollary \ref{secondmapinjective}) is a proof that any $p$-torsion $\SL_n(F)$-equivariant line bundle with connection on $\Sigma^1$ is uniquely determined by its underlying line bundle. We use this to give a more conceptual proof of the main result of our earlier paper \cite[Thm.\ 4.6]{TAY}, which we show follows from the general theory we develop regarding how the pushforward of the structure sheaf of an abelian Galois covering decomposes as a $\sD$-module (Theorem \ref{mainthmabelian}).

\addtocontents{toc}{\SkipTocEntry}
\subsection*{Outline of the Paper}

This paper is organised as follows.

In Section \ref{sect:background} we collect the necessary facts we need concerning $\sD$-modules, equivariant sheaves and Galois coverings. The results of this section are quite general, and the reader is encouraged to skip this section and refer back to it when necessary.

In Section \ref{sect:constantsheaf} we introduce and consider the sheaf of constant functions and relate this sheaf to the notion of geometric connectivity.

In Section \ref{sect:RH} we introduce the functor $\OO_X \otimes_k -$, the solution functor $\Hom_{G \text{-}\sD_X}(\OO_X, -)$, and prove Theorem D (Theorem \ref{FunctorkGGDX}).

In Section \ref{sect:finiteVB} we consider finite equivariant vector bundles with connection, and relate these to Galois coverings.

In Section \ref{sect:DmodGalExt} we establish the main properties of the functor $\Hom_{k[N]}(-, f_* \OO_X)$ (Theorem \ref{mainthm1}) and from this prove a decomposition theorem for the $G\text{-}\sD_Y$-module $f_* \OO_X$ (Corollary \ref{maincor2}).

In Section \ref{sect:final} we apply the results of Section \ref{sect:DmodGalExt} to the Drinfeld tower. The proof of the main results, Theorem A, Theorem B and Theorem C occupies Section \ref{sect:equivariantDmodulesonDrinfeldSpaces} to Section \ref{sect:essimage}.

We also include two appendices. In Appendix \ref{sect:generalities} we show that the property that a group action on a rigid space is continuous lifts along finite \'{e}tale covers, which we use in Section \ref{sect:final}. In Appendix \ref{sect:basechange} we study base change functors on $\VectCon^G(X)$ for quasi-Stein $X$, and show that these are compatible with taking homomorphisms. We use this in our proof of Theorems A and B to pass from the case when $K$ is algebraically closed to general fields $K$. 

\addtocontents{toc}{\SkipTocEntry}
\subsubsection*{Acknowledgements}

I would like to thank Konstantin Ardakov, Simon Wadsley, James Newton, Alex Horawa, Tom Adams and Jan Kohlhaase for their comments and suggestions. I would also like to thank Andreas Bode for his help regarding some questions in $p$-adic functional analysis, and sincerely thank the two anonymous referees for their many detailed and insightful comments which greatly improved the paper. This research was financially supported by the EPSRC.

\section{Preliminary Notions}\label{sect:background}

In this section we collect together all the relevant technical notions we will make use of and refer back to when necessary throughout the rest of the paper. As an overview, in this paper we are concerned with certain categories of equivariant $\sD$-modules (Section \ref{sect:generaleqDmodules}) which are defined in terms of the sheaf of differential operators (Section \ref{sect:Dmodules}) and equivariant sheaves (Section \ref{eqsheaves}). We are particularly interested in how for a Galois extension (Section \ref{sect:galoisextensions}) the various categories we obtain are related to one another (Section \ref{sect:equivofcats}). Throughout the paper we work with $\sD$-modules, and the universal property of the sheaf of differential operators comes from its description as a particular example of a Lie algebroid (Section \ref{sect:LieAlgebroid}). Lie algebroids are objects which are locally described by Lie-Rinehart algebras (Section \ref{sect:LRalgebra}), and so this is where we begin.

\subsection{Lie-Rinehart Algebras}\label{sect:LRalgebra}

In this section, let $R$ be a commutative ring and let $A$ be a commutative $R$-algebra. An $(R,A)$-\emph{Lie algebra} is a pair $(L,\rho)$, where $L$ is simultaneously a $R$-Lie algebra and $A$-module, and $\rho : L \rightarrow \Der_R(A)$ is a homomorphism of both $R$-Lie algebras and $A$-modules such that
\[
	[x,ay] = a[x,y] + \rho(x)(a)y
\]
for any $x,y \in L$ and $a \in A$. Often we write only $L$ when $\rho : L \rightarrow \Der_R(A)$ is implicit.
\begin{eg}
$\Der_R(A)$ is an $(R,A)$-Lie algebra, with $\rho$ as the identity.
\end{eg}
From such a pair, one can construct the universal enveloping algebra of $L$ \cite{RINE}.
This is an $R$-algebra $U(L)$, together with structure maps
\begin{align*}
\iota_L : L \rightarrow U(L), \\
\iota_A : A \rightarrow U(L),
\end{align*}
which are morphisms of $R$-Lie algebras and of $R$-algebras respectively, satisfying
\begin{align*}
\iota_L(ax) &= \iota_A(a)\iota_L(x), & [\iota_L(x), \iota_A(a)] &= \iota_A(\rho(x)(a)),
\end{align*}
for all $a \in A, x \in L$. The triple $(U(L), \iota_A, \iota_L)$ satisfies the following universal property.
\begin{lemma}[{\cite[\S 2]{RINE}}]\label{UEAUnivProp}
Suppose that $S$ is a unital associative $R$-algebra,
\begin{align*}
	\eta_L : L \rightarrow S, \\
	\eta_A : A \rightarrow S,
\end{align*}
are morphisms of $R$-Lie algebras and of $R$-algebras respectively, and that for any $a \in A, x \in L$,
\begin{align*}
	\eta_L(ax) &= \eta_A(a)\eta_L(x), & [\eta_L(x), \eta_A(a)] &= \eta_A(\rho(x)(a)).
\end{align*}
Then there is a unique homomorphism of $R$-algebras $\varphi : U(L) \rightarrow S$ with
\[\begin{tikzcd}
	A & {U(L)} & L \\
	& S
	\arrow["{\iota_L}"', from=1-3, to=1-2]
	\arrow["{\iota_A}", from=1-1, to=1-2]
	\arrow["\varphi"', from=1-2, to=2-2]
	\arrow["{\eta_A}"', from=1-1, to=2-2]
	\arrow["{\eta_L}", from=1-3, to=2-2]
\end{tikzcd}\]
\end{lemma}
\begin{eg}
When $A = R$, and $\rho \colon L \rightarrow \Der_R(R) = 0$ is the zero map, $U(L)$ is the classical universal enveloping algebra of $L$ over $R$. However, despite the notation (which doesn't mention $A$), these differ in general when $A \neq R$.
\end{eg}
\begin{eg}
Using the universal property with $S = \End_R(A)$, we see that $A$ is canonically a $U(L)$-module where $A$ acts by left multiplication and $L$ acts by $\rho \colon L \rightarrow \Der_R(A) \subset \End_R(A)$.
\end{eg}
The morphism $\iota_A : A \rightarrow U(L)$ is injective, and by construction $U(L)$ is generated as an $R$-algebra by the images $\iota_A(A)$ and $\iota_L(L)$ \cite[\S 2]{RINE}. There is a increasing exhaustive filtration on $U(L)$, where $F_0U(L) = A$, $F_1U(L) = A + \iota_L(L)$, and
\[
	F_n U(L) = F_1 U(L) \cdot F_{n-1}U(L) = A + \sum_{i = 1}^n \iota_L(L)^i
\]
for any $n \geq 2$. $A$ is central in $\gr U(L)$, and there is a natural surjection $\Sym_A(L) \rightarrow \gr U(L)$, which is an isomorphism whenever $L$ is projective \cite[Thm.\ 3.1]{RINE}. When $\Sym_A(L) \rightarrow \gr U(L)$ is an isomorphism the natural map $\iota_A \oplus \iota_L : A \oplus L \rightarrow U(L)$ is injective, and in this case we will identify both $A$ and $L$ with their images in $U(L)$.
\subsubsection{Functoriality}
In this section we describe in what sense the construction of the universal enveloping algebra is functorial. Fixing the commutative ring $R$, we have the following notion from \cite[\S 2.1]{ARD}.
\begin{defn}\label{phimorphism}
Suppose that $\varphi : A \rightarrow B$ is a morphism of commutative $R$-algebras. Suppose that $L$ is an $(R,A)$-Lie algebra, and $L'$ is an $(R,B)$-Lie algebra. Then $\tilde{\varphi} : L \rightarrow L'$ is a $\varphi$\emph{-morphism} if
\begin{itemize}
	\item $\tilde{\varphi}$ is a homomorphism of $R$-Lie algebras,
	\item $\tilde{\varphi}(a \cdot x) = \varphi(a) \cdot \tilde{\varphi}(x)$,
	\item $\varphi(\rho(x)a) = \rho'(\tilde{\varphi}(x))\varphi(a)$,
\end{itemize}
for all $a \in A$, $x \in L$.
\end{defn}

Given a commutative ring $R$, we denote by $\LR_R$ the category with objects consisting of pairs $(A,L)$, where $A$ is a commutative $R$-algebra and $L$ is an $(R,A)$-Lie algebra. Morphisms $(A, L) \rightarrow (B, L')$ are pairs $(\varphi, \tilde{\varphi})$, where $\varphi : A \rightarrow B$ is a homomorphism of $R$-algebras, and $\tilde{\varphi} : L \rightarrow L'$ is a $\varphi$-morphism. It is straightforward to check that the composition of morphisms is again a morphism.

\begin{lemma}[{\cite[Lem.\ 2.1.7]{ARD}}]\label{UEAfunctoriality}
	Suppose that $\varphi : A \rightarrow B$ is a morphism of commutative $R$-algebras. Suppose that $L$ is an $(R,A)$-Lie algebra, $L'$ is an $(R,B)$-Lie algebra. Then every $\varphi$-morphism $\tilde{\varphi} : L \rightarrow L'$ extends uniquely to a filtration preserving $R$-algebra homomorphism $U(\varphi, \tilde{\varphi}) : U(L) \rightarrow U(L')$ that makes the following diagram commute
\[\begin{tikzcd}
	{A \oplus L} & {B \oplus L'} \\
	{U(L)} & {U(L')}
	\arrow["{\iota_A \oplus \iota_L}"', from=1-1, to=2-1]
	\arrow["{\iota_{B} \oplus \iota_{L'}}", from=1-2, to=2-2]
	\arrow["{\varphi \oplus \tilde{\varphi}}", from=1-1, to=1-2]
	\arrow["{U(\varphi, \tilde{\varphi})}"', from=2-1, to=2-2]
\end{tikzcd}\]
\end{lemma}

Therefore, $U(-)$ is a functor from the category $\LR_R$ to the category of positively filtered $R$-algebras.

\subsubsection{Base Change}

We discuss the notion of base change for Lie-Rinehart Algebras, following \cite[\S 2.2]{AW1}. In this section, suppose that $L$ is an $(R,A)$-Lie algebra.

\begin{lemma}[{\cite[Lem.\ 2.2]{AW1}}]\label{basechangelemma}
Suppose $\varphi : A \rightarrow B$ is a homomorphism of commutative $R$-algebras, and that $\sigma : L \rightarrow \Der_R(B)$ is a $\varphi$-morphism. Then the $B$-module $B \otimes_A L$ has a unique structure of an $R$-Lie algebra such that 
\[
(B \otimes_A L, 1 \otimes \sigma :B \otimes_A L \rightarrow \Der_R(B))
\]
is an $(R,B)$-Lie algebra.
\end{lemma}

\begin{remark}\label{funtorbasechangeLRalgebras}
As a consequence of Lemma \ref{basechangelemma}, if $\varphi : A \rightarrow B$ is a homomorphism of commutative $R$-algebras, and $\psi \colon \Der_R(A) \rightarrow \Der_R(B)$ is a $\varphi$-morphism, then $\sigma \coloneqq \psi \circ \rho \colon L \rightarrow \Der_R(B)$ is a $\varphi$-morphism, and we have a well defined functor
\[
B \otimes_A - : (L, \rho) \mapsto (B \otimes_A L, 1 \otimes (\psi \circ \rho))
\]
from $(R,A)$-Lie algebras to $(R,B)$-Lie algebras. With respect to this structure, the natural map $L \rightarrow B \otimes_A L$ is a $\varphi$-morphism.
\end{remark}

In the case when $L$ is projective as an $A$-module, we can describe the universal enveloping algebra of the base change.

\begin{lemma}\label{UEAbasechange}
	Suppose $\varphi : A \rightarrow B$ is a homomorphism of commutative $R$-algebras, $\sigma : L \rightarrow \Der_R(B)$ is a $\varphi$-morphism, and $L$ is is projective. Then the natural map
	\[
		U(L) \rightarrow U(B \otimes_A L)
	\]
	of filtered $R$-algebras induces isomorphisms
	\[
		B \otimes_A U(L) \rightarrow U(B \otimes_A L), \qquad U(L) \otimes_A B \rightarrow U(B \otimes_A L),
	\]
	of left and right $B$-modules respectively.
\end{lemma}

\begin{proof}
	This is \cite[Prop.\ 2.3]{AW1}, the proof of which applies whenever the natural map $\Sym_A(L) \rightarrow \gr U(L)$ is an isomorphism, and therefore in particular whenever $L$ is projective, by Rinehart's Theorem \cite[Thm.\ 3.1]{RINE}.
\end{proof}

\begin{lemma}\label{basechangefactorisation}
Suppose that $\varphi : A \rightarrow B$ is a morphism of commutative $R$-algebras, $\psi \colon \Der_R(A) \rightarrow \Der_R(B)$ is a $\varphi$-morphism, and $\tilde{\varphi} : L \rightarrow L'$ is a $\varphi$-morphism. Then the natural map
\[
	U(\varphi, \tilde{\varphi})_B \colon B \otimes_A U(L) \rightarrow U(L')
\]
factors as
\[
	B \otimes_A U(L) \rightarrow U(B \otimes_A L) \rightarrow U(L'),
\]
where we consider $B \otimes_A L$ as an $(R,B)$-Lie algebra as in Remark \ref{funtorbasechangeLRalgebras}. In particular, $U(\varphi, \tilde{\varphi})_B$ is an isomorphism whenever $L$ is projective and $\tilde{\varphi}_B : B \otimes_A L \rightarrow L'$ is an isomorphism.
\end{lemma}

\begin{proof}
	In the category $\LR_R$, $(\varphi, \tilde{\varphi}) : L \rightarrow L'$, factors as the composition of
	\[
		(\varphi, i_L) :  L \rightarrow B \otimes_A L, \qquad i_L(m) = 1 \otimes m,
	\]
	and 
	\[
		(\id_B, \tilde{\varphi}_B): B \otimes_A L \rightarrow L', \qquad \tilde{\varphi}_B(b \otimes m) = b \tilde{\varphi}(m),
	\]
	and therefore $U(L) \rightarrow U(L')$ factorises as
	\[
		U(L) \rightarrow U(B \otimes_A L) \rightarrow U(L').
	\]
	Then the factorisation result follows from the $B$-linearity of $U(B \otimes_A L) \rightarrow U(L')$, and the final claim follows from Lemma \ref{UEAbasechange}.
\end{proof}

\subsection{Geometric Setup}\label{sect:geometricsetup}

From now on, $k$ will denote a field. Throughout we will work with a quadruple $(X, \sB, \sB', \Omega_{X / k})$ which is as in one of the following settings.

\begin{enumerate}
	\item[\textbf{(A)}]
	$X$ is a scheme over $k$, $\sB = \sB'$ is the set of affine open subsets of $X$, and $\Omega_{X / k}$ is the sheaf of relative differentials of $X$ over $k$.
	\item[\textbf{(B)}] 
	$X$ is a rigid space over $k$, $\sB$ is the set of admissible affinoid open subsets of $X$, $\sB'$ is the set of quasi-Stein admissible open subsets of $X$, and $\Omega_{X/k}$ is the sheaf of relative differentials of $X$ over $k$ (as described in \cite{FRGIII} and \cite[\S 4.4]{FvdP}).
\end{enumerate}

In order to give a uniform approach to both cases, we define a \emph{space over $k$} to be an $X$ as in either case (A) or case (B), and a \emph{morphism of spaces} over $k$ to be a morphism $X \rightarrow Y$, where $X$ and $Y$ are either both in case (A) or both in case (B).

\subsection{Sheaves on a Basis}

For any $X$ as in Section \ref{sect:geometricsetup}, we view $X$ as a G-topological space in the sense of \cite[\S 9.1.1]{BGR}. Specifically, in case (A) we view $X$ with the Zariski topology, and in case (B) we view $X$ with its natural G-topology. With respect to this G-topology on $X$, the sets $\sB$ and $\sB'$ are each a basis for the G-topology of $X$, meaning that every open subset of $X$ has an admissible open covering by elements of $\sB$. Note there is no assumption that a basis in this sense is closed under intersections. For a basis $\sA$, there is a notion of a sheaf on a basis \cite[Def. 9.1]{AW1}, and the restriction functor induces an equivalence of categories from sheaves on $X$ to sheaves on $\sA$ \cite[Thm.\ 9.1]{AW1}.

Suppose in what follows that $\sA = \sB$ or $\sA = \sB'$. Given $\sF$ a sheaf on $\sA$, the extension $\sF^{\ext}$ to a sheaf on $X$ is defined as follows \cite[Prop.\ A.2]{AW1}. For any admissible open subset $U \subset X$ and any covering $\sU = \{U_i\}_{i}$ of $U$, for each pair $i,j$ let $\{W_{ijk}\}_{k}$ be an admissible open covering of $U_i \cap U_j$ by elements of $\sA$. Set
\[
	\HH^0(\sU, \sF) \coloneqq \equaliser \left( \prod_i \sF(U_i) \rightrightarrows \prod_{i,j,k} \sF(W_{ijk}) \right),
\]
which is independent of the choice of the coverings $\{W_{ijk}\}_{k}$, and set
\[
	\sF^{\ext}(U) \coloneqq \varinjlim_{\sU} \HH^0(\sU, \sF).
\]
Now suppose that $\sF$ is a presheaf on $X$, such that $\sF|_{\sA}$ is a sheaf on the basis $\sA$. In this case $\sF^{\ext}$ is a sheaf, and there is a natural morphism of presheaves $\sF \rightarrow \sF^{\ext}$ defined by
\[
	\sF(U) \rightarrow \HH^0(\sU, \sF) \rightarrow \sF^{\ext}(U),
\]
where we choose any covering $\sU$ of $U$, and $\sF(U) \rightarrow \HH^0(\sU, \sF)$ is the natural restriction map. This morphism is independent of the choice of $\sU$. The following is direct to verify.

\begin{lemma}\label{extensionsheafification}
Suppose that $\sF$ is a presheaf on $X$ such that $\sF|_{\sA}$ is a sheaf on the basis $\sA$. Then the morphism $\sF \rightarrow \sF^{\ext}$ defined above is a sheafification of the presheaf $\sF$.
\end{lemma}

\subsection{Tangent Sheaf}\label{sect:tangentsheaf}

In the following lemma, which will be fundamental to our constructions, we work with a pair $(\varphi \colon A \rightarrow B, \Omega_{A/k})$ which is as in one of the following cases, each the affine version of the corresponding geometric framework of Section \ref{sect:geometricsetup}.
\begin{enumerate}
	\item[\textbf{(A)}]
	$\varphi : A \rightarrow B$ is an \'{e}tale morphism of commutative $k$-algebras, $\Omega_{A/k}$ is the module of K\"{a}hler differentials of $A$ over $k$,
	\item[\textbf{(B)}] 
	$\varphi : A \rightarrow B$ is an \'{e}tale morphism of affinoid algebras over $k$, $\Omega_{A/k}$ is the universal finite differential module of $A$ over $k$ \cite[\S 3.6]{FvdP}.
\end{enumerate}
\begin{lemma}\label{extendderivations}
	Let $\varphi : A \rightarrow B$ be as in (A) or (B) above. Then any $\partial \in \Der_k(B)$ is determined by its restriction to $A$, and there exists a $\varphi$-morphism
	\[
	\psi \colon \Der_k(A) \rightarrow \Der_k(B)
	\]
	uniquely determined as a function $\Der_k(A) \rightarrow \Der_k(B)$ by the property that for any $\partial \in \Der_k(A)$
	\[
	\psi(\partial) \circ \varphi = \varphi \circ \partial.
	\]
	If furthermore $\Omega_{A/k}$ is finitely generated projective over $A$, then the natural map
	\[
		B \otimes_A \Der_k(A) \rightarrow \Der_k(B), \qquad b \otimes \partial \mapsto b\psi(\partial),
	\]
	is an isomorphism.
\end{lemma}
\begin{proof}
	In either case (A) or case (B), because $\varphi : A \rightarrow B$ is \'{e}tale, $B \otimes_A \Omega_{A/k} \rightarrow \Omega_{B/k}$ is an isomorphism. In case (B) this is \cite[Prop.\ 3.5.3(i)]{BER}, noting that Berkovich's definition \cite[\S 3.3]{BER} of $\Omega_{A/k}$ agrees with that which we use, as explained in \cite[Remarks 3.6.2]{FvdP}. In case (A), locally in the Zariski topology $\psi$ is standard \'{e}tale \cite[Lem.\ 02GT]{STACK}, and for standard \'{e}tale extensions the map is an isomorphism \cite[Example 6.1.12]{QL}. Taking the $B$-linear dual, we obtain the following commutative diagram,
\[\begin{tikzcd}
	{\Hom_B(\Omega_{B/k}, B)} & {\Der_k(B)} \\
	{\Hom_B(B \otimes_A \Omega_{A/k}, B)} \\
	{\Hom_A(\Omega_{A/k}, B)} & {\Der_k(A,B)} \\
	{B \otimes_A\Hom_A(\Omega_{A/k},  A)} & {B \otimes_A \Der_k(A)} \\
	{\Hom_A(\Omega_{A/k},  A)} & {\Der_k(A)}
	\arrow["\sim", from=1-1, to=1-2]
	\arrow["\sim", from=3-1, to=3-2]
	\arrow["\sim", from=5-1, to=5-2]
	\arrow[from=5-1, to=4-1]
	\arrow[from=5-2, to=4-2]
	\arrow[from=4-2, to=3-2]
	\arrow[from=4-1, to=3-1]
	\arrow[from=1-2, to=3-2]
	\arrow["\sim"', from=1-1, to=2-1]
	\arrow["\sim"', from=2-1, to=3-1]
	\arrow["\sim", from=4-1, to=4-2]
\end{tikzcd}\]
In particular, we see that the restriction map,
\[
	\Der_k(B) \rightarrow \Der_k(A,B),
\]
is an isomorphism. The composite $A$-linear map,
\[
	\psi : \Der_k(A) \rightarrow \Der_k(B),
\]
satisfies,
\[
	\psi(\partial) \circ \varphi = \varphi \circ \partial,
\]
for any $\partial \in \Der_k(A)$, and consequently is the unique function $\Der_k(A) \rightarrow \Der_k(B)$ with this property. For any $\partial_1, \partial_2 \in \Der_k(A)$, $\psi([\partial_1, \partial_2])$ and $[\psi(\partial_1), \psi(\partial_2)]$ agree on $\varphi(A)$, and thus are equal, hence $\psi : \Der_k(A) \rightarrow \Der_k(B)$ is also a homomorphism of $k$-Lie algebras.

For the second claim, if $\Omega_{A/k}$ is finitely generated projective over $A$, then the map,
\[
B \otimes_A\Hom_A(\Omega_{A/k},  A) \rightarrow \Hom_B(B \otimes_A \Omega_{A/k}, B),	
\]
is an isomorphism \cite[Chapter II, \S 5, Prop. 7]{BOUR}. Therefore, the composite,
\[
	\gamma \colon B \otimes_A \Der_k(A) \rightarrow \Der_k(B), \qquad \gamma(b \otimes \partial) = b \psi(\partial)
\]
is an isomorphism.
\end{proof}

\begin{defn}
	If $\sF$, $\sG$ are sheaves of $k$-algebras on $X$, then a \emph{$k$-derivation from $\sF$ to $\sG$} is a morphism $\partial : \sF \rightarrow \sG$ of sheaves of $k$-vector spaces on $X$ such that for any admissible open subset $U$ of $X$ and $x,y \in \sF(U)$,
	\[
		\partial_U(xy) = x \partial_U(y) + \partial_U(x) y.
	\]
\end{defn}
\begin{defn}	
	The tangent sheaf $\sT_X$ on $X$ is subsheaf of $\underline{\End}_k(\OO_X)$ which has has value
	\[
	\sT_X(U) = \{f \in \End_{k}(\OO_X|_U) \mid f \text{ is a $k$-derivation}\}
	\]
	for an admissible open subset $U \subset X$.
\end{defn}

Note that for any admissible open subset $U \subset X$ (not necessarily in $\sB$), taking sections over $U$ defines action map $\sT_X(U) \rightarrow \Der_k(\OO_X(U))$.

\begin{remark}\label{rem:localdescTX}
	Suppose that $\Omega_{X/k}$ is locally free of finite rank. Then for any $U \in \sB$ the natural map from the sheaf associated to the $\OO_X(U)$-module $\Der_k(\OO_X(U))$,
	\[
		\widetilde{\Der_k(\OO_X(U))} \rightarrow \sT_X|_U,
	\]
	is an isomorphism by Lemma \ref{extendderivations}.
\end{remark}

\subsection{Lie Algebroids}\label{sect:LieAlgebroid}

We follow \cite{BB} and \cite{AW1} in the following definition. 
 
\begin{defn}
A \emph{Lie algebroid} on $X$ is a pair $(\rho, \sL)$ where,
\begin{itemize}
	\item $\sL$ is a coherent sheaf of $\OO_X$-modules,
	\item $\sL$ has the structure of a sheaf of $k$-Lie algebras,
	\item $\rho : \sL \rightarrow \sT_X$ is an $\OO_X$-linear morphism of sheaves of $k$-Lie algebras such that,
	\[
		[x, ay] = a[x,y] + \rho_U(x)(a)y,
	\]
	for any admissible open subset $U$, $x,y \in \sL(U)$, and $a \in \OO_X(U)$.
\end{itemize}
The pair $(\rho, \sL)$ is called \emph{smooth} if $\sL$ is locally free of finite rank.
\end{defn}

\begin{eg}\label{eg:mainegliealgebroid}
The pair $(\id, \sT_X)$ is a Lie algebroid, which is smooth if and only if $\sT_X$ is locally free of finite rank.
\end{eg}

\begin{remark}\label{liealgebroidlocallyLRAlgebra}
If $(\rho, \sL)$ is a Lie algebroid on $X$, then for any admissible open subset $U$ of $X$, $(\rho_U', \sL(U))$ is an $(k,\OO_X(U))$-Lie algebra, where $\rho'_U$ is the composition
\[
	\rho_U' : \sL(U) \xrightarrow{\rho_U} \sT_X(U) \rightarrow \Der_k(\OO_X(U)).
\]
\end{remark}

\begin{defn}
A morphism of Lie algebroids is a morphism of sheaves which is a morphism of $(k,\OO_X(U))$-Lie algebras for any admissible open subset $U$ of $X$.
\end{defn}

In the following, we call a $(k,A)$-Lie algebra $L$ finitely presented if it is finitely presented as an $A$-module, and smooth if $L$ is finitely generated projective as an $A$-module. The next proposition shows that Lie algebroids are globalisations of Lie-Rinehart algebras. This is \cite[Lem.\ 9.2.]{AW1} in case (B), the proof of which also generalises to case (A).

\begin{prop}[{\cite[Lem.\ 9.2.]{AW1}}]
Suppose that $X \in \sB$, and let $A = \OO_X(X)$. Then the global sections functor defines an equivalence of categories from the category of Lie algebroids on $X$ to the category of finitely presented $(k,A)$-Lie algebras. This restricts to an equivalence from smooth Lie algebroids on $X$ to smooth $(k,A)$-Lie algebras.
\end{prop}

\subsubsection{Universal Enveloping Algebra of a Lie Algebroid}

Suppose that $\sL$ is a Lie algebroid on $X$, $V \subset U$ are admissible open subsets of $X$, and let $\varphi : \OO_X(U) \rightarrow \OO_X(V)$, $\tilde{\varphi} : \sL(U) \rightarrow \sL(V)$ be the restriction maps. It is straightforward to verify that $\tilde{\varphi}$ is a $\varphi$-morphism, and therefore we have an induced map
\[
	U(\varphi, \tilde{\varphi}) : U(\sL(U)) \rightarrow U(\sL(V))
\]
by Lemma \ref{UEAfunctoriality}. Because $\sL(U)$ is a presheaf, so is $U(\sL(-)) \colon V \mapsto U(\sL(V))$.

\begin{defn}
We define $\sU(\sL)$ to be the sheafification of the presheaf $U(\sL(-))$.
\end{defn}

There are canonical morphisms
\[
	\iota_{\OO} : \OO_X \rightarrow \sU(\sL), \qquad \iota_{\sL} : \sL \rightarrow \sU(\sL),
\]
of $\OO_X$-modules and sheaves of $k$-Lie algebras respectively. For any admissible open subset $U \subset X$, $a \in \OO_X(U)$, and $x \in \sL(U)$, these satisfy
\[
	\iota_{\sL}(a x) = \iota_{\OO}(a)\iota_{\sL}(x), \qquad [\iota_{\sL}(x), \iota_{\OO}(a)] = \iota_{\OO}(\rho_U'(x)a).
\]
As for Lie-Rinehart algebras, the triple $(\sU(\sL), \iota_{\OO} :\OO_X \rightarrow \sU(\sL), \iota_{\sL}: \sL \rightarrow \sU(\sL))$ enjoys the following universal property.

\begin{lemma}\label{UEASheafUnivProp}
Suppose that $\sS$ is a sheaf of unital associative $k$-algebras on $X$, and
\[
	\eta_{\OO} : \OO_X \rightarrow \sS, \qquad \eta_{\sL} : \sL \rightarrow \sS,
\]
are morphisms of $\OO_X$-modules and sheaves of $k$-Lie algebras respectively which satisfy
\[
	\eta_{\sL}(a x) = \eta_{\OO}(a)\eta_{\sL}(x), \qquad [\eta_{\sL}(x), \eta_{\OO}(a)] = \eta_{\OO}(\rho_U'(x)a).
\]
for any admissible open subset $U \subset X$, $a \in \OO_X(U)$, and $x \in \sL(U)$. Then there is a unique morphism of sheaves of $k$-algebras $\sU(\sL) \rightarrow S$ such that
\[\begin{tikzcd}
	{\OO_X} & {\sU(\sL)} & \sL \\
	& \sS
	\arrow["{\iota_\sL}"', from=1-3, to=1-2]
	\arrow["{\iota_{\OO}}", from=1-1, to=1-2]
	\arrow["\varphi"', from=1-2, to=2-2]
	\arrow["{\eta_\OO}"', from=1-1, to=2-2]
	\arrow["{\eta_\sL}", from=1-3, to=2-2]
\end{tikzcd}\]
commutes.
\end{lemma}

\begin{proof}
	Given such a triple $(\sS, \iota_{\OO} :\OO_X \rightarrow \sS, \iota_{\sL}: \sL \rightarrow \sS)$, we define a morphism of presheaves $\varphi : U(\sL) \rightarrow \sS$, for any admissible open subset $U \subset X$ using Lemma \ref{UEAUnivProp} and setting
	\[
		\varphi_U : U(\sL(U)) \rightarrow \sS(U)
	\]
	to be the unique morphism induced from $\eta_{\OO, U} :\OO_X(U) \rightarrow \sS(U)$, $\eta_{\sL, U}: \sL(U) \rightarrow \sS(U)$. It is direct to check using Lemma \ref{UEAUnivProp} that this is a morphism of presheaves, and we set $\varphi : \sU(\sL) \rightarrow \sS$ to be the unique morphism determined by the universal property of the sheafification. The uniqueness follows directly from the uniqueness of Lemma \ref{UEAUnivProp}.
\end{proof}

As for Lie-Rinehart algebras, we can consider the universal enveloping algebra as a functor in the following manner. In the following, for a Lie algebroid $\sL$ on $X$, we write $\tau : \sL \rightarrow \sT_X$ for the composition of $\rho \colon \sL \rightarrow \sT_X$ with $\sT_X \hookrightarrow \underline{\End}_{k}(\OO_X)$.
\begin{defn}\label{varphimorphismsheaf}
Suppose that $\varphi : X \rightarrow Y$ is a morphism of spaces over $k$, and that $\sL'$, $\sL$ are Lie algebroids over $X$ and $Y$ respectively. Then a $\varphi$-morphism is a morphism of sheaves of $\OO_Y$-modules,
\[
	\tilde{\varphi} : \sL \rightarrow \varphi_* \sL',
\]
which is also a morphism of sheaves of $k$-Lie algebras, such that
\[\begin{tikzcd}
	\sL & {\underline{\End}_{k}(\OO_Y)} & {\underline{\Hom}_k(\OO_Y, \varphi_* \OO_X)} \\
	{\varphi_* \sL'} & {\varphi_* \underline{\End}_k(\OO_X)} & {\underline{\End}_k(\varphi_* \OO_X)}
	\arrow["\tau", from=1-1, to=1-2]
	\arrow["{\tilde{\varphi}}"', from=1-1, to=2-1]
	\arrow["{\varphi_*\tau'}", from=2-1, to=2-2]
	\arrow["{- \circ \varphi^\sharp}", from=1-2, to=1-3]
	\arrow[from=2-2, to=2-3]
	\arrow["{ \varphi^\sharp \circ -}"', from=2-3, to=1-3]
\end{tikzcd}\]
commutes.
\end{defn}

\begin{remark}\label{varphimorphismlocal}
	If $\varphi : X \rightarrow Y$ is a morphism of spaces over $k$, and $\tilde{\varphi} : \sL \rightarrow \varphi_* \sL'$ is a morphism of sheaves of sets, then $\tilde{\varphi}$ is a $\varphi$-morphism if and only if 
	\[
		\tilde{\varphi}_U : \sL(U) \rightarrow \sL'(\varphi^{-1}(U))
	\]
	is a $\varphi^\sharp_U$-morphism for any $U \in \sB_Y$, where $\varphi^\sharp_U : \OO_Y(U) \rightarrow \OO_X(\varphi^{-1}(U))$.
\end{remark}

We write $\sLR_k$ for the category of pairs $(X,\sL)$, where $X$ is a space over $k$ and $\sL$ is a Lie-algebroid on $X$. A morphism $(X, \sL) \rightarrow (Y,\sL')$ is a pair $(\varphi, \tilde{\varphi})$, where $\varphi \colon X \rightarrow Y$ is a morphism of spaces over $k$ and $\tilde{\varphi} \colon \varphi_* \sL \rightarrow \sL'$ is a $\varphi$-morphism. Given $(\varphi, \tilde{\varphi}) : (X,\sL) \rightarrow (Y, \sL')$ and $(\psi, \tilde{\psi}) : (Y, \sL') \rightarrow (Z, \sL'')$, then $\tilde{\psi} \circ \psi_* \tilde{\varphi}$ is a $(\psi \circ \varphi)$-morphism, and the composition is defined as $(\psi \circ \varphi, \tilde{\psi} \circ \psi_* \tilde{\varphi})$.

Note that $\sU(\sL)$ is naturally filtered, coming from the filtration on the presheaf $U(\sL(-))$ induced by the filtration on $U(L)$ described in Section \ref{sect:LRalgebra}.

\begin{lemma}\label{UEAsheaffunctoriality}
	Suppose that $\varphi : X \rightarrow Y$ is a morphism of spaces over $k$, and $\sL', \sL$ are Lie algebroids on $X$ and $Y$ respectively. Then every $\varphi$-morphism $\tilde{\varphi} : \sL \rightarrow \varphi_*\sL'$ extends uniquely to a filtration preserving morphism $\sU(\varphi, \tilde{\varphi}) : \sU(\sL) \rightarrow \varphi_* \sU(\sL')$ of sheaves of $k$-algebras on $Y$ that makes the following diagram commute
\[\begin{tikzcd}
	{\OO_Y \oplus \sL} & {\varphi_* \OO_X \oplus \varphi_*\sL'} \\
	{\sU(\sL)} & {\varphi_* \sU(\sL')}
	\arrow["{\iota_{\OO_Y} \oplus \iota_{\sL}}"', from=1-1, to=2-1]
	\arrow["{\varphi_*\iota_{\OO_X} \oplus \varphi_* \iota_{\sL'}}", from=1-2, to=2-2]
	\arrow["{\varphi^\sharp \oplus \tilde{\varphi}}", from=1-1, to=1-2]
	\arrow["{\sU(\varphi, \tilde{\varphi})}"', from=2-1, to=2-2]
\end{tikzcd}\]
\end{lemma}
\begin{proof}
	We define a morphism of presheaves,
	\[
		U(\varphi, \tilde{\varphi}) : U(\sL) \rightarrow \varphi_* U(\sL'),
	\]
	by setting for each admissible open $U \subset Y$,
	\[
		U(\varphi, \tilde{\varphi})_U \coloneqq U(\varphi^\sharp, \tilde{\varphi}_U) \colon U(\sL(U)) \rightarrow U(\sL'(\varphi^{-1}(U))),
	\]
	which is well-defined by Remark \ref{varphimorphismlocal}. Then we define the morphism $\sU(\varphi, \tilde{\varphi})$ by the universal property of sheafification to be the unique morphism making the diagram 
\[\begin{tikzcd}
	{\sU(\sL)} & {\varphi_* \sU(\sL')} \\
	{U(\sL)} & {\varphi_* U(\sL')}
	\arrow["{\sU(\varphi, \tilde{\varphi})}", from=1-1, to=1-2]
	\arrow["{U(\varphi, \tilde{\varphi})}"', from=2-1, to=2-2]
	\arrow[from=2-1, to=1-1]
	\arrow[from=2-2, to=1-2]
\end{tikzcd}\]
	commute.
\end{proof}

Therefore, we can view $\sU(-)$ as a functor from $\sLR_k$ to the category of pairs $(X, \sS)$, where $X$ is a space over $k$ and $\sS$ is a sheaf of unital associative filtered $k$-algebras on $X$. Morphisms are pairs $(\varphi, \phi)$ for $\varphi : X \rightarrow Y$ a morphism of schemes over $k$, and $\phi : \sS \rightarrow \varphi_*\sS'$ a morphism of sheaves of unital associative filtered $k$-algebras on $Y$.

We have a more explicit description of the sheaf $\sU(\sL)$ when $\sL$ is smooth. In the following, when $U \in \sB$ and $M$ is an $\OO_X(U)$-module, we can consider the presheaf defined by
\[
V \mapsto \OO_X(V) \otimes_{\OO_X(U)} M,
\]
as $V \subset U$ ranges over the admissible open subsets of $U$, and let $\widetilde{M}$ denote its sheafification.

If $V \in \sB$ with $V \subset U$ the natural map
\[
	\OO_X(U) \otimes_{\OO_X(X)} M \rightarrow \widetilde{M}(U)
\]
is an isomorphism, which is a standard fact in Case (A), and in Case (B) follows from \cite[\S 9.4.2]{BGR}.

\begin{lemma}\label{quasicoherent}
Let $(\rho, \sL)$ be a smooth Lie algebroid on $X$. Then $U(\sL(-))$ is a sheaf on the basis $\sB$, so $\sU(\sL) = U(\sL(-))^{\ext}$, and for any $U, V \in \sB$ with $V \subset U$ the canonical $\OO_X(V)$-module homomorphism
\[
	\OO_X(V) \otimes_{\OO_X(U)} U(\sL(U)) \rightarrow U(\sL(V))
\]
is an isomorphism. This also holds for $U \in \sB'$, $V \in \sB$ and $V \subset U$ if $\sup_{x \in U} \dim(\OO_{U,x}) < \infty$.
\end{lemma}

\begin{proof}
Let $\sA$ be the collection of admissible open subsets of $X$ consisting of $\sB$ and those $U \in \sB' \setminus \sB$ for which $\sup_{x \in U} \dim(\OO_{U,x}) < \infty$. For any pair $U \in \sA$ and $V \in \sB$ with $V \subset U$, the natural map
\[
	\OO_X(V) \otimes_{\OO_X(U)} U(\sL(U)) \rightarrow U(\sL(V))
\]
is an isomorphism by Lemma \ref{basechangefactorisation}, as $\sL$ is smooth and coherent as an $\OO_X$-module. The fact that $\sL(U)$ is projective as an $\OO(U)$-module follows when $U \in \sA \setminus \sB$ from \cite[Prop.\ A.2]{KOH}, which is where the condition on the dimension of $U$ arises. For any $U \in \sB$, these maps define an isomorphism
\[
	\widetilde{U(\sL(U))}|_{\sB_U} \xrightarrow{\sim} U(\sL(-))|_{\sB_U},
\]
where $\sB_U$ is the basis of $U$ consisting of $V \in \sB$ such that $V \subset U$. As this holds in particular for any $U \in \sB$, $U(\sL(-))$ defines a sheaf on the basis $\sB$, and by Lemma \ref{extensionsheafification} we see that $U(\sL(-))^{\ext}$ is a sheafification of the presheaf $U(\sL(-))$.
\end{proof}

For example, if $X$ is a scheme, then Lemma \ref{quasicoherent} says that $\sU(\sL)$ is a quasi-coherent sheaf whenever $\sL$ is smooth.

\subsection{Equivariant Sheaves}\label{eqsheaves}

Throughout this section \ref{eqsheaves}, let $G$ be a group and let $X$ be a set with a $G$-topology in the sense of \cite[\S 9.1.1]{BGR}. Suppose also that $G$ acts on $X$, by which we mean there is a group homomorphism $\rho : G \rightarrow \text{Homeo}(X)$, where $\text{Homeo}(X)$ is the group of homeomorphisms from $X$ to itself. For each $g \in G$, $g_*$ and $g^{-1}$ are inverse functors from $\text{Sh}(X)$ to itself, where 
\[
(g^{-1}\sF)(U) = \sF(gU), \: (g_* \sF)(U) = \sF(g^{-1}U),
\]
for all admissible open subsets $U$ of $X$. We summarise some notions from \cite[\S 5.1]{GRO}.

\begin{defn}
A \emph{$k$-linear $G$-equivariant sheaf} on $X$, is a pair $(\sF , \{g^\sF\}_{g \in G})$, such that $\sF$ is a sheaf of $k$-vector spaces on $X$, and for $g \in G$,
\[
g^\sF : \sF \rightarrow g^{-1} \sF,
\]
is an isomorphism of sheaves of $k$-vector spaces, and for all $g , h \in G$,
\[
(gh)^\sF = h^{-1}(g^\sF) \circ h^\sF.
\]
A morphism $\psi : (\sF , \{g^\sF\}_{g \in G}) \rightarrow (\sG , \{g^\sG\}_{g \in G})$ is a morphism of sheaves of $k$-vector spaces, such that
\[
g^{-1}(\psi) \circ g^\sF= g^\sG \circ \psi,
\]
for all $g \in G$.
\end{defn}

\begin{defn}
A \emph{$G$-equivariant sheaf of $k$-algebras} on $X$ is a pair $(\sA , \{g^\sA\}_{g \in G})$, such that 
\begin{itemize}
\item
$(\sA , \{g^\sA\}_{g \in G})$ is a $k$-linear $G$-equivariant sheaf,
\item
$\sA$ is a sheaf of $k$-algebras,
\item
For all $g \in G$,
\[
g^\sA : \sA \rightarrow g^{-1} \sA,
\]
is a morphism of sheaves of $k$-algebras.
\end{itemize}
\end{defn}

\begin{remark}
	If $U$ is a $G$-stable admissible open subset of $X$, then for $a \in \sA(U)$,
	\[
		g \cdot a \coloneqq g^{\sA}(a)
	\]
	defines an action of $G$ on $\sA(U)$ by $k$-algebra automorphisms. Therefore, for such a subset $U$ we can form the skew group ring $\sA(U) \rtimes G$.
\end{remark}

\begin{defn}
Let $\sA$ be a $G$-equivariant sheaf of $k$-algebras on $X$. A \emph{$G$-equivariant sheaf of $\sA$-modules} or $G$-$\sA$-\emph{module} on $X$ is a sheaf $\sM$ of left $\sA$-modules, together with a $k$-linear $G$-equivariant structure $(\sM , \{g^\sM\}_{g \in G})$ such that for any admissible open subset $U \subset X$,
\[
g^\sM_U(a \cdot m) = g^\sA_U(a) \cdot g^\sM_U(m),
\]
for all $g \in G$, $a \in \sA(U)$ and $m \in \sM(U)$. A morphism of $G$-$\sA$-modules is a morphism of sheaves of $\sA$-modules which is also a morphism of $k$-linear $G$-equivariant sheaves. We write $\Mod(G \text{-} \sA)$ for the category of $G\text{-}\sA$-modules on $X$.
\end{defn}

\begin{remark}\label{skewgroupringeg}
	Suppose that $U$ is a $G$-stable admissible open subset of $X$. Then $\Gamma(U, -)$ is a functor from $G$-$\sA$-modules to $\sA(U) \rtimes G$-modules \cite[Prop.\ 2.3.5]{ARD}. For $a \in \sA(U)$ and $g \in G$, $ag \in \sA(U) \rtimes G$ acts on $\sM(U)$ by
	\[
		ag \cdot m \coloneqq a \cdot g^{\sM}(m).
	\]
	for any $m \in \sM(U)$.
\end{remark}

\begin{remark}\label{rem:homspaceequivsheaf}
	Suppose that $\sA$ is a $G$-equivariant sheaf of $k$-algebras, and $\sM$ and $\sN$ are $G\text{-}\sA$-modules. Then it is straightforward to show that $\sF \coloneqq \underline{\Hom}_{\sA}(\sM, \sN)$ is a $G$-equivariant sheaf, where 
	\[
		g^{\sF} \colon \sF \rightarrow g^{-1}\sF
	\]
	is defined by
	\[
		g^{\sF}_U(f) = g_*(g^{\sN}|_{U} \circ f \circ (g^{\sM}|_U)^{-1})
	\]
	for any $g \in G$ and $f \in \sF(U)$. In particular, $\Hom_{\sA}(\sM, \sN)$ is a $k[G]$-module where $g \in G$ acts on $f \in \Hom_{\sA}(\sM, \sN)$ by
	\[
		g \cdot f = g_*(g^{\sN} \circ f \circ (g^{\sM})^{-1}).
	\]
\end{remark}

\begin{eg}\label{eg:conncompequiv}
We will often find ourselves in the following situation. Suppose that $\sA$ is a $G$-equivariant sheaf of $k$-algebras on $X$, and suppose that $X_0$ is an admissible open subset of $X$ such that $X$ is the disjoint union
 \[
 	X = \bigsqcup_{g \in G / G^0} g(X_0),
 \]
 indexed by the set of left cosets $G /G^0$, where $G^0 \coloneqq \Stab_G(X_0) \leq G$. Then it is straightforward to show that the restriction functor
 \[
 	\Mod(G \text{-} \sA) \rightarrow \Mod(G^0 \text{-} \sA|_{X_0}), \]
is an equivalence of categories.
\end{eg}

\begin{eg}\label{eg:quotients}
We will also later need to make use of the following equivalence between equivariant sheaves and sheaves on the quotient space. Suppose that $X$ has an action of a semi-direct product $H \rtimes G$ and $X_0$ is an admissible open subset of $X$ such that $X$ is the disjoint union
	\[
		X = \bigsqcup_{h \in H} h(X_0).
	\]
In this situation, we may form the quotient $G$-topological space $X / H$ as follows. As a set, $X / H$ is the quotient of $X$ by $H$, and writing $p \colon X \rightarrow X / H$ for the quotient map, a subset $U \subset X / H$ is defined to be admissible open if and only if $p^{-1}(U)$ is an admissible open subset of $X$. A collection $\{U_i\}_i$ of admissible open subsets of $X / H$ which covers the admissible open subset $U$ is defined to be an admissible open covering of $U$ if $\{p^{-1}(U_i)\}_i$ forms an admissible open covering of $p^{-1}(U)$ in $X$. It is routine to check that this defines a $G$-topology on $X / H$, for which the quotient map $p \colon X \rightarrow X / H$ is continuous. Furthermore, because $H$ is normal in $H \rtimes G$, the action of $H \rtimes G$ on $X$ induces an action of $H \rtimes G$ on $X/H$ for which the quotient map $p \colon X \rightarrow X / H$ is $H \rtimes G$-equivariant.
	
Suppose now that $\sA$ is a $(H \rtimes G)$-equivariant sheaf of $k$-algebras on $X$. For example, $\sA$ could just be the constant sheaf $\underline{k}$ (in which case a $(H \rtimes G)\text{-}\sA$-module is just a $(H \rtimes G)\text{-}$equivariant $k$-linear sheaf), or $X$ could be a locally ringed $G$-topological space and $\sA = \OO_X$ the structure sheaf.
	
For any $(H \rtimes G)$-equivariant sheaf $\sF$ on $X$, the presheaf $\sF^H$ on $X / H$ defined by
	\[
		\sF^H(U) \coloneqq \sF(p^{-1}(U))^H
	\]
	is a $G$-equivariant sheaf on $X / H$, where the $G$-equivariant structure $g^{\sF^H} \colon \sF^{H} \rightarrow g^{-1} \sF^H$ is defined as the restriction of $p_*g^{\sF} \colon p_* \sF \rightarrow p_* (g^{-1}\sF) = g^{-1} (p_* \sF)$ to $\sF^{H}$. The fact that $p_* (g^{-1} \sF) = g^{-1} (p_* \sF)$ follows from the $(H \rtimes G)$-equivariance of $p \colon X \rightarrow X / H$, and the fact that this restricts to a map $\sF^H \rightarrow g^{-1}\sF^{H}$ is because $H$ is normal in $H \rtimes G$. In particular, this applies to the $H \rtimes G$-equivariant sheaf $\sA$ on $X$, and we obtain the $G$-equivariant sheaf of $k$-algebras $\sA^H$ on $X/H$. Therefore, if $\sF$ is a $(H \rtimes G)\text{-}\sA$-module, then $\sF^H$ naturally acquires the structure of a $G\text{-}\sA^H$-module, and this defines a functor
	\[
		(-)^H \colon \Mod((H \rtimes G)\text{-}\sA) \rightarrow \Mod(G \text{-} \sA^H).
	\]
	In the other direction, we have the pullback functor
	\[
		p^{-1} \colon \Mod(G \text{-} \sA^H) \rightarrow \Mod((H \rtimes G)\text{-}\sA),
	\]
	defined (for simplicity because $p$ is an open map) by
	\[
		p^{-1}\sG(V) \coloneqq \sG(p(V))
	\]
	for $\sG \in \Mod(G \text{-} \sA^H)$. The $H \rtimes G$-equivariant structure on $p^{-1} \sG$ is defined by
	\[
		g^{p^{-1} \sG} \coloneqq p^{-1}g^{\sG} \colon p^{-1} \sG \rightarrow p^{-1} (g^{-1} \sG) = g^{-1} (p^{-1} \sG),
	\]
	for $g \in H \rtimes G$. It is direct to show that, because our assumption that $X$ is the disjoint union of copies of $X_0$ indexed by $H$, that these functors are mutually inverse quasi-equivalences of categories.
	
	We also note that under our assumption on $X$ and the action of $H$, the natural map $X_0 \hookrightarrow X \twoheadrightarrow X / H$ is an isomorphism of $G$-topological spaces, and the pushforward of $\sA|_{X_0}$ to $X / H$ is canonically identified with $\sA^H$. In particular, if $X$ is a rigid space, then taking $\sA = \OO_X$ we see that $X / H$ is also a rigid space, canonically identified with $X_0$.
\end{eg}

\subsection{\texorpdfstring{$\sD$-Modules}{D-Modules}}\label{sect:Dmodules}

In this section we specialise the geometric framework of Section \ref{sect:geometricsetup} in which we work, and additionally assume that \emph{the characteristic of $k$ is zero}, and
\begin{enumerate}
	\item[\textbf{(A)}]
	$X$ is a smooth scheme over $k$ ($X \rightarrow \Spec(k)$ is smooth),
	\item[\textbf{(B)}] 
	$X$ is a smooth rigid space over $k$ ($X \rightarrow \Spec(k)$ is smooth \cite[Def. 2.1]{FRGIII}).
\end{enumerate}

In both cases, this assumption implies that $\Omega_{X/k}$ is locally free of rank $\dim_x X$ at any $x \in X$ (which in case (B) is actually an equivalent condition \cite[Lem.\ 2.8]{FRGIII}). In particular, the dual sheaf $\sT_X$ is locally free (and hence coherent). We also note that in case (A), as $\charfield(k) = 0$, then $X \rightarrow \Spec(k)$ is smooth if and only if $\Omega_{X/k}$ is locally free and $X$ is locally of finite type over $k$ \cite[Lem.\ 04QN]{STACK}. 

\begin{defn}
	We define the sheaf of differential operators $\sD_X \coloneqq \sU(\sT_X)$.
\end{defn}

We write $\Mod(\sD_X)$ for the category of $\sD_X$-modules, and $\VectCon(X)$ for the full subcategory of integrable connections: $\sD_X$-modules for which the underlying $\OO_X$-module is a vector bundle.

\begin{remark}
When $X$ is a smooth rigid space over a field of characteristic zero, $\sD_X$ is by definition the sheaf of differential operators on $X$ \cite{AW1}. When $X$ is a smooth scheme over a field of characteristic zero, the sheaf $\sD_X$ we have defined coincides with the usual sheaf of Grothendieck differential operators in the sense that the natural morphism $\sD_X \rightarrow \underline{\End}_k(\OO_X)$ induced from $\OO_X \rightarrow \underline{\End}_k(\OO_X)$ and $\sT_X \hookrightarrow \underline{\End}_k(\OO_X)$ is injective, and has image the subsheaf of $\underline{\End}_k(\OO_X)$ generated by $\OO_X$ and $\sT_X$. This can be seen by taking the associated graded of this morphism, and using Rinehart's Theorem \cite[Thm.\  3.1]{RINE} and (for example) \cite[\S 1.1]{HTT}.
\end{remark}

We now record some results which will be useful later.

\begin{lemma}\label{connirred}
If $X$ is connected then $\OO_X$ is an irreducible $\sD_X$-module.
\end{lemma}

\begin{proof}
	In case (B), this follows from \cite[Prop.\ 3.1.3]{AW2}. In case (A), we can argue as in the proof of \emph{loc. cit.} to reduce to the case when $X$ is affine, which follows from \cite[Thm.\ 15.3.8]{MCR}.
\end{proof}

\begin{lemma}\label{cohimplieslocallyfree}
Any $\sD_X$-module which is coherent as an $\OO_X$-module is locally free.
\end{lemma}

\begin{proof}
	This follows from exactly the same proof that is given for \cite[Thm.\ 1.4.10]{HTT}.
\end{proof}

\subsubsection{Inverse Image and Direct Image}\label{section:inversedirectimage}

We now describe the inverse and direct image functors of $\sD$-modules for an \'{e}tale morphism $f \colon X \rightarrow Y$. We describe these explicitly, as we will make use of this explicit description later.

\addtocontents{toc}{\SkipTocEntry}
\subsection*{The Direct Image}
Let us first describe the direct image
\[
	f_* \colon \Mod(\sD_X) \rightarrow \Mod(\sD_Y).
\]
Let $\sN \in \Mod(\sD_X)$. As an $\OO_Y$-module, $f_* \sN$ is the $\OO$-module direct image of $\sN$. Because $f \colon X \rightarrow Y$ is \'{e}tale, the maps of Lemma \ref{extendderivations} glue to define a morphism $\sT_Y \rightarrow f_* \sT_X$, and because $\sN$ is a $\sD_X$-module, there is a morphism of sheaves of $k$-Lie algebras,
\[
	\sT_X \rightarrow \sD_X \rightarrow \underline{\End}_k(\sN).
\]
Using these we obtain an action of $\sT_Y$ on $f_* \sN$ by
\begin{align*}
	\sT_Y \rightarrow f_* \sT_X \rightarrow f_*  \underline{\End}_k(\sN) \rightarrow \underline{\End}_k(f_* \sN).
\end{align*}
One can check that this is appropriately compatible with the $\OO_Y$-module structure on $f_* \sN$ and hence both actions extend uniquely by Lemma \ref{UEASheafUnivProp} to a $\sD_Y$-module structure on $f_* \sN$. Concretely, for $U \in \sB_Y$, if $f^{-1}(U) \in \sB_X$ then
\[
(f_*\sN)(U) = \sN(f^{-1}(U)),
\]
which has action of $\sT_Y(U)$ through the morphism $\sT_Y(U) \rightarrow \sT_X(f^{-1}(U))$ of Lemma \ref{extendderivations}.

\addtocontents{toc}{\SkipTocEntry}
\subsection*{The Inverse Image}
Let us now describe the inverse image
\[
	f^* \colon \VectCon(Y) \rightarrow \VectCon(X).
\]
Suppose that $\sM \in \VectCon(Y)$. Then as an $\OO_X$-module
\[
	f^*\sM = \OO_X \otimes_{f^{-1} \OO_Y} f^{-1} \sM,
\]
is the usual $\OO$-module pullback. This obtains an action of $\sT_X$ through the canonical morphism $\sT_X \rightarrow f^* \sT_Y$, which extends to a $\sD_X$-module structure by Lemma \ref{UEASheafUnivProp}. We note that the morphism $\sT_X \rightarrow f^* \sT_Y$ exists whenever $Y$ is smooth, irrespective of the \'{e}taleness of $f \colon X \rightarrow Y$. Concretely, for $U \in \sB_X$ and $V \in \sB_Y$ with $f(U) \subset W$, then just as in the proof of Lemma \ref{extendderivations} there is a natural map
\[
\sT_X(U) \rightarrow \OO_X(U) \otimes_{\OO_Y(V)} \sT(V),
\]
and under the identification
\[
(f^* \sM)(U) = \OO_X(U) \otimes_{\OO_Y(V)} \sM(V)
\]
a local section $\partial \in \sT_X(U)$ acts on $(f^* \sM)(U)$ by
\[
	\partial (s \otimes m) = \partial(s) \otimes m + \sum_i s s_i \otimes \partial_i(m),
\]
where under the above map $\partial \mapsto \sum_i s_i \otimes \partial_i$.

\subsection{Equivariant $\sD$-Modules}\label{sect:generaleqDmodules}

Suppose in this section that $X$ is as in Section \ref{sect:geometricsetup}, and that $G$ is a group that acts on $X$ (on the right), given by the data of a group homomorphism 
\[
\rho : G^{\text{op}} \rightarrow \Aut_k(X), \qquad  g \mapsto (\rho(g) : X \rightarrow X, \rho(g)^\sharp : \OO_X \rightarrow \rho(g)_* \OO_X).
\]
Then we have a group homomorphism
\[
	G \rightarrow \text{Homeo}(X), \qquad g \mapsto \rho(g^{-1}),
\]
and we can consider $G$-equivariant sheaves on $X$.
\begin{eg}
The structure sheaf $\OO_X$ is naturally a $G$-equivariant sheaf of $k$-algebras, with the $G$-equivariant structure
\[
	g^{\OO_X} \coloneqq \rho(g)^\sharp : \OO_X \rightarrow \rho(g)_* \OO_X = g^{-1}\OO_X.
\]
\end{eg}
Suppose now that $(\rho, \sL)$ is a Lie algebroid on $X$, and that $\sL$ also has the structure of a $G\text{-}\OO_X$-module such that each $g^{\sL} \colon \sL \rightarrow g^{-1} \sL$ for $g \in G$ is a $g^{-1}$-morphism in the sense of Definition \ref{varphimorphismsheaf}. In this situation $\sU(\sL)$ is a $G$-equivariant sheaf of $k$-algebras via
\[
	g^{\sU(\sL)} \coloneqq \sU(g^{\OO_X}, g^{\sL}) \colon \sU(\sL) \rightarrow g^{-1} \sU(\sL),
\]
and $\OO_X$ with its natural $\sU(\sL)$-module and $G$-equivariant sheaf structures is a $G \text{-} \sU(\sL)$-module. 

\begin{eg}
If $X$ is additionally as in Section \ref{sect:Dmodules} (i.e.\ $X$ is smooth over a field of characteristic $0$) then the tangent sheaf $\sT_X$ is naturally a $G$-$\OO_X$-module, where we define $g^{\sT_X} : \sT_X \rightarrow g^{-1} \sT_X$ as follows. On each $U \in \sB$, $\sT_X(U) = \Der_k(\OO_X(U))$ (cf.\ Remark \ref{rem:localdescTX}) and we define
\[
g^{\sT_X}_U : \Der_k(\OO_X(U)) \rightarrow \Der_k(\OO_X(g(U)))
\]
by
\[
	g^{\sT_X}_U : \partial \mapsto g^{\OO_X}_U \circ \partial \circ (g^{\OO_X}_U)^{-1}.
\]
Each $g^{\sT_X}$ is a $g^{-1}$-morphism by Remark \ref{varphimorphismlocal} and the description of $g^{\sT_X}$ on $U \in \sB$ above. Therefore taking $\sL = \sT_X$ above, $\sD_X$ is naturally a $G$-equivariant sheaf of $k$-algebras for which $\OO_X$ is a $G\text{-}\sD_X$-module.
\end{eg}

\begin{defn}
For locally noetherian $U \in \sB'$, we call an $\OO_X(U)$-module $M$ \emph{coherent} if as an $\OO_X(U)$-module,
\begin{enumerate}
	\item[\textbf{(A)}] $M$ is finitely generated,
	\item[\textbf{(B)}] $M$ is coadmissible (in the sense of \cite{SchTeit}).
\end{enumerate}
\end{defn}
\begin{remark}
	Note that the locally noetherian hypothesis is only an extra condition in Case (A), where it is equivalent to the condition that $\OO_X(U)$ is a noetherian ring. In Case (B), for $U \in \sB'$, we are viewing $\OO_X(U)$ as a Fr\'{e}chet-Stein algebra (in the sense of \cite{SchTeit}) with respect to the $k$-Banach algebras $(\OO_X(U_i))_i$ for any admissible open quasi-Stein covering $(U_i)_i$ of $X$. We also note that when $U$ is in fact affinoid, $M$ is coadmissible if and only if $M$ is finitely generated, as in Case (A).
\end{remark}
Each $g \in G$ acts on $U(\sL(X))$ via
\[
	U(g^{\OO_X}_X, g^{\sL}_X) \colon U(\sL(X)) \rightarrow U(\sL(X)),
\]
and we can make the following definition.
\begin{defn}
We write $\Mod_{\coh}(G \text{-} \sU(\sL))$ for the full subcategory of $\Mod(G \text{-} \sU(\sL))$ consisting of objects for which the underlying $\OO_X$-module is coherent. 

We write $\Mod_{\coh}(U(\sL(X)) \rtimes G)$ for the full subcategory of $\Mod(U(\sL(X)) \rtimes G)$ consisting of objects for which the underlying $\OO_X(X)$-module is coherent. 
\end{defn}
\begin{prop}\label{prop:LAonQS}
	Suppose that $X \in \sB$ and let $\sL$ be a smooth Lie algebroid on $X$. Then the global sections functor defines an equivalence of categories 
		\[
			\Gamma(X,-) \colon \Mod_{\coh}(G \text{-} \sU(\sL)) \xrightarrow{\sim} \Mod_{\coh}(U(\sL(X)) \rtimes G).
		\]
	The same is true for $X \in \sB'$ if $\sup_{x \in X} \dim(\OO_{X,x}) < \infty$.
\end{prop}

 \begin{proof}
	Suppose that $X \in \sB'$ satisfies $\sup_{x \in X} \dim(\OO_{X,x}) < \infty$, which is automatic if $X \in \sB$. Then the global sections functor induces an equivalence
	\[
		\Gamma(X,-) \colon \Mod_{\coh}(\OO_X) \xrightarrow{\sim} \Mod_{\coh}(\OO_X(X)),
	\]
	with inverse sending an $A$-module $M$ to $\widetilde{M}$. In Case (A) this is \cite[\S 5.1.3, Prop.\ 1.11]{QL}, and in Case (B) this is \cite[Satz 2.4]{KIEHL}, or alternatively follows from \cite[Cor.\ 3.3]{SchTeit}. If $\sM \in \Mod_{\coh}(G \text{-} \sU(\sL))$, then $\sM(X)$ is naturally a $U(\sL(X)) \rtimes G$-module via the natural map
	\[
		\sU(\sL(X)) \rightarrow U(\sL)(X)
	\]
	and the action of $g \in G$ by $g^{\sM}_X \colon \sM(X) \rightarrow \sM(g(X)) = \sM(X)$. Furthermore, for any $G\text{-}\sU(\sL)$-morphism $f \colon \sM \rightarrow \sN$ the map on global sections will be $U(\sL(X)) \rtimes G$-linear. 
	
	Suppose now that $M \in \Mod_{\coh}(U(\sL(X)) \rtimes G)$. For any $U \in \sB$, then using Lemma \ref{quasicoherent}, which is where the hypothesis on the dimension of $X$ arises, we can factorise the natural map
	\begin{align*}
		\widetilde{M}(U) &= \OO_X(U) \otimes_{\OO_X(X)} \sM(X), \\
		&\xrightarrow{\sim} \OO_X(U) \otimes_{\OO_X(X)} U(\sL(X)) \otimes_{U(\sL(X))} \sM(X), \\
		 &\xrightarrow{\sim} U(\sL(U)) \otimes_{U(\sL(X))} \sM(X)
	\end{align*}
	as the composition of isomorphisms. Therefore, letting $U(\sL(U))$ act via left multiplication on the left factor $U(\sL(U))$, the sheaf $\widetilde{M}$ naturally has the structure of a sheaf of $\sU(\sL)$-modules. This is further a $G \text{-} \sU(\sL)$-module with respect to the $G$-equivariant structure
	\[
		g^{\widetilde{M}} \colon \widetilde{M} \rightarrow g^{-1}\widetilde{M}
	\]
	defined on $U \in \sB$ by
	\[
		g^{\widetilde{M}}_U = g^{\OO_X}_U \otimes g \colon \OO_X(U) \otimes_{A} \sM(X) \rightarrow \OO_X(g(U)) \otimes_{A} \sM(X).
	\]
	Any morphism $\lambda \colon M \rightarrow N$ naturally induces a morphism $1 \otimes \lambda \colon \widetilde{M} \rightarrow \widetilde{N}$, and this is easily shown to be a morphism of $G \text{-} \sU(\sL)$-modules. It is straightforward to check that with these definitions, the natural isomorphisms
	\[
		\widetilde{\sM(X)} \xlongrightarrow{\sim} \sM, \qquad M \xlongrightarrow{\sim} \widetilde{M}(X),
	\]
	are isomorphisms of $G \text{-} \sU(\sL)$-modules and $U(\sL(X)) \rtimes G$-modules respectively.
 \end{proof}

 \begin{remark}
In case (A) (so the condition that $X \in \sB$ means that $X$ is an affine scheme), the same proof shows that taking the global sections also defines an equivalence
\[
	\Gamma(X,-) \colon \Mod_{\text{qc}}(G \text{-} \sU(\sL)) \xrightarrow{\sim} \Mod(\sU(\sL(X)) \rtimes G).
\]
where $\Mod_{\text{qc}}(G \text{-} \sU(\sL))$ is the full subcategory of $\Mod(G \text{-} \sU(\sL))$ consisting of objects for which the underlying $\OO_X$-module is quasi-coherent.
 \end{remark}

 Suppose now that $X$ is as in Section \ref{sect:Dmodules} ($X$ is smooth over a field of characteristic $0$). In this case, we write $\VectCon^G(X)$ for the category $\Mod_{\coh}(G\text{-}\sD_X)$ which in light of Lemma \ref{cohimplieslocallyfree} is the category of $G$-equivariant vector bundles with connection on $X$. Because coherent sheaves are closed under kernels and cokernels, $\VectCon^G(X)$ is abelian and furthermore $\VectCon^G(X)$ is a rigid abelian tensor category in the sense of \cite{DELMIL}. The tensor product $\sV \otimes \sW$ of $\sV, \sW \in \VectCon^G(X)$ is defined to be the tensor product of $\OO_X$-modules, with $\sD_X$-module structure
\[
	\partial \cdot (x \otimes y) = x \otimes \partial(y) + \partial(x) \otimes y,
\]
for a local section $\partial$ of $\sT_X$, and $G$-equivariant structure
\begin{align*}
	g^{\sV \otimes \sW} \coloneqq g^{\sV} \otimes g^{\sW} \colon \sV \otimes \sW \rightarrow g^{-1}(\sV \otimes \sW).
\end{align*}
With this tensor structure, then $\underline{\Hom}(\sV, \sW)$ is given by the internal hom of $\OO_X$-modules with $\sD_X$-module structure 
\[
	(\partial \cdot f)(x) = \partial \cdot f(x) - f(\partial \cdot x)
\]
for a local section $\partial$ of $\sT_X$, and $G$-equivariant structure as described in Remark \ref{rem:homspaceequivsheaf}.

\subsection{Galois Extensions}\label{sect:galoisextensions}

Let $G$ be a group, and let $X$ and $Y$ be as in Section \ref{sect:geometricsetup}. For a morphism $f \colon X \rightarrow Y$ of spaces over $k$ we denote by $\Aut_Y(X)$ the group of automorphism of $X$ over $Y$. Suppose that we have a (right) action of $G$ on $X$ over $Y$, by which we mean a group homomorphism $G^{\op} \rightarrow \Aut_Y(X)$. In this situation, using the notation of Section \ref{sect:generaleqDmodules}, the sheaf of $\OO_Y$-modules $f_* \OO_X$ has a (left) action of $G$,
\[
	G \rightarrow \Aut_k(f_*\OO_X), \qquad g \mapsto (g^{\OO_X}_{f^{-1}(U)} : f_*\OO_X(U) \rightarrow f_*\OO_X(U))_{U \subset Y},
\]
which is well-defined as $g^{\OO_X} \colon \OO_X \rightarrow g^{-1} \OO_X$, and $g(f^{-1}(U)) = f^{-1}(U)$ for any admissible open subset $U \subset Y$. Therefore we can consider the sheaf of $\OO_Y$-modules $(f_* \OO_X)^{G}$ defined by
\[
	(f_* \OO_X)^{G}(U) = \OO_X(f^{-1}(U))^{G},
\]
for any admissible open subset $U$ of $Y$, which is a sheaf because $(-)^G$ preserves products and equalisers. 
\begin{defn}\label{defn:galois}
Suppose that $G$ is a finite group, $f \colon X \rightarrow Y$ is a finite \'{e}tale morphism, and $G$ acts on $X$ over $Y$. Then $f \colon X \rightarrow Y$ is a \emph{Galois covering with Galois group $G$} if the natural map $\OO_Y \rightarrow (f_* \OO_X)^{G}$ is an isomorphism of $\OO_Y$-modules and
\[
	p_X \times a \colon X \times \underline{G} \rightarrow X \times_Y X
\]
is an isomorphism of spaces over $k$, where $a \colon X \times \underline{G} \rightarrow X$ denotes the action map.
\end{defn}

We will also make use of the notion of a Galois extension of commutative $k$-algebras, which is the affine version of Definition \ref{defn:galois}.

\begin{defn}[{\cite[Def. 1.4]{CHR}}]\label{defn:galoisring}
	Suppose that $\varphi : A \hookrightarrow B$ is an injective homomorphism of commutative $k$-algebras, and $G$ is a finite subgroup of $\Aut_{\Alg_k}(B)$. Then $\varphi : A \hookrightarrow B$ is called a Galois extension with Galois group $G$ if $A = B^G$, and 
\begin{align*}
	\beta : B \otimes_A B \rightarrow B \otimes_k \OO(G), \\
	\beta(x \otimes y) = \sum_{g \in G} x (g\cdot y) \otimes \delta_g,
\end{align*}
is an isomorphism, where $\OO(G) = \OO(\underline{G})$ is the $k$-algebra of functions from $G$ to $k$.
\end{defn}

\begin{eg}\label{eg:galextrings}
Suppose that $X \rightarrow Y$ is a morphism of affine (resp. affinoid) spaces over $k$, given on global sections by a morphism $\varphi \colon A \rightarrow B$. Then an action of $G$ on $X$ over $Y$ is equivalent to a homomorphism $\rho \colon G \rightarrow \Aut_{\Alg_k}(B)$ with $\varphi(A) \subset B^G$.

If $f \colon X \rightarrow Y$ is a finite \'{e}tale Galois covering with Galois group $G$, $\rho$ is injective (from the surjectivity of $\beta$), and $\varphi \colon A \rightarrow B$ is injective and a Galois extension with Galois group $G$.

Conversely, if $\varphi \colon A \rightarrow B$ is injective and a Galois extension with Galois group $G$, then $f \colon X \rightarrow Y$ is a finite \'{e}tale Galois covering with Galois group $G$ in the sense of Definition \ref{defn:galois}, as $\varphi \colon A \rightarrow B$ will automatically be finite \'{e}tale by \cite[Thm.\ 1.3]{CHR}.
\end{eg}

\begin{remark}
We note in Definition \ref{defn:galois} we obtain an equivalent definition if we replace the condition that $\OO_Y \rightarrow (f_* \OO_X)^{G}$ is an isomorphism by the condition that $f \colon X \rightarrow Y$ is faithfully flat, or equivalently that $f \colon X \rightarrow Y$ is surjective. Indeed, if $\OO_Y \rightarrow (f_* \OO_X)^{G}$ is an isomorphism, then for any $U \in \sB_Y$, the extension $\OO_Y(U) \rightarrow \OO_X(f^{-1}(U))$ is Galois with Galois group $G$, and thus faithfully flat \cite[Lem.\ 1.9]{GRE}. Conversely, if $f \colon X \rightarrow Y$ is faithfully flat, then for any $U \in \sB_Y$ $A \coloneqq \OO_Y(U) \hookrightarrow B \coloneqq \OO_X(f^{-1}(U))$. The map $\beta$ of Definition \ref{defn:galoisring} is an isomorphism, and therefore because $\beta(b \otimes 1) = \beta(1 \otimes b)$, $b \otimes 1 = 1 \otimes b$ in $B \otimes_A B$. But then because the composition $A \rightarrow B^G \rightarrow B$ is faithfully flat, the fact that $A = B^G$ follows from the following elementary lemma.
\end{remark}

\begin{lemma}
Suppose that $R \subset S \subset T$ are commutative rings with $R \hookrightarrow T$ faithfully flat and $s \otimes 1 = 1 \otimes s$ in $T \otimes_R T$ for any $s \in S$. Then $R = S$.
\end{lemma}

\begin{proof}
Because $R \hookrightarrow T$ is faithfully flat it is sufficient to show that $ T \otimes_R S / R = 0$. By the exact sequence,
\[
T \otimes_R R \rightarrow T \otimes_R S \rightarrow T \otimes_R S / R \rightarrow 0,
\]
it is sufficient to show that $T \otimes_R R \rightarrow T \otimes_R S$ is surjective. We can view $T \otimes_R S \subset T \otimes_R T$, because $T$ is flat over $R$, and therefore for any pure tensor $t \otimes s \in T \otimes _R S$,
\[
ts \otimes 1 \mapsto ts \otimes 1 =( t \otimes 1)(s \otimes 1) = ( t \otimes 1)(1 \otimes s) = t \otimes s,
\]
and so the map is surjective.
\end{proof}

We note here the following lemma for later use, which says that derivations lifted along a Galois extension commute with the Galois action.

\begin{lemma}\label{dercommuteG}
	Suppose that $G$ is a finite group and $\varphi \colon A \hookrightarrow B$ is a Galois extension of commutative $k$-algebras with Galois group $G$. Then any $\partial \in \Der_k(A)$ and $g \in G$,
	\[
		g \circ \psi(\partial) = \psi(\partial) \circ g,
	\] 
	where $\psi : \Der_k(A) \rightarrow \Der_k(B)$ is the $A$-linear homomorphism of Lemma \ref{extendderivations}.
\end{lemma}
	
\begin{proof}
	For any $g \in G$, both $\psi$ and,
	\[
	\psi_g : \partial \mapsto g \circ \psi(\partial) \circ g^{-1},
	\]
	are $A$-linear maps $\Der_k(A) \rightarrow \Der_k(B)$. Furthermore,
	\[
		\psi_g(\partial) \circ \varphi = \varphi \circ \partial,
	\]
	as for any $a \in A$, the left-hand side is,
	\begin{align*}
		(\psi_g(\partial) \circ \varphi)(a) &= g(\psi(\partial)(g^{-1}(\varphi(a)))),\\
		&= g(\psi(\partial)(\varphi(a))),\\
		&= g(\varphi(\partial(a))),\\
		&= \varphi(\partial(a)),
	\end{align*}
	which is exactly the right-hand side. Therefore, by the uniqueness of $\psi$, $\psi = \psi_g$.
\end{proof}

We will need the following basic fact concerning connected components of Galois coverings.

\begin{lemma}\label{lemma:conncompgaloiscovering}
Suppose that $f \colon X \rightarrow Y$ is a finite \'{e}tale Galois covering with Galois group $G$. Then for any admissible open subset $Y_0 \subset Y$,
\[
f|_{f^{-1}(Y_0)} \colon f^{-1}(Y_0) \rightarrow Y_0
\]
is a finite \'{e}tale Galois covering with Galois group $G$.

Suppose additionally that $Y_0$ is connected. Then $G$ acts on the set of connected components of $f^{-1}(Y_0)$ transitively, and for any connected component $X_0$ of $f^{-1}(Y_0)$, 
\[
f_0 \coloneqq f|_{X_0} \colon X_0 \rightarrow Y_0
\]
is a finite \'{e}tale Galois covering with Galois group $H \coloneqq \Stab_G(X_0)$.
\end{lemma}

\begin{proof}
	This is based on \cite[Prop.\ 3.8]{GRE}. It is direct to see that $f \colon f^{-1}(Y_0) \rightarrow Y_0$ is finite \'{e}tale Galois with Galois group $G$. For the second claim, first note that the morphism $f : X \rightarrow Y$ is finite locally free of constant rank $|G|$. Indeed, for any $U \in \sB_Y$ and $V \coloneqq f^{-1}(U)$, $\OO_X(V)$ is finitely generated projective over $\OO_Y(U)$ \cite[Thm.\ 1.3 (c)]{CHR} and the isomorphism $\beta$ for the Galois extension $\OO_Y(U) \rightarrow \OO_X(V)$ shows that $\rank_{\OO_Y(U)}(\OO_X(V)) = |G|$. Therefore, because $Y_0$ is connected, $f^{-1}(Y_0)$ has at most $|G|$ connected components, and we let $S$ be the corresponding set of primitive orthogonal idempotents of $\OO_X(f^{-1}(Y_0))$. $G$ acts on $S$, and this action is transitive: if $\{ e^1, ... , e^m\}$ is an orbit of $S$, then the sum $e^1 + \cdots + e^m$ is $G$-invariant, hence is an non-zero idempotent of $\OO_Y(Y_0)$. Therefore, $e^1 + \cdots + e^m = 1$, and the orbit is the whole of $S$.

 	If $e$ is the idempotent of $\OO_X(f^{-1}(Y_0))$ corresponding to the connected component $X_0$, then $S = \{g(e) \mid g \in G\}$ and the stabiliser $H$ of $e$ in $G$ acts on $X_0$. Furthermore, as $G$ acts on $S$ transitively with stabiliser $H$, $|S| = |G| / |H|$. The group $G$ provides isomorphisms between all connected components of $f^{-1}(Y_0)$ and these isomorphisms respect the morphism to $Y$. Therefore, writing $\deg(f \colon X \rightarrow Y) \coloneqq \rank_Y(f_* \OO_X) \in \bZ_{\geq 1}$ and $\deg(f_0 : X_0 \rightarrow Y_0) \coloneqq \rank_{Y_0}(f_* \OO_{X_0}) \in \bZ_{\geq 1}$,
 	\[
 		|G| = \deg(f : X \rightarrow Y) = (|G| / |H|) \cdot \deg(f_0 : X_0 \rightarrow Y_0),
 	\]
 	and so $|H| = \deg(f_0 : X_0 \rightarrow Y_0)$. We will make use of this fact below.

	To show that $f_0 : X_0 \rightarrow Y_0$ is Galois with Galois group $H$, it is sufficient to show that for any $U \in \sB_{Y_0}$, if $V_0 \coloneqq f_0^{-1}(U) \subset X_0$, that $f|_{V_0} \colon V_0 \rightarrow U$ is Galois with Galois group $H$.
	
	Set $A \coloneqq \OO_Y(U), B_0 \coloneqq \OO_X(V_0)$, and $B \coloneqq \OO_X(f^{-1}(U))$. By Remark \ref{eg:galextrings}, we know that $A \hookrightarrow B$ is a $G$-Galois extension of $k$-algebras, and we are reduced to showing that $A \hookrightarrow B_0$ is a $H$-Galois extension of $k$-algebras. From the above we have that the natural map
	\[
		B \rightarrow \prod_{g \in G/H} g(e) B, \qquad b \mapsto (g(e)b)_{g \in G/H}
	\]
	is an isomorphism and $B_0 = e B$. We first show that $B_0^{H} = A$. For $b \in B_0^{H}$, consider 
 	\[
 		s \coloneqq \sum_{g \in G/H} g(b) \in B^{G} = A.
 	\]
	Because the action of $G$ on $S$ is transitive with stabiliser $H$, $e g(b) = 0$ whenever $g \not\in H$, thus
\[
es =  \sum_{g \in G/H} eg(b) = e b = b,
\]	
where $eb = b$ because $b \in B_0$. Therefore $b \in A$, as $es \in A$, which follows from the fact that $s \in A$ and $ea = a$ for any $a \in A \subset B_0$.

 	The $A$-module $B_0$ is finitely generated projective, being a direct summand of $B$, and thus to show that $A \hookrightarrow B_0$ is a $H$-Galois extension of $k$-algebras, by \cite[Thm.\ 1.3]{CHR} it is sufficient to show that the natural map
 	\[
 		j_0 : B_0 \rtimes H \rightarrow \End_A(B_0)
 	\]
 	from the skew group ring $B_0 \rtimes H$ to the ring of $A$-module endomorphisms $\End_A(B_0)$ is an isomorphism. In fact, because these are both finitely generated projective of rank $|H|^2$ over $A$ (because $|H| = \deg(f_0 : X_0 \rightarrow Y_0)$), it is sufficient for us to show this is surjective.

	Given any $\psi \in \End_A(B_0)$, then $\psi e \in \End_A(B)$ (where we are identify any $b \in B$ with the left multiplication map by $b$). Therefore, because
 	\[
 		j : B \rtimes G \rightarrow \End_A(B)
 	\]
	 is surjective, there are elements $b_{g} \in B$ such that $\sum_{g \in G} b_{g} g = \psi e$. If we multiply by $e$ on both sides we obtain $\sum_{g \in H} e b_{g} g e = e \psi e$ because $e g e = 0$ if $g \not\in H$. When we restrict these functions to $B_0$, as $e$ acts as the identity $\sum_{g \in H} e b_{g} g = \psi$ and $j_0$ is surjective as required.
 \end{proof}

\subsection{Equivalence of Categories}\label{sect:equivofcats}

In this section we specialise the geometric framework of Section \ref{sect:geometricsetup} and Section \ref{sect:Dmodules} further and suppose that $\charfield (k) = 0$,
\begin{enumerate}
	\item[\textbf{(A)}]
	$f \colon X \rightarrow Y$ is a Galois covering of smooth schemes over $k$ with Galois group $H$,
	\item[\textbf{(B)}] 
	$f \colon X \rightarrow Y$ is a Galois covering of smooth rigid spaces over $k$ with Galois group $H$.
\end{enumerate}

We suppose further that $G$ is a group which contains $H$ as a normal subgroup, and that the actions of $H$ on $X$ and $Y$ extend to actions of $G$ for which the morphism $f \colon X \rightarrow Y$ is $G$-equivariant. For example, in this generality $G$ could be infinite or simply $H$.

In this section, we show there is a canonically defined equivalence between $\VectCon^{G/H}(Y)$ and $\VectCon^{G}(X)$ (Proposition \ref{FunctorEquivalence}). This equivalence is completely formal, and taking $G = H$ and forgetting the action of $\sD$ restricts to the well-known equivalence between $\Coh(Y)$ and $\Coh^H(X)$. We briefly describe both functors of the equivalence explicitly, as we shall later make use of these descriptions.

\addtocontents{toc}{\SkipTocEntry}
\subsection*{Inverse Image Functor}
First, we show that the inverse image functor $f^* \colon \VectCon(Y) \rightarrow \VectCon(X)$ described in Section \ref{section:inversedirectimage} extends to a functor
\[
	f^* \colon \VectCon^{G/H}(Y) \rightarrow \VectCon^{G}(X).
\]
Suppose that $\sM$ is a $(G/H) \text{-} \sD_Y$-module, coherent as an $\OO_Y$-module. Then $f^{*}\sM$ is naturally a $G$-equivariant sheaf via 
\[
	g^{f^* \sM} \coloneqq g^{\OO_X} \otimes f^{-1}g^{\sM} \colon f^* \sM \rightarrow g^{-1} (f^* \sM),
\]
using that $g^{-1} (f^{-1} \sM) = f^{-1} (g^{-1} \sM)$ as $f \colon X \rightarrow Y$ is $G$-equivariant. It is direct to check that this makes $f^* \sM$ a $G$-$\sD_X$-module.

\addtocontents{toc}{\SkipTocEntry}
\subsection*{Invariants Functor} Next, we use the direct image functor $f_* \colon \Mod(\sD_X) \rightarrow \Mod(\sD_Y)$ of Section \ref{section:inversedirectimage} to define a functor
\[
	(-)^H \colon \VectCon^{G}(X) \rightarrow \VectCon^{G/H}(Y).
\]
as follows. Suppose that $\sN \in \VectCon^{G}(X)$. Because $f \colon X \rightarrow Y$ is equivariant with respect to the trivial action of $H$ on $Y$,
\[
	H \rightarrow \Aut_k(f_* \sN), \qquad h \mapsto f_* h^{\sN},
\]
is a well-defined group homomorphism. We set
\[
	\sN^H(U) \coloneqq \sN(f^{-1}(U))^H,
\]
which is a sheaf because taking $H$-invariants commutes with products and equalisers.

The sheaf $f_* \sN$ is naturally a $\OO_Y$-module via $\OO_Y \rightarrow f_* \OO_X$. In fact, the action of $\OO_Y(U)$ on $\sN(f^{-1}(U))$ preserves $\sN(f^{-1}(U))^H$ because $\OO_Y \rightarrow f_* \OO_X$ has image $\OO_X^H$ and $\sN$ is a $H$-$\OO_X$-module. From the $\sD_Y$-module structure on $f_* \sN$ we have
\[
	\sT_Y \rightarrow \underline{\End}_k(f_* \sN),
\]
and in fact this action preserves the subsheaf $\sN^H$ and induces:
\[
	\sT_Y \rightarrow \underline{\End}_k(\sN^H).
\]
Indeed, for any $U, V \in \sB_Y$ with $V \subset U$ and $\partial \in \sT_Y(U)$,
\[
	\partial_V|_{\sN^H(V)}: \sN^H(V) \rightarrow \sN(V)
\]
factors as
\[
	\partial_V|_{\sN^H(V)}: \sN^H(V) \rightarrow \sN^H(V) \rightarrow \sN(V),
\]
by Lemma \ref{dercommuteG}. Using Lemma \ref{UEASheafUnivProp} this extends to give $\sN^H$ the structure of a $\sD_Y$-module. For $g \in G/H$, because $H$ is normal in $G$, we have a well defined morphism of sheaves
\[
	g^{\sN^H} \coloneqq f_* g^{\sN} \colon \sN^H \rightarrow g^{-1} \sN^H,
\]
which gives $\sN^H$ the structure of a $G$-$\sD_Y$-module. Finally, note that $\sN^H$ is coherent, being the kernel of the morphism of coherent sheaves
\[
\bigoplus_{h \in H} (f_*h^{\sN} - \id) \colon f_* \sN \rightarrow \bigoplus_{h \in H} f_* \sN.
\]

\begin{prop}\label{FunctorEquivalence}
The functors,
\begin{align*}
	f^* &\colon \VectCon^G(Y) \rightarrow \VectCon^{G/H}(X), \\
	(-)^H &\colon \VectCon^{G/H}(X) \rightarrow \VectCon^G(Y),
\end{align*}
are quasi-inverse equivalences of monoidal categories.
\end{prop}

\begin{proof}
	Let $\sM \in \VectCon^{G/H}(Y)$ and $\sN \in \VectCon^{G}(X)$. We define natural transformations
	\[
		\phi \colon f^* \sN^H \rightarrow \sN, \qquad \psi \colon \sM \rightarrow (f^*\sM)^H,
	\]
	and show that they are isomorphisms. We first consider $\phi$. Let $V \in \sB_Y$, $U \coloneqq f^{-1}(V) \in \sB_X$, and define $\phi|_U \colon f^* \sN^H|_U\rightarrow \sN|_U$ on global sections to be the natural multiplication map
	\[
		\OO_X(U) \otimes_{\OO_Y(V)} \sN(U)^H \rightarrow \sN(U).
	\]
	It is straightforward to check that this is $\sD_X(U)$-linear using our local description of the $\sD$-module structure on $f^* \sN$ of Section \ref{section:inversedirectimage}, and an isomorphism because $\OO_Y(V) \rightarrow \OO_X(U)$ is a Galois extension of $k$-algebras \cite[Thm.\ 1.3(d)]{CHR}. For different choices of $V$, the morphisms $\phi|_U$ agree on their intersection, and thus glue to define a morphism of sheaves $\phi \colon f^* \sN^H \rightarrow \sN$. Because $\{f^{-1}(V)\}_{V \in \sB_Y}$ is an admissible open cover of $X$, $\phi$ is $\sD_X$-linear. Similarly, to show that $\phi$ is $G$-linear, because $\phi$ is a morphism of coherent sheaves it is sufficient to show that for any $g \in G$ and $V \in \sB_Y$, $U = f^{-1}(V)$ as above, that
\[\begin{tikzcd}[ampersand replacement=\&]
	{\OO_X(U) \otimes_{\OO_Y(V)} \sN(U)^H} \& {\sN(U)} \\
	{\OO_X(g(U)) \otimes_{\OO_Y(g(V))} \sN(g(U))^H} \& {\sN(g(U))}
	\arrow["{\phi_U}", from=1-1, to=1-2]
	\arrow["{g^{f^* \sN^H}_U}"', from=1-1, to=2-1]
	\arrow["{g^{\sN}_U}", from=1-2, to=2-2]
	\arrow["{\phi_{g(U)}}"', from=2-1, to=2-2]
\end{tikzcd}\]
commutes, which follows directly from the definition of $g^{f^* \sN^H}$.
	
	Now let us consider $\psi \colon \sM \rightarrow (f^*\sM)^H$. For any $V \in \sB_Y$, setting $U \coloneqq f^{-1}(V)$ we define $\psi|_V \colon \sM|_V \rightarrow (f^*\sM)^H|_V$ on global sections to be the natural inclusion
	\[
		\sM(V) \rightarrow (\OO_X(U) \otimes_{\OO_Y(V)} \sM(V))^H,
	\]
	which is $\sD_Y(V)$-linear (again using our description of the $\sD$-module structure on $f^* \sM$ of Section \ref{section:inversedirectimage}), and glues to a well-defined morphism $\psi \colon \sM \rightarrow (f^*\sM)^H$. Each restriction $\psi_V$ is an isomorphism, because the $H$-action on $M$ is trivial and $\OO_X(U)$ is projective over $\OO_Y(V)$, and therefore $\psi$ is an isomorphism. To show that $\psi$ is $G$-equivariant, it is sufficient to show that
\[\begin{tikzcd}[ampersand replacement=\&]
	{\sM(V)} \& {(\OO_X(U) \otimes_{\OO_Y(V)} \sM(V))^H} \\
	{\sM(g(V))} \& {(\OO_X(g(U)) \otimes_{\OO_Y(g(V))} \sM(g(V)))^H}
	\arrow[from=1-1, to=1-2]
	\arrow["{g^{\sM}}", from=1-1, to=2-1]
	\arrow["{g^{(f^* \sM)^H}}", from=1-2, to=2-2]
	\arrow[from=2-1, to=2-2]
\end{tikzcd}\]
which similarly follows directly from the definition of the morphism $g^{(f^* \sM)^H}$. Finally, given $\sV, \sW \in \VectCon^G(Y)$, the canonical $\OO_X$-linear isomorphism
\[
	f^* \sV \otimes f^* \sW \xrightarrow{\sim} f^*(\sV \otimes \sW)
\]
can be checked to be $G\text{-}\sD_X$-linear, using the explicit description of the $\sD_X$-module structure on the inverse image, and therefore the equivalences are equivalences of monoidal categories.
\end{proof}

\begin{remark}
If $X$ is a scheme, then the same proof shows that the same equivalences hold when one more generally considers equivariant $\sD$-modules for which the underlying $\OO$-module is quasi-coherent.
\end{remark}

\section{The Sheaf of Constant Functions}\label{sect:constantsheaf}

In this section we work in the same geometric framework of Section \ref{sect:Dmodules} and assume that $k$ has characteristic zero and
\begin{enumerate}
	\item[\textbf{(A)}]
	$X$ is a smooth scheme over $k$ ($X \rightarrow \Spec(k)$ is smooth),
	\item[\textbf{(B)}] 
	$X$ is a smooth rigid space over $k$ ($X \rightarrow \Spec(k)$ is smooth \cite[Def. 2.1]{FRGIII}).
\end{enumerate}

In this section we are interested in the following sheaf, known as the sheaf of constant functions, and how it relates to geometric connectivity.

\begin{defn}
Let $c_X \coloneqq \ker(d \colon \OO_X \rightarrow \Omega_{X/k})$.
\end{defn}

This sheaf of $k$-algebras has been considered by Berkovich in the setting of Berkovich spaces \cite{BER04,BER07}. We can give a more explicit description of $c_X$ on the basis $\sB$.

\begin{lemma}\label{affinedesccX}
For any $U \in \sB$,
\[
	c_X(U) = \OO_X(U)^{\sT_X(U) = 0} = \{f \in \OO_X(U) \mid \partial(f) = 0 \text{ for all } \partial \in \sT_X(U)\}.
\]
In particular, $c_X$ is the sheaf extension of the sheaf $U \mapsto \OO(U)^{\sT_X(U) = 0}$ on the basis $\sB$.
\end{lemma}
\begin{proof}
For $U \in \sB$ and $A \coloneqq \OO_X(U)$, we want to show that
\[
	\ker(d \colon A \rightarrow \Omega_{A/k}) = A^{\Der_k(A)},
\]
where $\Omega_{A/k}$ has the same meaning in cases (A) and (B) as it does in Section \ref{sect:tangentsheaf}. Composition with $d$ induces an isomorphism,
\[
	\Hom_A(\Omega_{A/k},A) \xrightarrow{\sim} \Der_k(A),
\]
and so if $a \in \ker(d)$, then $a \in A^{\Der_k(A)}$. On the other hand, if $\partial(a) = 0$ for all $\partial \in \Der_k(A)$, then for any $A$-linear $f \colon \Omega_{A/k} \rightarrow A$, $f(d(a)) = 0$. Therefore $d(a) = 0$, which can be seen by picking a dual basis for the module $\Omega_{A/k}$, which is projective because $X$ is smooth.
\end{proof}

We also have a third description of $c_X$, which will be the most relevant for us.

\begin{lemma}\label{altdesccX}
The isomorphism of sheaves of $k$-algebras,
\[
	\OO_X \xrightarrow{\sim} \underline{\End}_{\OO_X}(\OO_X)
\]
restricts to an isomorphism
\[
	c_X \xrightarrow{\sim} \underline{\End}_{\sD_X}(\OO_X).
\]
\end{lemma}

\begin{proof}
	It is sufficient to show that for any $U \in \sB$, the isomorphism $\phi$,
	\[
		\OO_X(U) \xrightarrow{\phi} \End_{\OO_X|_U}(\OO_X|_U) \xrightarrow{\sim} \End_{\OO_X(U)}(\OO_X(U))
	\]
	restricts to an isomorphism,
	\[
		c_X(U) \xrightarrow{\phi} \End_{\sD_X|_U}(\OO_X|_U) \xrightarrow{\sim} \End_{\sD_X(X)}(\OO_X(U)).
	\]
	Here we are using the equivalence of Proposition \ref{prop:LAonQS} in the case that $G$ is trivial. Given $f \in \OO_X(U)$, the corresponding $\OO_X(U)$-linear endomorphism $\phi_f $ of $\OO_X(U)$ is defined by $\phi_f(x) = fx$. For any $x \in \OO_X(U)$ and $\partial \in \sT_X(U)$,
	\begin{align*}
		\partial(\phi_f(x)) &= \partial(fx), \\
		&= f \partial(x) + x \partial(f),\\
		&= \phi_f(\partial(x)) + x \partial(f).
	\end{align*} 
	Therefore, $\phi_f$ is $\sD_X(U)$-linear if and only if $\partial(f) = 0$ for all $\partial \in \sT_X(U)$.
\end{proof}

Before describing the sheaf $c_X$ in more detail, we have the following lemma regarding admissible open coverings of $X$.

\begin{lemma}\label{lem:opencoveringlemma}
	Suppose that $X$ is as described at the start of this section and non-empty. Then:
	\begin{enumerate}
		\item $X$ is a disjoint union of its connected components,
		\item $X$ has an admissible open covering by connected elements of $\sB$.
	\end{enumerate}
	Suppose additionally that $X$ is connected and $\sV$ is an admissible open covering of $X$. Then:
	\begin{enumerate}
		\item[(3)] $\OO_X(X)$ is an integral domain,
		\item[(4)] For any non-empty $U,V \in \sV$, there is a finite sequence $V_0, V_1, ... , V_n \in \sV$ with $V_0 = U$, $V_n = V$, and $V_m \cap V_{m+1} \neq \emptyset$ for $0 \leq m \leq n-1$.
	\end{enumerate}
\end{lemma}

\begin{proof}
	First suppose that we are in case (A) and $X$ is a scheme. Because $X \rightarrow \Spec(k)$ is smooth, $X \rightarrow \Spec(k)$ is locally of finite type. In particular, $X$ is locally noetherian, and therefore the underlying topological space of $X$ is locally noetherian \cite[Lem.\ 01OZ]{STACK} and thus locally connected \cite[Lem.\ 04MF]{STACK}. This gives point (1) by \cite[Lem.\ 04ME]{STACK}. From the fact that $X$ is locally noetherian we also have point (3) by \cite[Lem.\ 033N]{STACK} and \cite[Lem.\ 0358]{STACK}. Point (4) is \cite[Thm.\ 26.15]{WILLARD}.
	
	Now suppose that we are in case (B) and $X$ is a rigid space. Points (1) and (4) both follow from the definition and basic properties of connected components of rigid spaces, which can be found on \cite[pg.\ 492]{CON}, and point (3) follows from \cite[Lem.\ 3.1.5]{AW2}.

	In either case, point (2) follows from point (1): given any open covering $\{U_i\}_i$ of $X$ by elements of $\sB$, the refinement obtained by replacing each $U_i$ with its connected components is an admissible open covering, as these form an admissible open covering of $U_i$ by point (1).
\end{proof}

We first study how $c_X$ behaves under base change.

\begin{lemma}\label{baseechangecX}
For any finite field extension $L / k$, there is a canonical isomorphism
\[
	L \otimes_k c_X(X) \xrightarrow{\sim} c_{X_L}(X_L),
\]
where $X_L$ denotes the base change of $X$ to $L$.
\end{lemma}

\begin{proof}
	Suppose first that $U \in \sB_X$, and set $A \coloneqq \OO_X(U)$. By Lemma \ref{extendderivations}, as $A_L \coloneqq L \otimes_k A$ is \'{e}tale over $A$, there is an isomorphism
	\[
		L \otimes_k \Der_k(A) \rightarrow \Der_k(A_L),
	\]
	which by the uniqueness part of Lemma \ref{extendderivations} is explicitly given by letting $\lambda \otimes \partial$ act on $A_L$ by
	\[
		(\lambda \otimes \partial)(\mu \otimes a) =  \lambda \mu \otimes \partial(a).
	\]
	In particular we see that $\Der_k(A_L) = \Der_L(A_L)$, and that if $a \in A^{\Der_k(A)}$ and $\mu \in L$, then
	\[
		(\lambda \otimes \partial)(\mu \otimes a) = \lambda \mu \otimes \partial(a) = 0
	\]
	for any $\lambda \otimes \partial \in \Der_L(A_L)$ and so the identity of $A_L$ induces a well-defined map 
	\begin{equation}\label{eqn:basechangederivations}
		L \otimes_k A^{\Der_k(A)} \hookrightarrow A_L^{\Der_L(A_L)}.
	\end{equation}
	To see that this is actually an isomorphism, let $e_1, ... , e_r$ be a $k$-basis of $L$, so that any element $b \in A_L$ has the form
	\[
	b = \sum_{i = 1}^r e_i \otimes a_i
	\]
	for unique $a_i \in A$. If $b \in A_L^{\Der_L(A_L)}$, then for any $\partial \in \Der_k(A)$, $b$ is fixed by $1 \otimes \partial$, so
	\[
	\sum_{i = 1}^r e_i \otimes a_i = b = (1 \otimes \partial)(b) = \sum_{i = 1}^r e \otimes \partial(a_i).
	\]
	By the uniqueness of this description, $\partial(a_i) = a_i$ for any $i = 1, ... ,r$ and $\partial \in \Der_k(A)$, and thus $b$ is in the image of the map (\ref{eqn:basechangederivations}). Writing $\phi \colon X_L \rightarrow X$ for the projection, the morphism of sheaves 
	\[
	\OO_X \rightarrow \phi_* \OO_{X_L}
	\]
	maps the subsheaf $c_X$ into the subsheaf $\phi_*c_{X_L}$, this being true for any $U \in \sB_X$ by the above. Similarly, being true on $\sB_X$, the canonical morphism of sheaves
	\[
	L \otimes_k c_X \rightarrow \phi_*c_{X_L},
	\]
	is an isomorphism, and taking global sections we obtain the desired result. Here we have used that $L / k$ is finite to ensure that $L \otimes_k c_X$, defined by $(L \otimes_k c_X)(U) \coloneqq L \otimes_k c_X(U)$ for an admissible open subset $U \subset X$, is indeed a sheaf.
\end{proof}

The sheaf $c_X$ is a sheaf of $k$-algebras. The next lemma shows that for any admissible open subset $U$ of $X$, $c_X(U)$ is always a product of finite field extensions of $k$.

\begin{lemma}\label{lem:finiteextension}
Suppose that $U \subset X$ is a non-empty connected admissible open subset of $X$. Then $c_X(U)$ is a finite field extension of $k$.
\end{lemma}
\begin{proof}
	If $U \in \sB$, using the description of Lemma \ref{affinedesccX}, then in either case (A) or (B) the proof of \cite[Prop.\ 3.1.6]{AW2} shows that $c_X(U) = \OO_X(U)^{\sT_X(U) = 0}$ is a finite field extension of $k$.
	Now suppose that $U$ is any non-empty connected admissible open subset of $X$. Let $\sV$ be any admissible open covering of $U$ by connected elements of $\sB$, which exists by Lemma \ref{lem:opencoveringlemma}(2). In order to see that $c_X(U)$ is a field, suppose that $f \in c_X(U)$ with $f \neq 0$. As $f \neq 0$, there is some non-empty $V_0 \in \sV$ with $f_0 \coloneqq f|_{V_0} \neq 0$. For any other non-empty element $V \in \sV$, by Lemma \ref{lem:opencoveringlemma}(4) there is some finite sequence $V_1, ... , V_n \in \sV$ with $V_{m} \cap V_{m+1} \neq \emptyset$ for all $0 \leq m \leq n-1$, and $V_n = V$. Setting $f_m \coloneqq f|_{V_m}$, because each $c_X(V_m)$ is a field and $V_{m} \cap V_{m+1} \neq \emptyset$ the restriction maps $c_X(V_m) \rightarrow c_X(V_{m} \cap V_{m+1})$ are injective, and hence by induction $f|_V = f_n \neq 0$. Therefore $f \in c_X(U)$ is non-zero and has an inverse when restricted to any $V \in \sV$, and thus $f \in c_X(U)^\times$. Finally, as $c_X(U)$ is a field, $c_X(U) \rightarrow c_X(V)$ is injective for any $V \in \sV$, and therefore $c_X(U)$ is also a finite extension of $k$. 
\end{proof}

We now relate $c_X$ to how the connectivity of $X$ changes under field extension.

\begin{defn}\label{def:geomconn}
	We say that $X$ is \emph{geometrically connected} if for any finite extension $L$ of $k$, $X \times_k L$ is connected. 
\end{defn}

\begin{remark}\label{rem:geomconnsubtle}
If $X$ is a scheme, then $X$ is geometrically connected if and only if $X \times_k L$ is connected for any field extension $L$ of $k$ \cite[Lem.\ 0389]{STACK}.

For rigid spaces, the base change functor is more subtle. When $L$ is a finite extension of $k$, then just as for schemes $L \times_k -$ is defined as the fibre product functor $\Sp(L) \times_{\Sp(k)} -$ on the category of rigid spaces over $k$. When $L/k$ is an infinite extension of complete fields, this definition no longer makes sense, as $L$ is no longer a $k$-affinoid algebra. Nevertheless, one can still define a base change functor $L \times_k -$ for quasi-separated rigid spaces $Y$ over $k$ (see \cite[\S 9.3.6]{BGR} and \cite[\S 3.1]{CON} for more details). In this case, it is shown in \cite[\S 3.2]{CON} that (just like for schemes), $Y$ is geometrically connected if and only if $L \times_k Y$ is connected for any complete field extension $L$ of $k$. In particular, this holds for quasi-Stein $Y$ \cite[Prop.\ 9.6.1(7)]{BGR}.
\end{remark}

\begin{cor}\label{cor:cXuequalsk1}
Suppose that $U \subset X$ is a non-empty connected admissible open subset of $X$, and let $L$ be a Galois closure of $c_X(U)$. Then $U_L$ has at least $\dim_k c_X(U)$ connected components.

In particular, if $U$ is geometrically connected then $c_X(U) = k$.
\end{cor}

\begin{proof}
	The field $L$ is a finite extension of $k$ by Lemma \ref{lem:finiteextension}. Because $L/k$ is Galois, $L \otimes_k c_X(U)$ is isomorphic to the product of $[c_X(U):k]$ copies of $L$, and $L \otimes_k c_X(U) \xrightarrow{\sim} c_{X_L}(U_L) \subset \OO_{X_L}(U_L)$ by Lemma \ref{baseechangecX}. In particular, these $[c_X(U):k]$ idempotents of $\OO_{X_L}(U_L)$ give a disjoint union decomposition of $U_L$ into $[c_X(U):k]$ non-trivial parts, and therefore $U_L$ has at least $[c_X(U):k]$ connected components.
\end{proof}

We would now like to show the converse.

\begin{cor}\label{cor:cXuequalsk2}
Suppose that $U \subset X$ is a non-empty connected admissible open subset of $X$. Let $L$ be the Galois closure of $c_X(U)$ over $k$ in some fixed algebraic closure of $c_X(U)$. Then $U_{L}$ is a disjoint union of $\dim_{k} c_X(U)$ geometrically connected components.

In particular, if $c_X(U) = k$, then $U$ is geometrically connected.
\end{cor}

\begin{proof}
	Write $U_1, ... ,U_m$ for the connected components of $U_L$, which decompose $U_L$ as a disjoint union by Lemma \ref{lem:opencoveringlemma}(1). Each $U_i$ is a space over $L$, so $c_{X_L}(U_i) \supset L$, and thus by Lemma \ref{baseechangecX} 
	\[
		\dim_k c_X(U) = \dim_L(c_{X_L}(U_L)) = \sum_{i = 1}^{m} \dim_L c_{X_L}(U_i) \geq  m.
	\]
	On the other hand, $U_L$ has at least $\dim_k c_X(U)$ connected components by Corollary \ref{cor:cXuequalsk1}, and therefore $m = \dim_k c_X(U)$ and each $c_{X_L}(U_i) = L$. To show the claim it is thus sufficient to show that if $c_X(U) = k$, then $U$ is geometrically connected. If there was some finite extension $K$ of $k$ with $U_K$ disconnected, then $c_{X_K}(U_K)$ contains at least $K \times K$, the $K$-span of each idempotent. However,
	\[
		\dim_K c_{X_K}(U_K) = \dim_k c_X(U) = 1
	\]
	by Lemma \ref{baseechangecX}, a contradiction.
\end{proof}

We can also be more precise about the value of $c_X$ on connected admissible open subsets.

\begin{cor}\label{cor:maxfieldext}
	Suppose that $U \subset X$ is a non-empty connected admissible open subset of $X$. Then $c_X(U)$ is the unique maximal finite field extension of $k$ contained in $\OO_X(U)$.
\end{cor}

\begin{proof}
	Let $L$ be any finite field extension of $k$ contained in $\OO_X(U)$, and let $L \cdot c_X(U)$ be the $k$-algebra generated by $L$ and $c_X(U)$. If $\{e_i\}$ is a $k$-basis of $L$ and $\{f_j\}$ is a $k$-basis of $c_X(U)$ then this is generated as a $k$-vector space by $\{e_i f_j\}$. In particular, because $\OO_X(U)$ is an integral domain by Lemma \ref{lem:opencoveringlemma}(3) and $c_X(U)$ is a finite field extension of $k$ by Lemma \ref{lem:finiteextension}, $L \cdot c_X(U)$ is a finite dimensional field extension of $c_X(U)$. Therefore, it is sufficient to show that if $L$ is a finite field extension of $k$ containing $c_X(U)$, then $L = c_X(U)$. Suppose that $L \supset c_X(U)$ is such a field extension, and write $K$ for a Galois closure of $L$. By Lemma \ref{baseechangecX},
	\[
		\dim_k c_X(U) = \dim_K c_{X_K}(U_K).
	\]
	However, $U_K$ has at least $\dim_k L$ connected components by Corollary \ref{cor:cXuequalsk1}, and so the right-hand side is at least $\dim_k L$. Therefore, $\dim_k c_X(U) \geq \dim_k L$, and $c_X(U) = L$ as required.
\end{proof}

Our collected facts about the sheaf $c_X$ have the following interesting consequence.

\begin{remark}
Any connected space $X$ for which $\OO_X(X)$ contains a proper non-trivial finite field extension of $k$ provides an ``obvious'' example of a connected but not geometrically connected space (cf. the proof of Corollary \ref{cor:cXuequalsk1}). 

In fact, we now see that this is the only way that $X$ can fail be to geometrically connected: if $X$ is connected but not geometrically connected, then $\OO_X(X)$ contains a proper field extension of $k$, namely $c_X(X)$, by Corollary \ref{cor:cXuequalsk2}.
\end{remark}

When $X$ is a scheme, we can describe the sheaf $c_X$ more concretely. For a $k$-algebra $A$ we denote by $\underline{A}$ the constant sheaf with value $A$, by which we mean the sheafification of the constant presheaf $A^{\text{psh}}$, which has value $A$ on every admissible open subset $U \subset X$ and all restriction maps the identity morphism of $A$.

\begin{prop}
When $X$ is a connected scheme, $c_X$ is the constant sheaf $\underline{c_X(X)}$.
\end{prop}
\begin{proof}
	We will show that for any non-empty connected open subset $U \subset X$ the restriction map $c_X(X) \rightarrow c_X(U)$ is an isomorphism, from which it is a straightforward sheaf-theoretic argument to show that the natural map $c_X(X)^{\text{psh}} \rightarrow c_X$ is a sheafification. To this end, let $U \subset X$ be a non-empty connected open subset. Because $c_X(X)$ is a field by Lemma \ref{lem:finiteextension} and $U$ is non-empty, $c_X(X) \rightarrow c_X(U)$ is injective. Writing $L$ for a Galois closure of $c_X(X)$, then $U_L \hookrightarrow X_L$, and $X_L$ is a disjoint union of $\dim_k(c_X(X))$ geometrically connected components by Corollary \ref{cor:cXuequalsk2}. Each connected component $V$ of $U_L$ is contained in a geometrically connected component $Y$ of $X_L$. As $Y$ is smooth and geometrically connected, $Y$ is geometrically normal and thus geometrically irreducible. In particular, $V$ is also geometrically irreducible, and hence $c_{X_L}(V) = L$ by Corollary \ref{cor:cXuequalsk1}. As any connected component $Y$ of $X_L$ is irreducible, $U_L$ is thus a disjoint union of at most $\dim_k(c_X(X))$ geometrically connected components, hence
	\[
		\dim_k(c_X(U)) = \dim_L(c_{X_L}(U_L)) \leq \dim_k(c_X(X))
	\]
	by Lemma \ref{baseechangecX} and $c_X(X) \xrightarrow{\sim} c_X(U)$.
\end{proof}
However in the rigid case, case (B), the sheaf $c_X$ can be very far from a constant sheaf. This is shown in the following example, due to J\'{e}r\^{o}me Poineau.

\begin{eg}\label{sheaffcXondisk}
Let $X = \bD = \Sp(\bQ_p\langle x \rangle)$ be the rigid analytic unit disk over $\bQ_p$. Then $\bD$ is geometrically connected, but we can construct open subsets $U \subset \bD$ which are connected but not geometrically connected, and in fact have $c_X(U)$ being an extension of $\bQ_p$ of arbitrarily large degree. Indeed, let $w \in \overline{\bQ_p} \setminus \bQ_p$ with $|w| \leq 1$, and let $f(x)$ be the minimal polynomial of $w$ over $\bQ_p$. For $r \in p^{\bQ}$, we can consider the affinoid open subset 
\[
	U = \{z \in \bD \mid |f(z)| \leq r \} \subset \bD.
\]
If $L$ is a splitting field for $f(x)$ over $\bQ_p$, then for $r$ small enough, $U_L \subset \bD_L$ is the disjoint union of $\dim_{\bQ_p} L$ closed disks. By Lemma \ref{lemma:conncompgaloiscovering}, the Galois covering $U_L \rightarrow U$ over $\bQ_p$ restricts to an isomorphism from each connected component to $U$, and thus $U$ is connected. Therefore $U$ is connected but not geometrically connected, and in fact $\dim_{\bQ_p} c_X(U) = \dim_{\bQ_p} L$ by Lemma \ref{baseechangecX}.
\end{eg}

\section{The Functors \texorpdfstring{$\OO_X \otimes_k -$}{O_X o_k -} and \texorpdfstring{$\Hom_{G\text{-}\sD_X}(\OO_X, -)$}{Hom_G-D(O, -)}}\label{sect:RH}

In this section, we have the following running assumptions.
\begin{enumerate}
	\item $X$ is as in Section \ref{sect:Dmodules} with $k$ of characteristic $0$,
	\item $X$ has an action by a product of abstract groups $G \times H$.
\end{enumerate}

For example, both $G$ and $H$ could be trivial, or infinite.
\begin{remark}
We will often impose the assumption that there are no global non-trivial $G$-invariant constant functions: $c_X(X)^G = k$. This is always satisfied whenever the connected components of $X$ are geometrically connected and the action of $G$ on the set of connected components is transitive. When $G$ is trivial, then the condition that $c_X(X) = k$ is equivalent to the assumption that $X$ is geometrically connected by Corollary \ref{cor:cXuequalsk1} and Corollary \ref{cor:cXuequalsk2}.
\end{remark}
We will make use of the following consequence of the assumption that $c_X(X)^G = k$.
\begin{lemma}\label{lem:OXirrGDX}
Suppose that $c_X(X)^G = k$. Then $\OO_X$ is irreducible as a $G\text{-}\sD_X$-module.
\end{lemma}
\begin{proof}
	First note that $X$ is the disjoint union of its connected components by Lemma \ref{lem:opencoveringlemma}(1). When $X$ is a rigid space this follows by definition (see \cite[\S 2.1]{CON}), and when $X$ is a scheme this follows because $X$ is locally noetherian. Furthermore, $G$ acts transitively on the set of connected components of $X$. Indeed, if $X = X_1 \sqcup X_2$ were a $G$-stable disjoint union, then we would have $k \times k \subset c(X_1)^G \times c(X_2)^G \subset c_X(X)^G$. Suppose now that $\sF$ is a proper non-trivial $G$-$\sD_X$ submodule of $\OO_X$. Then there is some admissible open subset $U \subset X$ with $\sF(U) \neq 0$, and so we may find some non-zero $x \in \sF(U)$. As $x \neq 0$ in $\sF(U)$, and the connected components $\{X_i\}_i$ form an admissible open cover of $X$, there is some $X_i$ with
	\[
	0 \neq x|_{U \cap X_i} \in \sF(U \cap X_i) \subset \OO_{X_i}(U \cap X_i).
	\]
	Therefore $\sF|_{X_i}$ is a non-trivial $\sD_{X_i}$-submodule of $\OO_{X_i}$, and thus $\sF|_{X_i} = \OO_{X_i}$ by Lemma \ref{connirred} because $X_i$ is connected. By the transitivity of the $G$ action on the set of connected components and the fact that $\sF \hookrightarrow \OO_{X}$ is $G$-linear, then for any other connected component $X_j$ we also have that $\sF|_{X_j} = \OO_{X_j}$. Therefore, as this is true for any $X_j$, $\sF = \OO_X$.
\end{proof}

From the action of $G \times H$ on $X$, $\sD_X$ is a $G \times H$-equivariant sheaf, and we can consider the category $\VectCon^{G \times H}(X)$ (see Section \ref{sect:generaleqDmodules}). In this section we define and study properties of a pair of functors between $\Mod_{k[H]}^{\fd}$ and $\VectCon^{G \times H}(X)$. 

\addtocontents{toc}{\SkipTocEntry}
\subsection*{The Functor \texorpdfstring{$\OO_X \otimes_k -$}{O_X o_k -}}
In one direction, we define,
\[
\OO_X \otimes_k - \colon \Mod_{k[H]}^{\fd} \rightarrow \VectCon^{G \times H}(X)
\]
by setting $\OO_X \otimes_k V$ to be the sheaf of $\sD_X$-modules with
\[
	(\OO_X \otimes_k V)(U) = \OO_X(U) \otimes_k V
\]
for any admissible open subset $U \subset X$, where $x \in \sD_X(U)$ acts by
\[
	x \cdot (s \otimes v) \coloneqq x s \otimes v.
\]
This has a $G \times H$-equivariant structure defined by
\[
	(g,h)^V \coloneqq (g,h)^{\OO_X} \otimes h(-) \colon \OO_X \otimes_k V \rightarrow  (g,h)^{-1} (\OO_X \otimes_k V),
\]
and with this $G \times H$-equivariant structure $\OO_X \otimes_k V$ is a $G \times H$-equivariant sheaf of $\sD_X$-modules. Given a $k[H]$-module homomorphism $V \rightarrow W$, the induced morphism $1 \otimes f$ is defined in the obvious manner. It is furthermore direct to verify that the functor $\OO_X \otimes_k -$ is monoidal.

\addtocontents{toc}{\SkipTocEntry}
\subsection*{The Functor \texorpdfstring{$\Hom_{G \text{-}\sD_X}(\OO_X, -)$}{Hom_G-D(O, -)}}
In the other direction, we have the \emph{solution functor},
\[
	\Hom_{G\text{-}\sD_X}(\OO_X, -) \colon \VectCon^{G \times H}(X) \rightarrow \Mod_{k[H]}^{\fd}.
\]
Here, for $\sM \in \VectCon^{G \times H}(X)$, $H$ acts on $\Hom_{\sD_X}(\OO_X, \sM)$ as in Remark \ref{rem:homspaceequivsheaf}, and this restricts to an action of $H$ on $\Hom_{G\text{-}\sD_X}(\OO_X, \sM)$ because the actions of $G$ and $H$ commute.

We first show that $\Hom_{G\text{-}\sD_X}(\OO_X, \sM)$ is well-defined: that if $\sM$ is a $(G \times H)\text{-}\sD_X$-module which is coherent as an $\OO_X$-module, then $\Hom_{G\text{-}\sD_X}(\OO_X, \sM)$ is finite dimensional as a $k$-vector space. This is a consequence of the following lemma.

\begin{lemma}\label{lem:solfunctorinjective}
Suppose that $c_X(X)^G = k$ and $\sM \in \VectCon^{G \times H}(X)$. Then the natural map
\[
	 \OO_X \otimes_k \Hom_{G \text{-}\sD_X}(\OO_X, \sM) \rightarrow \sM
\]
is $(G \times H)\text{-}\sD_X$-linear and injective. In particular, $\dim_k(\Hom_{G\text{-}\sD_X}(\OO_X, \sM)) \leq \rank(\sM)$.
\end{lemma}

\begin{proof}
	The $(G \times H)\text{-}\sD_X$-linearity is direct to verify from the definitions. For the injectivity, we proceed as follows. Suppose that $f_1, ... ,f_r \in \Hom_{G \text{-}\sD_X}(\OO_X, \sM)$ are $k$-linearly independent, and define $e_1, ... , e_r \in \sM(X)^G$ by $e_i \coloneqq f_i(1_X)$. These are also $k$-linearly independent by \cite[Lem.\ 3.1.4]{AW2}. Because $\OO_X$ is irreducible as a $G$-$\sD_X$-module by Lemma \ref{lem:OXirrGDX}, it is sufficient for us to show that the sum
	\[
		\sum_{i = 1}^r \OO_X \cdot e_i \hookrightarrow \sM
	\]
	is direct. We prove this by induction on $r \geq 1$. When $r = 1$ this is trivially true, so suppose that the statement is true for some fixed $r \geq 1$, and consider
	\[
		\sum_{i = 1}^{r+1} \OO_X \cdot e_i.
	\]
	After rearranging the factors if necessary, it is sufficient to show that
	\[
		\left( \bigoplus_{i = 1}^{r} \OO_X \cdot e_i \right) \bigcap \OO_X \cdot e_{r+1} = 0.
	\]
	If this intersection were non-zero, then 
	\[
		\left( \bigoplus_{i = 1}^{r} \OO_X \cdot e_i \right) \bigcap  \OO_X \cdot e_{r+1} =  \OO_X \cdot e_{r+1},
	\]
	by the irreducibility of $\OO_X \cdot e_{k+1}$ (Lemma \ref{lem:OXirrGDX}). We can therefore write
	\[
		e_{r+1} = \lambda_1 e_1 + \cdots + \lambda_r e_r
	\]
	for unique $\lambda_i \in \OO_X(X)$. We have that for any admissible open subset $U$ and $\partial \in \sT_X(U)$,
	\begin{align*}
		0 &= \partial(e_{r+1}) = \partial(\lambda_1 e_1 + \cdots + \lambda_k e_r), \\
		&= \partial(\lambda_1)e_1 + \cdots + \partial(\lambda_k) e_r,
	\end{align*}
	in $\sM(U)$, and therefore as the sum is direct, $\partial(\lambda_i) = 0$ for all $i = 1, ... ,r$. Therefore, as this holds for each admissible open $U \in \sB$, each $\lambda_i$ is a global section of the sheaf $c_X$ by Lemma \ref{affinedesccX}. Furthermore, each $\lambda_i \in c_X(X)^G = k$ because each $e_i \in \sM(X)^G$. However this is a contradiction, as we know that $e_1, ... ,e_r, e_{r+1}$ are linearly independent over $k$.
\end{proof}

\begin{thm}\label{FunctorkGGDX}
Suppose that $X, G$ and $H$ are as described at the start of this section. Then:
\begin{enumerate}
	\item The functor
\[
	\OO_X \otimes_k - \colon \Mod_{k[H]}^{\fd} \rightarrow \VectCon^{G \times H}(X)
\]
is exact, monoidal and faithful.
\end{enumerate}
Suppose additionally that $c_X(X)^G = k$. Then:
\begin{enumerate}
\item[(2)] The functor $\OO_X \otimes_k -$ is full,
\item[(3)]
The essential image of $\OO_X \otimes_k -$ is the full subcategory with objects those $\sM$ which satisfy
\[
	\dim_k(\Hom_{G \text{-}\sD_X}(\OO_X, \sM)) = \rank(\sM),
\]	
and on this subcategory $\Hom_{G \text{-}\sD_X}(\OO_X, -)$ is a quasi-inverse for $\OO_X \otimes_k -$.
\item[(4)]
The essential image of $\OO_X \otimes_k -$ is closed under sub-quotients.
\end{enumerate}
\end{thm}

\begin{proof}
	The first point (1) is direct. For point (2), let $\sU$ be an admissible open covering by elements of $\sB$. Given any $(G \times H)$-$\sD_X$-linear morphism 
	\[
		f \colon \OO_X \otimes_k V \rightarrow \OO_X \otimes_k W,
	\]
	then for each $U \in \sU$, $v \in V$ and $\partial \in \sT_U(U)$,
	\[
		\partial(1_U \otimes v) = \partial(1_U) \otimes v = 0,
	\]
	and hence
	\[
		f_U(1_U \otimes v) \in (\OO_X(U) \otimes_k W)^{\sT_U(U) = 0} = c_X(U) \otimes_k W,
	\]
	by Lemma \ref{affinedesccX} so
	\[
		f_X(1_X \otimes v) \in c_X(X) \otimes_k W.
	\]
	Now because $f$ is a morphism of $G$-equivariant sheaves, and using that $c_X(X)^G = k$,
	\[
		f_X(1_X \otimes v) \in c_X(X)^G \otimes_k W = k \otimes_k W,
	\]
	which allows us to define a morphism $\lambda \colon V \rightarrow W$ uniquely determined by the property that
	\[
	f_X(1_X \otimes v) = 1 \otimes \lambda(v).
	\]
	Because $f$ is $H$-$\sD_X$-linear, $\lambda$ is $k[H]$-linear, and $f = 1 \otimes \lambda$. Therefore $\OO_X \otimes_k -$ is full.

	For point (3), by Lemma \ref{lem:solfunctorinjective} we have an $(G \times H)\text{-}\sD_X$-linear injection
	\begin{equation}\label{eqn:injsamerank}
		\OO_X \otimes_k \Hom_{G \text{-}\sD_X}(\OO_X, \sM) \hookrightarrow \sM
	\end{equation}
	for any $\sM \in \VectCon^{G \times H}(X)$, and by taking the quotient we can extend this to a short exact sequence,
	\[
		0 \rightarrow \OO_X \otimes_k \Hom_{G \text{-}\sD_X}(\OO_X, \sM) \rightarrow \sM \rightarrow \sN \rightarrow 0.
	\]
	The quotient $\sN$ is coherent as an $\OO_X$-module, and therefore by Lemma \ref{cohimplieslocallyfree}, $\sN$ is in fact locally free. Therefore, if $\rank(\OO_X \otimes_k \Hom_{G \text{-}\sD_X}(\OO_X, \sM)) = \rank(\sM)$, then $\rank(\sN) = 0$ and thus $\sN = 0$, so (\ref{eqn:injsamerank}) is an isomorphism. Consequently, if $\dim_k \Hom_{G \text{-}\sD_X}(\OO_X, \sM) = \rank(\sM)$ then $\sM$ is in the essential image of $\OO_X \otimes_k -$ and $\Hom_{G \text{-}\sD_X}(\OO_X, -)$ provides a right quasi-inverse to $\OO_X \otimes_k -$ on this full subcategory.

	On the other hand, if $\sM = \OO_X \otimes_k V$ is in the image of $\OO_X \otimes_k -$, then
	\[
		\Hom_{G \text{-}\sD_X}(\OO_X, \OO_X \otimes_k V) = \Hom_{G \text{-}\sD_X}(\OO_X \otimes_k k, \OO_X \otimes_k V) = V,
	\]
	 which follows from the fully faithfulness of $\OO_X \otimes_k -$ (part (1)) in the case that $H$ is trivial. Therefore, any $\sM$ in the essential image satisfies $\dim_k \Hom_{G \text{-}\sD_X}(\OO_X, \sM) = \rank(\sM)$, and $\Hom_{G \text{-}\sD_X}(\OO_X, -)$ provides a left quasi-inverse to $\OO_X \otimes_k -$.

	For point (4), because the functor is exact and full, it is sufficient to show that the essential image is closed under sub-objects. Suppose that Given $V \in \Mod_{k[H]}^{\fd}$, suppose that
	\[
	\sM \subset \OO_X \otimes_k V
	\]
	in $\VectCon^{G\times H}(X)$. As above, the quotient of $\sN$ of $\OO_X \otimes_k V$ by $\sM$ is in $\VectCon^{G\times H}(X)$, and applying the left-exact $\sH \coloneqq \Hom_{G \text{-}\sD_X}(\OO_X, -)$ and then the exact $\OO_X \otimes_k -$ we have a commutative diagram with exact rows
\[\begin{tikzcd}
	0 & {\OO_X \otimes_k \sH(\sM)} & {\OO_X \otimes_k \sH(\OO_X \otimes_k V)} & {\OO_X \otimes_k \sH( \sN)} \\
	0 & \sM & {\OO_X \otimes_k V} & \sN
	\arrow[from=1-1, to=1-2]
	\arrow[from=1-2, to=1-3]
	\arrow[hook, from=1-2, to=2-2]
	\arrow[from=1-3, to=1-4]
	\arrow[hook, from=1-3, to=2-3]
	\arrow[hook, from=1-4, to=2-4]
	\arrow[from=2-1, to=2-2]
	\arrow[from=2-2, to=2-3]
	\arrow[from=2-3, to=2-4]
\end{tikzcd}\]
	A simple diagram chase shows that the left vertical map is an isomorphism, and therefore $\sM$ is in the essential image of $\OO_X \otimes_k -$.
\end{proof}

\section{Finite Equivariant Vector Bundles with Connection}\label{sect:finiteVB}

In this section we let $X$ and $G$ be as described at the start of Section \ref{sect:RH} (with $H$ taken to be trivial) and assume that $X$ has a $k$-rational point $z \in X(k)$. As described in Section \ref{sect:generaleqDmodules} $\VectCon^G(X)$ has a natural tensor structure and as such is a rigid abelian tensor category. An identity object of $\VectCon^G(X)$ (in the sense of \cite{DELMIL}) is given by $\OO_X$, and by our assumption that $c_X(X)^G = k$ and Lemma \ref{altdesccX}, $\End(\OO_X) = k$ in $\VectCon^G(X)$. Furthermore, the $k$-rational point $z \in X(k)$ allows us to define a fibre functor
\[
\omega_z \colon \VectCon^{G}(X) \rightarrow \Vect_k, \qquad \omega_z \colon \sV \mapsto \sV(z) \coloneqq \sV_z \otimes_{\OO_{X,z}} k(z),
\]
which is exact because the sheaves in $\VectCon^{G}(X)$ are locally free. By \cite[Prop.\ 1.19]{DELMIL}, $\omega_z$ is faithful, and thus $\VectCon^{G}(X)$ is in this way a neutral Tannakian category \cite[Prop.\ 1.20]{DELMIL}.

In this section we consider the (rigid tensor) subcategory
\[
\VectCon^{G}(X)_{\fin} \subset \VectCon^{G}(X)
\]
of $G$-equivariant vector bundles with connection on $X$ which are \emph{finite}. This is a notion that was introduced by Andr\'{e} Weil \cite{WEIL} in the context of vector bundles on complex projective varieties, who showed that the pushforward of the structure sheaf along a finite \'{e}tale morphism is a finite vector bundle. The converse was shown by Nori for proper integral schemes over a field with a rational point \cite{NORI}. In this section we prove the corresponding version of their results in our context, for the category $\VectCon^{G}(X)$. The idea to remove the properness assumption and replace this by the extra data of an integrable connection was first considered by Esnault and Hai \cite{ESHAI}. Our method of proof is most similar in style to the approach of Biswas and O'Sullivan \cite{BisOS}, who prove the analogue of Nori's result when $X$ is a complex analytic space and $G$ is a complex Lie group.

First, let $\sC$ be any neutral Tannakian category in the sense of \cite[Def.\ 2.19]{DELMIL}. We will later specialise to $\sC = \VectCon^G(X)$. We have the following immediate property of $\sC$.

\begin{lemma}\label{lem:KRS}
	The Krull-Remak-Schmidt Theorem holds in $\sC$: any non-zero object $\sV \in \sC$ can be written as a direct sum of indecomposable objects
	\[
		\sV \cong \sV_1 \oplus \cdots \oplus \sV_n,
	\]
	and any such description is unique up to isomorphism and permutation.
\end{lemma}
	
\begin{proof}
	This follows from \cite[Thm.\ 1]{ATI}, using \cite[\S 3 Cor.]{ATI} and the fact that $\Hom(\sV ,\sW)$ is finite dimensional over $k$ for any $\sV, \sW \in \sC$ because the fibre functor is faithful.
\end{proof}

\begin{defn}
For $f  = \sum_{n \geq 0} a_n x^n \in \bZ_{\geq 0}[x]$ and $\sV \in \sC$, we define
\[
	f(\sV) \coloneqq \bigoplus_{n \geq 0} (\sV^{\otimes n})^{\oplus a_n} \in \sC.
\]
\end{defn}

\begin{defn}\label{defn:finiteobject}
	$\sV \in \sC$ is called \emph{finite} if there exist $f,g \in \bZ_{\geq 0}[x]$ with $f \neq g$ and $f(\sV) \cong g(\sV)$. We write $\sC_{\fin}$ for the full subcategory of finite objects.
\end{defn}

\begin{defn}
For $\sV \in \sC$, we write $I(\sV)$ for the set of isomorphism classes of indecomposable direct summands of $\sV$ in $\sC$, and set
\[
S(\sV) \coloneqq \bigcup_{r \geq 0} I(\sV^{\otimes r}).
\]
\end{defn}

\begin{lemma}\label{lem:finiteequiv}
Suppose that $\sV \in \sC$. Then:
\begin{itemize}
	\item $I(\sV)$ is a finite set,
	\item $S(\sV)$ is a finite set if and only if $\sV$ is finite.
\end{itemize}
In particular, if $\sL \in \VectCon^G(X)$ has rank $1$, then $\sL$ is finite if and only if $[\sL] \in \Pic^G(X)$ is torsion.
\end{lemma}

\begin{proof}
	Let $\sV \in \sC$. If $\sV = 0$ then the statement holds, so we may assume that $\sV \neq 0$. That $I(\sV)$ is finite is a direct consequence of the Krull-Remak-Schmidt Theorem (Lemma \ref{lem:KRS} above).
	For the second point we follow the proof of \cite[Prop.\ 6.7.4]{SZAM}, which we include for the convenience of the reader. Suppose that $S(\sV)$ is finite, and consider the free abelian group $A$ generated by the isomorphism classes $[\sW]$ of indecomposable objects of $\sC$, with subgroup $A(\sV)$ generated by the set $S(\sV)$. There is a well-defined $\bZ$-linear map $m_{\sV} \colon A \rightarrow A$ defined on each generator $[\sW]$ to be
	\[
		m_{\sV}([\sW]) \coloneqq [\sV \otimes \sW] \coloneqq [\sW_1] + \cdots + [\sW_n],
	\]
	where 
	\[
		\sV \otimes \sW = \sW_1 \oplus \cdots \oplus \sW_n
	\]
	is the unique decomposition into indecomposable objects of Lemma \ref{lem:KRS}, and this linear map preserves $A(\sV)$. By the assumption that $S(\sV)$ is finite, $A(\sV)$ is a finitely generated free abelian group, which is non-zero because $\sV \neq 0$, and we may therefore consider the characteristic polynomial $\chi \in \bZ[x]$ of $m_{\sV}$, which is monic and satisfies $\chi(m_{\sV}) = 0$. If we write $\chi = f - g$, where $f, g \in \bZ_{\geq 0}[x]$, then applying $\chi(m_{\sV})$ to the identity object of $\sC$ we see that $f(\sV) \cong g(\sV)$.

	Conversely, suppose that we have $f, g \in \bZ_{\geq 0}[x]$ with $f \neq g$, and that $f(\sV) \cong g(\sV)$. As $\sV \neq 0$, $f-g$ has degree $d > 0$. By the uniqueness of Lemma \ref{lem:KRS} we may write any element of $I(\sV^{\otimes d})$ as a sum of elements of $I(\sV^{\otimes k})$ for $k < d$. Similarly, for any $i \geq 0$, using the isomorphism $(x^if)(\sV) \cong (x^ig)(\sV)$, we can write, for any $m \geq d$, $\sV^{\otimes m}$ as the sum of elements of $I(\sV^{\otimes k})$ for $k < d$, and thus
	\[
		S(\sV) \subset \bigcup_{k = 0}^{d-1} I(\sV^{\otimes k})
	\]
	and thus $S(\sV)$ is finite. For the final claim, note that because $\End(\OO_X) = k$, the rank of any object of $\VectCon^G(X)$ is $\bN$-valued, and so if $\sL$ has rank $1$ then each $\sL^{\otimes k}$ is indecomposable. In particular, $S(\sL) = \{\sL^{\otimes k}\}_{k \geq 0}$ and thus $\sL$ is finite if and only if $[\sL] \in \Pic^G(X)$ is torsion.
\end{proof}

\begin{cor}\label{cor:finiteclosure}
	$\sC_{\fin}$ is closed under duals, direct sums, direct summands, and tensor products.
\end{cor}

Now we specialise to the case where $\sC = \VectCon^G(X)$.

\begin{prop}\label{prop:galoisisfinite}
Suppose that $f \colon Z \rightarrow X$ is a $G$-equivariant finite \'{e}tale Galois covering and that the action of $G$ commutes with the Galois action. Then $f_*\OO_Z \in \VectCon^G(X)_{\fin}$.
\end{prop}

\begin{proof}
	Let $\sA \coloneqq f_* \OO_Z$. Writing $H$ for the Galois group of $Z \rightarrow X$, the Galois isomorphism
	\[
		Z \times \underline{H} \xrightarrow{\sim} Z \times_X Z
	\]
	induces an isomorphism of $\sA \otimes \sA$ with the direct sum of $|H|$ copies of $\sA$ in $\VectCon^G(X)$, this isomorphism being $\sD$-linear as the action of $\sT_X$ and $H$ commute (Lemma \ref{dercommuteG}), and $G$-equivariant because the actions of $G$ and $H$ commute. Therefore, $I(\sA^{\otimes k}) \subset I(\sA)$ for all $k \geq 1$, and thus $\sA$ is finite by Lemma \ref{lem:finiteequiv}. 
\end{proof}

Now we would like to show the converse. We will make use of the following notion. For a set $\sS$ of objects of a rigid abelian tensor category $\sC$, we write $\sC(\sS)$ for the full subcategory of $\sC$ with objects those isomorphic to the underlying object of $\text{coker}(W \hookrightarrow V)$, where $W \hookrightarrow V$ is a monomorphism between objects of $\sC$, and $V$ is such that there exists a monomorphism
\[
	V \hookrightarrow \bigoplus_{i = 1}^r S_i
\]
for some $r \geq 1$ and $S_1, ... , S_r \in \sS$. This is an abelian subcategory of $\sC$, which is rigid whenever $\sS$ is closed under duality up to isomorphism. This is further a tensor subcategory of $\sC$ whenever the tensor product of any two objects of $\sS$ is isomorphic to a sub-object of a direct sum of elements of $\sS$.

\begin{prop}\label{prop:finitecomesfromcovering}
	Suppose that $\sV \in \VectCon^G(X)_{\fin}$ and $k$ is algebraically closed. Then there is some $G$-equivariant finite \'{e}tale Galois covering $f \colon Z \rightarrow X$ such that the action of $G$ commutes with the Galois action and $\sV$ is a direct summand of $f_* \OO_Z$.
\end{prop}

\begin{proof}
	Suppose that $\sV \in \VectCon^G(X)_{\fin}$, and write $\sC$ for the full abelian subcategory $\sC(\sS_0)$ of $\VectCon^G(X)$ as defined above, where $\sS_0$ is a set of representatives of $S(\sV) \cup S(\sV^*)$. The set $\sS_0$ is closed under duality up to isomorphism and the tensor product of any two objects of $S_0$ is isomorphic to a sub-object of a direct sum of elements of $\sS_0$ because $\sV$ is finite, hence $\sC$ is a rigid abelian tensor subcategory. In particular, with the fibre functor $\omega_z$, $\sC$ is a neutral Tannakian category in its own right. Therefore, by \cite[Thm.\ 2.11]{DELMIL}, there is an equivalence of categories
	\[
		F \colon \Rep_k(H) \rightarrow \sC
	\]
	for some affine group scheme $H$ over $k$, which is finite because every object of $\sC$ is isomorphic to a sub-quotient of a direct sum of copies of $T$, the direct sum of all elements of $\sS_0$ \cite[Prop.\ 2.20(a)]{DELMIL}. Note further, that because $k$ has characteristic $0$ and is algebraically closed, $H$ is a constant group. Because $\sV$ is in the essential image of $F$, $\sV$ is a direct summand of a finite direct sum of copies of $\sA \coloneqq F(\OO(H))$. Therefore, it is sufficient for us to show that there is a $G$-equivariant finite \'{e}tale Galois covering $f \colon Z \rightarrow X$ such that $\sA = f_* \OO_Z$. Indeed, given such a covering $f \colon Z \rightarrow X$ with Galois group $H$, then for any $n \geq 1$ we can consider the $G$-equivariant Galois covering given by
	\[
	h \colon Y \coloneqq \bigsqcup_{i = 1}^n Z \rightarrow X,
	\]
	with Galois group $C_n \times H$ where $C_n$ permutes the disjoint union, which has $h_*\OO_Y = \sA^{\oplus n}$.

	We construct such a covering $f \colon Z \rightarrow X$ as follows. First, interpreting $\Rep_k(\OO(H))$ as the category of finite dimensional $\OO(H)$-comodules, we note that the algebra multiplication
	\[
		m \colon \OO(H) \otimes_k \OO(H) \rightarrow \OO(H)
	\]
	is an $\OO(H)$-comodule homomorphism, and therefore we may apply $F$ to obtain
	\[
		F(m) \colon \sA \otimes_{\OO_X} \sA \rightarrow \sA,
	\]
	which gives $\sA$ the structure of a finite sheaf of $\OO_X$-algebras. We may therefore define
	\[
	f \colon Z \coloneqq \underline{\Spec}(\sA) \rightarrow X	
	\]
	to be the unique finite covering of $X$ with $f_* \OO_Z = \sA$ (as sheaves of $\OO_X$-algebras). We now define an action of the group $G$ on $Z$, which we give as a group homomorphism $\rho \colon G^{\op} \rightarrow \Aut_X(Z)$. For $g \in G$ and $U \in \sB_X$, $g^{\sA}_U \colon \sA(U) \rightarrow \sA(g(U))$ is an algebra morphism because $F(m)$ is $G$-equivariant, being a morphism in $\VectCon^G(X)$. Therefore we may define
	\[
		\rho(g)|_U \coloneqq \Spec(g^{\sA}_U) \colon Z|_{f^{-1}(U)} = \Spec(\sA(U)) \rightarrow \Spec(\sA(g(U))) = Z|_{f^{-1}(g(U))},
	\]
	which glue to define the required $\rho(g) \in \Aut_X(Z)$. It is direct to check that $\rho$ is a group homomorphism, and with respect to this action of $G$ the equality $f_* \OO_Z = \sA$ is one of $G$-$\OO_X$-modules.
	
	We further define an action of $H$ on $Z$ as follows. For any $h \in H$, we have the $\OO(H)$-comodule homomorphism
	\[
		r_h \colon \OO(H) \rightarrow \OO(H), \qquad r_h(\phi) = \phi(-h),
	\]
	which over all $H$ defines a group homomorphism $H \rightarrow \Aut_{\text{Comod}}(\OO(H))$. Therefore, we obtain a left action of $H$ on $\sA$ by letting $h \in H$ act by $F(r_h)$. Because each $r_h$ is furthermore a $k$-algebra homomorphism of $\OO(H)$, we obtain a group homomorphism $H \rightarrow \Aut(Z)$ and therefore an action $a \colon Z \times \underline{H} \rightarrow Z$ on the space $Z$, which commutes with the action of $G$ and for which the morphism $f \colon Z \rightarrow X$ is equivariant with respect to the trivial action of $H$ on $X$. To check that this is finite \'{e}tale Galois, first note that the morphism $\OO_X \rightarrow \OO_Z^H$ is an isomorphism because this is the image under $F$ of the morphism $k \rightarrow \OO(H)^H$ which is itself an isomorphism. The morphism
	\[
		p_Z \times a \colon Z \times \underline{H} \rightarrow Z \times_X Z
	\]
	corresponds to the morphism
	\[
		\sA \otimes_{\OO_X} \sA \rightarrow \sA \otimes_k \OO(H)
	\]
	in $\VectCon^G(X)$, which is an isomorphism precisely because
	\[
		\OO(H) \otimes_k \OO(H) \rightarrow \OO(H) \otimes_k \OO(H), \qquad \phi \otimes \psi \mapsto \sum_{h \in H} \phi \psi(-h) \otimes \delta_h
	\]
	is an isomorphism. The fact that $f \colon Z \rightarrow X$ is finite \'{e}tale follows from the fact that $p_Z$ is an isomorphism as, working locally with $U \in \sB_X$, the Galois extension of commutative $k$-algebras $\OO_X(U) \rightarrow \OO_Z(f^{-1}(U))$ is automatically finite \'{e}tale by \cite[Thm.\ 1.3(a)]{CHR}.

	Finally, we need to verify that the equality $f_* \OO_Z = \sA$ of $G$-$\OO_X$-modules is actually an equality of $G$-$\sD_X$-modules, for which it suffices to show that for any $U \in \sB_X$, the action of $\sT(U)$ on $\sA(U)$ is the same as the natural action of $\sT(U)$ on $(f_*\OO_Z)(U) = \OO(V)$, where $V \coloneqq f^{-1}(U)$.
	
	By definition of the $\sD$-module pushforward (Section \ref{section:inversedirectimage}), $\sT(U)$ acts on $\OO(V)$ via the $\OO(U)$-linear map $\psi \colon \sT(U) \rightarrow \sT(V)$, which is uniquely characterised amongst functions $\sT(U) \rightarrow \sT(V)$ by the property that $\iota \circ \partial = \psi(\partial) \circ \iota$ for all $\partial \in \sT(U)$, where $\iota \colon \OO(U) \hookrightarrow \OO(V)$ denotes the inclusion map (Lemma \ref{extendderivations}).

	For the action of $\sT(U)$ on $\sA(U)$, note that $\sT(U)$ acts via derivations on $\sA(U)$ because the multiplication map $\sA \otimes_{\OO_X} \sA \rightarrow \sA$ is a morphism in $\VectCon^G(X)$, and write $\phi \colon \sT(U) \rightarrow \Der_k(\sA(U))$ for the induced map. Then, noting that $F(k) = \OO_X$, applying $F$ to the inclusion $k \hookrightarrow \OO(H)$ we see that $\OO_X \hookrightarrow \sA$ is a morphism in $\VectCon^G(X)$ and thus $\iota \colon \OO(U) \hookrightarrow \sA(U)$ is $\sD(U)$-linear. In particular, $\iota \circ \partial = \phi(\partial) \circ \iota$ for all $\partial \in \sT(U)$, and hence $\phi = \psi$.
\end{proof}

As an example application of Proposition \ref{prop:finitecomesfromcovering}, we have the following characterisation of finite objects in the category $\Rep^{\fd}_k(G)$ of finite-dimensional $k$-representations of $G$.

\begin{cor}\label{cor:inflationfinite}
	Suppose $G$ is a group and $k$ is a characteristic $0$ field. Then $V \in \Rep^{\fd}_k(G)$ is finite if and only if $V$ is inflation-finite: inflated from a representation of a finite quotient of $G$.
\end{cor}

\begin{proof}
	Firstly, any $k$-representation of a finite group $H$ is finite because there are only finitely many indecomposable representations of $H$ (as $k$ has characteristic $0$). In particular, as inflation preserves indecomposability, any inflation-finite representation is finite. Conversely, suppose that $V \in \Rep^{\fd}_k(G)$ is finite. Note that the base change $V_{\overline{k}}$ is also finite, and to show $V$ is inflation-finite it is sufficient to show that the same is true of $V_{\overline{k}}$. Taking $X = \Spec(\overline{k})$ with trivial action of $G$, taking global sections identifies $\VectCon^G(X)$ with $\Rep_{\overline{k}}^{\fd}(G)$ and we may apply Proposition \ref{prop:finitecomesfromcovering} to see that $V_{\overline{k}}$ is a sub-representation of $A = \OO(Z)$, for some finite \'{e}tale $G$-equivariant covering $Z \rightarrow X$. The action of $G$ on $A$ factors through the group $\Aut_k(A)$, which is finite because $A/k$ is \'{e}tale, and thus we see that $A$ and hence $V_{\overline{k}}$ are inflation-finite.
\end{proof}

\section{\texorpdfstring{$\sD$-Modules and Galois Coverings}{D-Modules and Galois Coverings}}\label{sect:DmodGalExt}

In this section we state our main result concerning Galois coverings and $\sD$-modules, Theorem \ref{mainthm1}. We specialise our geometric framework of the previous sections further and suppose that $\charfield(k) = 0$, $G$, $H$ and $N$ are groups, $N \triangleleft H$ a normal subgroup of $H$,
\begin{enumerate}
	\item[\textbf{(A)}]
	$f \colon X \rightarrow Y$ is a Galois covering of smooth schemes over $k$ with Galois group $N$,
	\item[\textbf{(B)}] 
	$f \colon X \rightarrow Y$ is a Galois covering of smooth rigid spaces over $k$ with Galois group $N$,
\end{enumerate}
and $X$ and $Y$ have an action of the abstract group $G \times H$ which extends the action of $N$ and for which $f \colon X \rightarrow Y$ is equivariant.

\subsection*{The Functor $\Hom_{k[N]}(-, f_* \OO_X)$} In this section we consider the contravariant functor
\[
	\Hom_{k[N]}(-, f_* \OO_X) : \Mod_{k[H]}^{\fd} \rightarrow \VectCon^{G \times H / N}(Y),
\]
defined as follows. The sheaf
\[
\Hom_{k}(V, f_* \OO_X) \colon U \mapsto \Hom_{k}(V, f_* \OO_X(U))
\]
on $Y$ has a natural structure as a $(G \times H)\text{-}\sD_Y$-module, where the $\sD_Y$-module structure is defined, for $U \in \sB_Y$, $\partial \in \sT_Y(U)$, and local section $\phi \in \Hom_{k}(V, f_* \OO_X(U))$, by
\[
	\partial \cdot \phi \coloneqq \partial \circ \phi,
\]
where $\partial$ acts on $f_* \OO_X(U)$ through the $\sD_Y$-module structure on $f_* \OO_X$ as described in Section \ref{section:inversedirectimage}. The equivariant structure on $\Hom_{k}(V, f_* \OO_X)$ is defined by viewing $V$ as a $k[G \times H]
$-module with trivial action of $G$, and setting, for any admissible open subset $U \subset Y$ and $g \in G \times H$,
\[
	g^{\Hom_{k}(V, f_* \OO_X)}_U(\phi) \coloneqq g^{f_*\OO_X}_U \circ \phi \circ g^{-1},
\]
using the natural structure of $f_*\OO_X$ as a $G \times H$-equivariant sheaf. Because $N$ is normal in $G \times H$, this induces the structure of a $G \times H / N$-equivariant sheaf on
\[
\Hom_{k[N]}(V, f_* \OO_X) \colon U \mapsto \Hom_{k[N]}(V, f_* \OO_X(U)),
\]
which is further a $(G \times H/N)\text{-}\sD_Y$-module, noting the the $\sD_Y$-action on $\Hom_{k}(V, f_* \OO_X)$ restricts to an action of $\sD_Y$ on $\Hom_{k[N]}(V, f_* \OO_X)$ by Lemma \ref{dercommuteG}.

\begin{thm}\label{mainthm1}
	Suppose that $G, H, N$, and $f \colon X \rightarrow Y$ are as described above. Then:
	\begin{enumerate}
		\item For $V \in \Mod_{k[H]}^{\fd}$, there is a natural isomorphism
		\[
		\Hom_{k[N]}(V, f_* \OO_X) \cong (\OO_X \otimes_k V^* )^N,
		\]
		\item The functor $\Hom_{k[N]}(-, f_* \OO_X)$ is exact, monoidal and faithful, and commutes with duals, symmetric powers, exterior powers, and determinants.
		\item For $V \in \Mod_{k[H]}^{\fd}$,
		\[
			\rank_Y (\Hom_{k[N]}(V, f_* \OO_X)) = \dim_k(V).
		\]
		\item If $H=N$, then $\Hom_{k[H]}(V, f_* \OO_X) \in \VectCon^{G}(Y)_{\fin}$ for any $V  \in \Mod_{k[H]}^{\fd}$.
	\end{enumerate}
	\noindent Suppose additionally that $c_X(X)^G = k$. Then:
	\begin{enumerate}
		\item[(5)] The functor $\Hom_{k[N]}(-, f_* \OO_X)$ is full,
		\item[(6)] For any $\sM \in \VectCon^{G \times H / N}(Y)$,
		\[
			\Hom_{(G \times H)\text{-}\sD_X}(\OO_X, f^{*}\sM) \leq \rank_Y(\sM),
		\]
		and $\sM$ is in the essential image of $\Hom_{k[N]}(V, f_* \OO_X)$ if and only if this is an equality.
	\item[(7)] The essential image of $\Hom_{k[N]}(V, f_* \OO_X)$ is closed under sub-quotients.
	\end{enumerate}
\end{thm}

\begin{proof}
	First point (1), let $U \in \sB_Y$ and note that because $V$ is finite dimensional there is a $\OO_Y(U)[N]$-module isomorphism
	\[
		\OO_X(f^{-1}(U)) \otimes_k V^* \xrightarrow{\sim} \Hom_{k}(V, \OO_X(f^{-1}(U))), \qquad x \otimes \lambda \mapsto \lambda(-)x.
	\]
	When we take $N$-invariants, we obtain an isomorphism of $\OO_Y(U)$-modules, 
	\[
	(\OO_X(f^{-1}(U)) \otimes_k V^*)^N \xrightarrow{\sim} \Hom_{k[N]}(V, \OO_X(f^{-1}(U))),
	\]
	which is furthermore $\sD_Y(U)$-linear as for any $\partial \in \sT_Y(U)$,
	\[
		\partial \cdot (x \otimes \phi) = \partial(x) \otimes \phi \mapsto \phi(-) \partial(x) = \partial \cdot (\phi(-)x).
	\]
	The $G \times H/N$-linearity follows from the commutativity of
\[\begin{tikzcd}
	{(\OO_X(f^{-1}(U)) \otimes_k V^*)^N} & {\Hom_{k[N]}(V, \OO_X(f^{-1}(U)))} \\
	{(\OO_X(f^{-1}(g(U))) \otimes_k V^*)^N} & {\Hom_{k[N]}(V, \OO_X(f^{-1}(g(U))))}
	\arrow[from=1-1, to=1-2]
	\arrow["{g^{\OO_X}_{f^{-1}(U)} \otimes g^*}", from=1-1, to=2-1]
	\arrow["{g^{\OO_X}_{f^{-1}(U)} \circ - \circ g^{-1}}", from=1-2, to=2-2]
	\arrow[from=2-1, to=2-2]
\end{tikzcd}\]
and these isomorphisms glue to give point (1). The functor $(\OO_X \otimes_k - )^N$ is the composition
	\[
	 	\Mod_{k[H]}^{\fd} \rightarrow \VectCon^{G \times H}(X) \xrightarrow{\sim} \VectCon^{G \times H / N}(Y),
	\]
	of the functors of Theorem \ref{FunctorkGGDX} and Proposition \ref{FunctorEquivalence}. As the Galois equivalence preserves rank this gives point (3). Both of these functors are exact, monoidal and faithful by Theorem \ref{FunctorkGGDX}(1), and when $c_X(X)^G = k$ both are full with the essential image is closed under sub-quotients. This gives statements (5), (7), and the first part of point (2). Statement (6) is simply a restatement of Theorem \ref{FunctorkGGDX}(3), and statement (4) follows directly from Lemma \ref{prop:galoisisfinite}. For the remaining statements of point (2), if $X$ is in either $\Mod_{k[H]}^{\fd}$ or $\VectCon^{G \times H / N}(Y)$, then the symmetric power $S^n(X)$ is the coequaliser of the $n!$ symmetries $X^{\otimes n} \rightarrow X^{\otimes n}$, and the $n$th exterior power can be described as the image of the antisymmetrisation map
	\[
	\sum_{\sigma \in S_n} (-1)^{\text{sgn}(\sigma)} \sigma \colon X^{\otimes n} \rightarrow X^{\otimes n},
	\]
	both of which are preserved by $\Hom_{k[N]}(V, f_* \OO_X)$. In particular, for $X$ of constant rank $n$, the determinant $\det(X) = \wedge^n X$ of $X$ is preserved. Duals are also preserved, as the functor $\Hom_{k[N]}(V, f_* \OO_X)$ is a tensor functor between rigid tensor categories.
\end{proof}

\subsection{The Image of the Regular Representation}\label{sect:imregrep} Suppose now that the quotient map $H \rightarrow H / N$ has a section $s \colon H / N \rightarrow H$. We write $H_s = s(H/N)$ for the corresponding subgroup of $H$ with $N \cap H_s = 1$ and $N H_s = H$, or in other words $H = N \rtimes H_s$.

With respect to this section, we can view $k[N]$ as a left $k[H]$ module where $n \in N$ acts by left multiplication on $k[N]$ and $h \in H_s$ acts on $n \in N$ through the group homomorphism $H_s \rightarrow \Aut_{\Alg_k}(k[N])$ defined by $h * n = hnh^{-1}$. One can check this defines an action of the semi-direct product $H = N \rtimes H_s$.

Similarly, with respect to this section, we view the direct image $f_* \OO_X$ as a $(G \times H/N)\text{-}\sD_Y$-module, through the natural $(G \times H)\text{-}\sD_Y$-module structure on $f_* \OO_X$ and composing along the section $s \colon H / N \rightarrow H$.

\begin{eg}\label{eg:HNnosection}
In the important special case that $H = N$, the trivial section is the unique choice for $s$, for which the action of $k[N]$ on itself is by left multiplication and $f_* \OO_X$ is equipped with its natural $G\text{-}\sD_Y$-module structure.
\end{eg}

We now show that for a given section $s$ as above, these two structures correspond under the functor $\Hom_{k[N]}(-, f_* \OO_X)$ in the sense that $k[N]$ (with this action of $k[H]$, which is dependent on $s$) is sent to the $(G \times H/N)\text{-}\sD_Y$-module $f_* \OO_X$ (the $G \times H/N$-equivariant structure being also dependent on $s$). In the following we also make use of the isomorphism
\[
	k[N]^{H/N, \op} \xrightarrow{\sim} \End_{k[H]}(k[N]),
\]
which is obtained as the restriction of the isomorphism
\[
k[N]^{\op} \xrightarrow{\sim} \End_{k[N]}(k[N])
\]
defined by the action of $k[N]$ on itself by right multiplication. Here $H/N$ is considered to act on $k[N]$ through the section $s \colon H / N \rightarrow H_s$ composed with the group homomorphism $H_s \rightarrow \Aut_{\Alg_k}(k[N])$, $h * n = hnh^{-1}$, as defined above.

\begin{prop}\label{prop:regrepn}
	The isomorphism 
	\[
		\Hom_{k[N]}(k[N], f_* \OO_X) \xrightarrow{\sim} f_* \OO_X, \qquad \phi \mapsto \phi(1),
	\]
	is $(G \times H/N)\text{-}\sD_Y$-linear, and the functorially induced isomorphism (cf.\ Theorem \ref{mainthm1}(2,5))
	\[
	k[N]^{H/N} \xrightarrow{\sim} \End_{k[H]}(k[N])^{\op} \xrightarrow{\sim} \End_{(G\times H/N)\text{-}\sD_Y}(f_* \OO_X)
	\]
	coincides with the natural action of $k[N]$ on $f_*\OO_X$.
\end{prop}

\begin{proof}
	The given isomorphism is easily seen to be $G\text{-}\sD_Y$-linear. To see that it is $H/N$-linear, let $hN \in H/N$, with $h = s(h) n$ the unique representation of $h$ in $N \rtimes H_s$. By definition $hN$ acts on a local section $\phi$ over $U \in \sB_Y$ by
	\begin{align*}
		(h*\phi)(-) &= (s(h)n)^{f_* \OO_X}_U \phi(n^{-1} s(h)^{-1} (-) s(h)), \\
		&= s(h)^{f_* \OO_X}_U \phi(s(h)^{-1} (-) s(h)), \\
		&\mapsto s(h)^{f_* \OO_X}_U(\phi(1)),
	\end{align*}
	which is exactly the action of $hN$ on $f_* \OO_X$. Finally, the functorially induced action of $n \in k[N]$ on $\phi$ maps $\phi \mapsto \phi \circ r_n$, which when evaluated at $1$ is $\phi(n) = n^{f_* \OO_X}_U(\phi(1))$.
\end{proof}

\begin{eg}\label{eg:HNequalisom}
	When $H = N$ as in Example \ref{eg:HNnosection}, Proposition \ref{prop:regrepn} becomes the statement that
	\[
		\Hom_{k[H]}(k[H], f_* \OO_X) \xrightarrow{\sim} f_* \OO_X, \qquad \phi \mapsto \phi(1),
	\]
	is a $G\text{-}\sD_Y$-linear isomorphism, where $k[H]$ is viewed as a left $k[H]$-module, and the functorially induced isomorphism
	\[
		k[H] \xrightarrow{\sim} \End_{G\text{-}\sD_Y}(f_* \OO_X)
	\]
	agrees with the natural action of $k[H]$ on $f_* \OO_X$. In particular, when $H = N$, by Theorem \ref{mainthm1}(7) the essential image of the functor
	\[
		\Hom_{k[H]}(-, f_* \OO_X) : \Mod_{k[H]}^{\fd} \rightarrow \VectCon^{G}(Y)
	\]
	can be described as the full subcategory of objects which admit a $G\text{-}\sD_Y$-linear embedding into $(f_*\OO_X)^{\oplus n}$ for some $n \geq 1$.
\end{eg}

\subsection{Decomposing $f_* \OO_X$} For each $\rho \in \Irr(N)$ with character $\chi_{\rho}$, we let
\[
	e_{\rho} = \frac{\dim_k(\rho)}{d_{\rho} \cdot |N|} \sum_{n \in N} \chi_{\rho}(n^{-1}) n \in k[N],
\]
be the corresponding central primitive idempotent, where 
\[
d_{\rho} \coloneqq \dim_k \End_{k[N]}(\rho).
\]
Under the assumptions of this section, we may forget $H$ and apply the results of this section to the new triple $(G',H',N') \coloneqq (G,N,N)$. Considering Theorem \ref{mainthm1}, Proposition \ref{prop:regrepn} and Example \ref{eg:HNequalisom}, we have the following explicit decomposition of $f_* \OO_X$ as an object of $\VectCon^{G}(Y)$, due to the fact that $f_* \OO_X$ is semisimple with endomorphism ring $k[N]$ whenever $c_X(X)^G = k$.

\begin{cor}\label{maincor2}
	There is a decomposition
	\[
			f_* \OO_X = \bigoplus_{\rho \in \Irr(N)} e_{\rho} \cdot f_* \OO_X.
	\]
	in $\VectCon^G(Y)$. When $c_X(X)^G = k$:
	\begin{enumerate}
		\item $f_* \OO_X$ is a semisimple object of $\VectCon^G(Y)$ and this decomposition is isotypic.
		\item For any $\rho \in \Irr(N)$, there is a one-to-one correspondence between decompositions of $e_{\rho}$ as a sum of primitive orthogonal idempotents in $e_{\rho} \cdot k[N]$ and decompositions of $e_{\rho} \cdot f_* \OO_X$ as a direct sum of irreducible objects of $\VectCon^G(Y)$. Here we send,
		\[
			e_{\rho} = \sum_{i = 1}^{\dim(\rho) / d_{\rho}} e_i \qquad \longmapsto \qquad e_{\rho} \cdot f_* \OO_X = \bigoplus_{i = 1}^{\dim(\rho) / d_{\rho}} e_i \cdot f_* \OO_X.
		\]
		\item For any irreducible $\OO_X$-coherent $G \text{-}\sD_Y$-submodule $\sM$ of $e_\rho \cdot f_* \OO_X$, $\sM = e \cdot f_* \OO_X$ for some primitive idempotent $e$ of $e_{\rho} \cdot k[N]$, and the natural $k$-algebra homomorphisms
		\begin{align*}
		e \cdot k[N] \cdot e &\rightarrow \End_{G \text{-} \sD_Y}(e \cdot f_* \OO_X), \\
		e_{\rho} \cdot k[N] &\rightarrow \End_{G \text{-} \sD_Y}(e_{\rho} \cdot f_* \OO_X),
		\end{align*}
		are isomorphisms.
	\end{enumerate}
\end{cor}

\begin{remark}
When $G$ is trivial and $H = N$, the setup of this section is that $f \colon X \rightarrow Y$ is a finite \'{e}tale Galois morphism of smooth spaces over $k$, and the assumption that $c_X(X) = c_X(X)^G=k$ is simply that $X$ is geometrically connected by Corollary \ref{cor:cXuequalsk1} and Corollary \ref{cor:cXuequalsk2}.
\end{remark}

\begin{remark}\label{rmk:geomconnnec}
	The condition that $c_X(X)^G = k$ is in general necessary in Theorem \ref{FunctorkGGDX}, Theorem \ref{mainthm1} and Corollary \ref{maincor2}. For example, suppose that $G$ is trivial and $H = N$, and consider $L$ a finite Galois extension of $k$ of degree $n$ with Galois group $N$. Then the corresponding extension (which also makes sense in case (B)),
	\[
		f \colon X = \Spec(L) \rightarrow \Spec(k) = Y,
	\]
	is a Galois extension with Galois group $N$, with $X$ connected but not geometrically connected whenever $n > 1$, as then $c_X(X) = L > k$. Because $L/k$ is separable, $\Der_k(L) = 0$, $\sD_X(X) =L$ and $\sD_Y(Y) = k$. The dual functor $(\OO_X \otimes_k -)^N$ (cf.\ Theorem \ref{mainthm1}(1)) becomes the composition
\[\begin{tikzcd}
	{\Mod_{k[N]}^{\fd}} & {\Mod_{L \rtimes N}^{\fd}} \\
	& {\Vect_k^{\fd}}
	\arrow["{L \otimes_k -}", from=1-1, to=1-2]
	\arrow["{(\OO_X \otimes_k -)^N}"', from=1-1, to=2-2]
	\arrow["{(-)^N}"', shift right, from=1-2, to=2-2]
	\arrow["{f^* = L \otimes_k -}"', shift right, from=2-2, to=1-2]
\end{tikzcd}\]
	where the vertical arrows are the equivalence of Galois descent. Unlike Example \ref{eg:HNequalisom}, the map
	\[
		k[N] \rightarrow \End_{\sD_Y}(f_* \OO_X) = \End_k L
	\]
	will never be an isomorphism when $n > 1$, as the left-hand side has $k$-dimension $n$ and the right-hand side $k$-dimension $n^2$. The essential image of $(\OO_X \otimes_k -)^N$ also need not be closed under sub-objects, as this functor need not preserve irreducibility: when $N$ is non-abelian $k[N]$ has length strictly less than $n$ as a $k[N]$-module, but $L$ has length $n$ as a $k$-vector space.
	\end{remark}

\subsection{Compatibility with Intermediate Coverings}

We are interested in the compatibility of $\Hom_{k[N]}(-, f_* \OO_X)$ when passing to Galois sub-covers. Suppose that $N_0 \triangleleft H$ with $N_0$ contained in $N$, and let $Z$ be the intermediate covering of $Y$ defined by $N_0$, with Galois group $N / N_0$. We have Galois coverings
\[
X \rightarrow Z \xrightarrow{\phi} Y	
\]
with Galois groups $N_0$ and $N / N_0$ respectively. Note that because the action of $G$ commutes with the action of $N$, there is an induced action of $G \times H / N_0$ on $Z$ for which $\phi \colon Z \rightarrow Y$ is $G \times H / N_0$-equivariant, and the Galois action of $N / N_0$ on $Z$ is through this action and the inclusion $N / N_0 \hookrightarrow H / N_0$. We have an inflation functor
\[
	\iota \colon \Mod^{\fd}_{k[H / N_0]} \rightarrow \Mod^{\fd}_{k[H]}.
\]
The next lemma relates the functors $\Hom_{k[N]}(-, f_* \OO_X)$ and $\Hom_{k[N/N_0]}(-, \phi_* \OO_Z)$.

\begin{lemma}\label{compatiblelemma}
Suppose that $N_0 \triangleleft H$ is a normal subgroup of $H$ which is contained in $N$, and $\phi \colon Z \rightarrow Y$ is the intermediate covering defined by $N_0$, with Galois group $N/N_0$. Then there is a canonical identification of $(G \times H/N)\text{-}\sD_Y$-modules
\[
	\Hom_{k[N]}(\iota(W), f_*\OO_X) \xrightarrow{\sim} \Hom_{k[N/N_0]}(W, \phi_* \OO_Z)
\]
which is natural in $W \in \Mod^{\fd}_{k[H / N_0]}$.
\end{lemma}
\begin{proof}
For any admissible open subset $U$ of $Y$, $\phi_* \OO_Z(U) = \OO_X(f^{-1}(U))^{N_0}$, and therefore
\[
	\Hom_{k[N/N_0]}(W, \phi_* \OO_Z(U)) = \Hom_{k[N/N_0]}(W, \OO_X(f^{-1}(U))^{N_0}) = \Hom_{k[N]}(\iota(W), \OO_X(f^{-1}(U))).
\]
It is direct to see this defines a $(G \times H/N) \text{-}\sD_Y$-linear natural isomorphism.
\end{proof}

\subsection{Compatibility with Connected Components}

We are also interested in the compatibility of $\Hom_{k[N]}(-, f_* \OO_X)$ when passing to a connected component of a Galois covering.

\begin{lemma}\label{lem:conncompgaloisDmodule}
	Suppose that $f \colon X \rightarrow Y$ is as described at the start of Section \ref{sect:DmodGalExt}, with $H = N$ and $Y$ connected. Let $X_0$ be a connected component of $X$ stabilised by $G$, and write $f_0 \colon X_0 \rightarrow Y$ for the induced Galois extension with Galois group $H_0 \coloneqq \Stab_H(X_0)$ (cf.\ Lemma \ref{lemma:conncompgaloiscovering}). Then the diagram
\[\begin{tikzcd}
	& {\Mod_{k[H]}^{\fd}} \\
	{\Mod_{k[H_0]}^{\fd}} & {} & {\VectCon^G(Y)}
	\arrow[from=1-2, to=2-1]
	\arrow[from=1-2, to=2-3]
	\arrow[from=2-1, to=2-3]
\end{tikzcd}
\]
commutes up to natural isomorphism. Explicitly, if $e_0 \in \OO_X(X)$ is the idempotent corresponding to $X_0$,
	\[
		\Hom_{k[H]}(V, f_* \OO_X) \rightarrow \Hom_{k[H_0]}(V|_{H_0}, f_{0,*} \OO_{X_0}), \qquad \phi \mapsto e_0 \cdot \phi(-),
	\]
	is an isomorphism of $G \text{-} \sD_Y$-modules, natural in $V$.
	\end{lemma}
	
	\begin{proof}
		Because $G$ preserves the connected component $X_0$, this is an isomorphism of $G$-equivariant sheaves. To see that this morphism is $\sD_Y$-linear, let $U \in \sB_Y$, $A \coloneqq \OO_Y(U)$, $B \coloneqq \OO_X(f^{-1}(U))$, $B_0 \coloneqq \OO_X(f_0^{-1}(U))$, so $B_0 \subset B$, and $A \hookrightarrow B_0$, $A \hookrightarrow B$ are Galois extensions of commutative $k$-algebras with Galois groups $H_0$, $H$ respectively. It is sufficient for us to show that 
		\[
			\Hom_{k[H]}(V, B) \rightarrow \Hom_{k[H_0]}(V|_{H_0}, B_0), \qquad \phi \mapsto \pi_0(\phi(-)),	
		\]
		is an isomorphism of $\sD(A)$-modules, where $\pi_0 \colon B \rightarrow B_0$ is the projection. First, we show that this is an isomorphism. There is an isomorphism of $k[H]$-modules,
		\[
			B \rightarrow \Ind_{H_0}^H B_0 = \Hom_{k[H_0]}(k[H], B_0), \qquad b \mapsto \psi_b, \quad \psi_b(h) \coloneqq \pi_0(hb)
		\]
		for $h \in H$. Here $\Ind_{H_0}^H B_0$ is a left $k[H]$-module via $(g \cdot \sigma)(h) = \sigma(hg)$. The inverse is given by choosing left coset representatives $h_1, ... , h_k$ of $H / H_0$ and mapping
		\[
			\sigma \mapsto \sum_{i = 1}^k h_i \sigma(h_i^{-1}) \in B,
		\]
		and is independent of the choice of coset representatives. 
		Frobenius reciprocity gives us isomorphisms,
		\[
			\Hom_{k[H]}(V,\Ind_{H_0}^H B_0) \xleftrightarrow{\sim} \Hom_{k[H_0]}(V|_{H_0}, B_0),
		\]
		explicitly given by
		\[
			[f \colon V \rightarrow \Ind_{H_0}^H B_0] \mapsto \Phi(f), \qquad \Phi(f)(v) \coloneqq f(v)(1),
		\]
		and
		\[
			[\lambda \coloneqq V|_{H} \rightarrow B_0] \mapsto \Pi(\lambda), \qquad \Pi(\lambda)(v)(k) \coloneqq \lambda(kv).
		\]
		Then our map of interest is simply the composition of the induced isomorphisms,
		\[
			\Hom_{k[H]}(V,B) \xrightarrow{\sim} \Hom_{k[H]}(V,\Ind_{H_0}^H B_0) \xrightarrow{\sim} \Hom_{k[H_0]}(V|_{H_0}, B_0).
		\]
		We therefore are left with showing that the isomorphism is $\sD(A)$-linear. First note that for $e_0$ the idempotent of $B$ defining $B_0$, any $\partial \in \Der_k(B)$ satisfies $\partial(e_0) = 0$, and therefore, $\partial \mapsto \partial_{B_0}$ is a well-defined $B_0$-linear restriction map. There are $A$-linear maps $\psi \colon \Der_k(A) \rightarrow \Der_k(B)$, $\psi_0 \colon \Der_k(A) \rightarrow \Der_k(B_0)$ provided by Lemma \ref{extendderivations}, and by the uniqueness part of Lemma \ref{extendderivations} we see that the composition
		\[
			\Der_k(A) \xrightarrow{\psi} \Der_k(B) \rightarrow \Der_k(B_0)
		\] 
		is equal to $\psi_0$. Therefore, the projection $\pi_0 \colon B \rightarrow B_0$ is $\sD(A)$-linear. Finally, we see that the isomorphism is $\sD(A)$-linear, as given $x \in \sD(A)$ and $f \in \Hom_{k[H]}(V, B)$,
		\[
			\pi_0(x \cdot f(-)) = x \cdot \pi_0(f(-)). \qedhere
		\]
	\end{proof}

\subsection{Abelian Galois Coverings}

In the section, we continue with the assumptions from the start of Section \ref{sect:DmodGalExt}.

\begin{defn}
	We define $\PicCon^G(Y)$ to be the set of isomorphism classes of $G$-equivariant line bundles with connection on $Y$ (i.e.\ rank $1$ elements of $\VectCon^G(Y)$).
\end{defn}

In the case that $G$ is trivial, we write $\PicCon(Y)$ for $\PicCon^G(Y)$, the set of isomorphism classes of line bundles with connection on $Y$. The tensor product (as described in Section \ref{sect:generaleqDmodules}) induces an abelian group structure on $\PicCon^G(Y)$ for which the natural forgetful map
\[
	\PicCon^G(Y) \rightarrow \Pic(Y)
\]
is a group homomorphism.

We can now state a consequence of Corollary \ref{maincor2}. Write $\widehat{N}$ for the group of $k$-valued characters of $N$. For such a character, let $e_{\chi}$ denote the corresponding primitive central idempotent of $k[N]$ as defined immediately before Corollary \ref{maincor2}, and set
\[
 \sL_{\chi} \coloneqq e_{\chi} \cdot f_* \OO_X \in \VectCon^G(Y).
\]
We let $e$ denote the exponent of $\widehat{N}$.

\begin{cor}\label{mainthmabelian}
 The map
\[
	\widehat{N} \rightarrow \PicCon^G(Y)[e], \qquad \chi \mapsto \Hom_{k[N]}(\chi, f_* \OO_X),
\]
is a group homomorphism. If $c_X(X)^G = k$, this is injective and for any $\chi \in \widehat{N}$ the natural map
\begin{equation}\label{eqn:natmap}
\Hom_{k[N]}(\chi, f_* \OO_X) \rightarrow \sL_{\chi}, \qquad f \mapsto f(1),
\end{equation}
is an isomorphism in $\VectCon^G(Y)$. The map (\ref{eqn:natmap}) is also an isomorphism when $N$ is abelian and $k$ is splits $N$, in which case there is a decomposition
\[
f_* \OO_X = \bigoplus_{\chi \in \widehat{N}} \sL_{\chi}
\]
in $\VectCon^G(Y)$.
\end{cor}

\begin{proof}
	This map is the restriction to the set of isomorphism classes of $1$-dimensional representations of the functor
	\[
		\Hom_{k[N]}(-, f_* \OO_X) \colon \Mod_{k[N]}^{\fd} \rightarrow \VectCon^{G}(Y)
	\]
	of Theorem \ref{mainthm1}, with $H$ taken to be $N$. Therefore this map is a well-defined group homomorphism because this functor is monoidal by Theorem \ref{mainthm1}(2) and sends $1$-dimensional representations to rank $1$ objects by Theorem \ref{mainthm1}(3). The map (\ref{eqn:natmap}) is always injective and therefore an isomorphism when $c_X(X)^G = k$ because in this case $e_{\chi} \cdot f_* \OO_X$ is simple by Corollary \ref{maincor2}(2) and the fact that $\chi$ is $1$-dimensional. The injectivity when $c_X(X)^G = k$ follows from the isotypic decomposition of Corollary \ref{maincor2}(1). Suppose now that $N$ is abelian and $k$ splits $N$. The decomposition of $f_* \OO_X$ follows from the decomposition of Corollary \ref{maincor2}. By \cite[Prop.\ 2.3]{TAY}, each $\sL_{\chi}$ is a line bundle, and therefore because $\Hom_{k[N]}(\chi, f_* \OO_X)$ also has rank $1$, the inclusion (\ref{eqn:natmap}) is also an isomorphism in this case by the same argument as in the proof of Theorem \ref{FunctorkGGDX}(3).
\end{proof}

\section{The Drinfeld Tower}\label{sect:final}

We are now in a position to give the main applications of the general results we have established to the Drinfeld tower. We first recall some basic facts about Drinfeld spaces (Section \ref{sect:thedrinfeldtower}) and their geometrically connected components (Section \ref{sect:geomconncomp}), and use the general theory of Section \ref{sect:DmodGalExt} to give a more conceptual proof of our earlier result concerning line bundles on the first Drinfeld covering (Section \ref{sect:linebundles}). Then in the remainder of this section we prove Theorem A, Theorem B and Theorem C (Section \ref{sect:equivariantDmodulesonDrinfeldSpaces} onwards).

\subsection{The Drinfeld Tower}\label{sect:thedrinfeldtower}

From now on, suppose that $K$ contains $L \coloneqq \breve{F}$, the completion of the maximal unramified extension of $F$, and that $n \geq 1$. We write $\Omega$ for the $(n-1)$-dimensional Drinfeld symmetric space over $K$. Let $D$ be the division algebra over $F$ of invariant $1/n$ with ring of integers $\OO_D$, let $\Pi$ denote a uniformiser of $\OO_D$, and write $\Nrd \colon D^\times \rightarrow F^\times$ for the reduced norm of $D$. The Drinfeld tower is a system of rigid analytic spaces over $K$,
\[
\sM \leftarrow \sM_1 \leftarrow \sM_2 \leftarrow \cdots,	
\]
for which each space has an action of $\GL_{n}(F) \times D^\times$ such that the transition morphisms are equivariant. Background material on these spaces is contained in \cite{DRI, BC, RZ}. The set of connected components of the space $\sM$ is canonically identified with $\bZ$, and under this identification $(g, \delta) \in \GL_{n}(F) \times D^\times$ acts on this set of connected components by addition by $\nu(\det(g)\Nrd(\delta^{-1}))$ \cite[Thm.\ 0.20]{BZ}. In particular, the connected components of $\sM$ are permuted simply transitively by the action of $\Pi$. Let $\sN$ be a connected component of $\sM$. The Grothendieck-Messing period morphism
\[
\pi_{\text{GM}} \colon \sM \rightarrow \Omega
\]
is an \'{e}tale $\GL_n(F) \times D^\times$-equivariant morphism, where $\Omega$ is considered with the trivial action of $D^\times$, which induces a $\GL_n(F)$-equivariant isomorphism
\[
\pi_{\text{GM}}\colon \sM / H \xrightarrow{\sim} \Omega,
\]
where $H \coloneqq D^\times / \OO_D^\times$. In particular, the composition
\[
\sN \hookrightarrow \sM \xrightarrow{p} \sM / H \rightarrow \Omega
\]
is a $G^0$-equivariant isomorphism, where
\[
G^0 \coloneqq \{g \in \GL_{n}(F) \mid \nu(\det(g)) = 0\},
\]
which we will use to identify $\sN$ with $\Omega$.

Considering the preimage $\sN_m$ of $\sN$ in each covering space $(\sM_n)_{n \geq 1}$, we obtain a sub-tower,
\[
	\sN \leftarrow \sN_{1} \leftarrow \sN_{2} \leftarrow \cdots.
\]
which is stable under the action of $G^0 \times \OO_D^\times$, as this is true for $\sN$. The subgroup $1 + \Pi^m \OO_D \leq \OO_D^{\times}$ acts trivially on $\sM_{m}$, and the morphisms $\sM_{m} \rightarrow \sM$, $\sN_{m} \rightarrow \sN$ are Galois with Galois group $\OO_D^{\times} / (1 + \Pi^m \OO_D)$ \cite[Thm.\ 2.2]{KOH}. 

\subsection{Continuous Action on The Drinfeld Tower}\label{sect:ctsactionDrinfeldTower}

As a consequence of the general results established in Appendix \ref{sect:generalities} on continuous actions and finite \'{e}tale coverings (to which we direct the reader for the definition of a continuous action) we have the following.

\begin{cor}\label{cor:ctsactionG0N}
	The action of $G^0$ on each of the spaces $\sN, \sN_1, \sN_2, ...$ is continuous. 
\end{cor}

\begin{proof}
	The natural action of $\GL_n(\OO_F)$ on $\Omega$ is continuous by \cite[Prop.\ 3.1.12(b)]{ARD} and \cite[Lem.\ 3.1.9(b)]{ARD}, and thus the action of $\GL_n(F)$ on $\Omega$ is continuous by \cite[Lem.\ 3.1.9(c)]{ARD}. In particular, the restriction to the subgroup $G^0$ is continuous, and therefore by the $G^0$-equivariant isomorphism above $\sN \xrightarrow{\sim} \Omega$ the action of $G^0$ on $\sN$ is continuous. Then the result follows from Proposition \ref{prop:continuousactionlifts}.
\end{proof}

\begin{cor}\label{cor:ctsactionGLnM}
The action of $\GL_n(F)$ on each of the spaces $\sM, \sM_1, \sM_2, ...$ is continuous. 
\end{cor}

\begin{proof}
	This follows directly from Lemma \ref{lem:ctsactionconncomp} and Corollary \ref{cor:ctsactionG0N}.
\end{proof}

\subsection{Geometrically Connected Components of The Drinfeld Tower}\label{sect:geomconncomp}

Each of the spaces $(\sN_m)_{m \geq 1}$ is connected over $L$ \cite[Thm.\ 2.5]{KOH}, but none of these spaces are geometrically connected. The connected components of the covering spaces $\sN_m$ and $\sM_m$ over $\bC_p$ have been described by Boutot and Zink \cite[Thm.\ 0.20]{BZ} using global methods and $p$-adic uniformisation of Shimura curves.

In this section we describe how the cofinal system of spaces $(\sM_{nm})_{m \geq 1}$ of the covering spaces $(\sM_{m})_{m \geq 1}$ breaks up into geometrically connected components over a finite extension of $L$. The result we obtain is not quite as strong as that of Boutot and Zink, as it only describes the action of the subgroup $\GL_n(\OO_F) \times \OO_D^\times$, but on the other hand the proof is elementary, self-contained, and sufficient for our purposes. The proof uses the theory of the sheaf $c_X$ developed in Section \ref{sect:constantsheaf} together with a result of Kohlhaase on the maximal sub-field of the global sections of $\sN_{nm}$ \cite[Prop.\ 2.7]{KOH}.

Henceforth, we write $L_m$ for the compositum of the $m$th Lubin-Tate extension of $F$ with $L$.

\begin{thm}\label{thm:geomconncompdescDT}
Suppose that $m \geq 1$ and that $K$ contains $L_m$. Then there is an isomorphism of $\GL_n(\OO_F) \times \OO_D^\times$-sets,
\[
	\pi_0(\sN_{nm}) \xrightarrow{\sim} \OO_F^\times / (1 + \pi^m \OO_F),
\] 
where $(g, \delta) \in \GL_n(\OO_F) \times \OO_D^\times$ acts on $\OO_F^\times / (1 + \pi^m \OO_F)$ by multiplication by $\det(g) \Nrd(\delta^{-1}) \in \OO_F^\times$. These are compatible in the sense that for any $1 \leq r \leq m$, the diagram
\[\begin{tikzcd}
	{\pi_0(\sN_{nm})} & {\OO_F^\times / (1 + \pi^m \OO_F)} \\
	{\pi_0(\sN_{nr})} & {\OO_F^\times / (1 + \pi^r \OO_F)}
	\arrow[two heads, from=1-1, to=2-1]
	\arrow["\sim", from=1-1, to=1-2]
	\arrow[two heads, from=1-2, to=2-2]
	\arrow["\sim", from=2-1, to=2-2]
\end{tikzcd}\]
commutes.
\end{thm}

\begin{proof}
	Let $m \geq 1$. By Lemma \ref{baseechangecX} there is a $G^0 \times \OO_D^\times$-equivariant isomorphism of $L_m$-algebras,
	\[
		c(\sN_{nm, L}) \otimes_{L} L_m \xrightarrow{\sim} c(\sN_{{nm}, L_m}).
	\]
	The field $c(\sN_{nm, L})$ is the maximal field extension of $L$ contained in $\OO(\sN_{nm})$ by Corollary \ref{cor:maxfieldext}, and the result \cite[Prop.\ 2.7]{KOH} of Kohlhaase shows that this is a copy of $L_m$. Furthermore, it is shown that the action of $(g, \delta) \in \GL_n(\OO_F) \times \OO_D^\times$ on $L_m$ is by multiplication by $\det(g) \Nrd(\delta^{-1}) \in \OO_F^\times$ \cite[Thm.\ 2.8(ii)]{KOH}, where this element of $\OO_F^\times$ is viewed as an automorphism of $L_m$ over $L$ under the identification $\OO^\times_F / (1 + \pi^m \OO_F) \xrightarrow{\sim} \Gal(L_m / L)$ of Lubin-Tate theory. Because $L_m / L$ is Galois, we have canonical isomorphisms of $L_m$-algebras
	\[
		c(\sN_{nm, L}) \otimes_{L} L_m = L_m \otimes_L L_m \xrightarrow{\sim} \prod_{\sigma \in \Gal(L_m / L)} L_m, \qquad x \otimes y \mapsto (\sigma(x)y)_{\sigma}.
	\]
	Composing these isomorphisms, we have a $\GL_n(\OO_F) \times \OO_D^\times$-equivariant $L_m$-algebra isomorphism,
	\[
		c(\sN_{nm, L_m}) \xrightarrow{\sim} \prod_{\OO^\times_F / (1 + \pi^m \OO_F)} L_m,
	\]
	and thus by Lemma \ref{lem:finiteextension} a $\GL_n(\OO_F) \times \OO_D^\times$-equivariant bijection,
	\[
		\pi_0(\sN_{nm, L_m}) \xrightarrow{\sim} \OO^\times_F / (1 + \pi^m \OO_F).
	\]
	The compatibility follows from the compatibility of the isomorphisms $\OO^\times_F / (1 + \pi^m \OO_F) \xrightarrow{\sim} \Gal(L_m / L)$. Finally, because $\dim_{L_m} c(\sN_{nm, L_m}) = |\pi_0(\sN_{nm, L_m})|$, each connected component is geometrically connected by Corollary \ref{cor:cXuequalsk2}. We obtain the result for general complete extensions $K$ of $L_m$ by Remark \ref{rem:geomconnsubtle}.
\end{proof}

The following is the key result which allows us to apply Theorem \ref{mainthm1} to the Drinfeld tower.

\begin{cor}\label{cor:drinfeldinvconstantfunctions}
For any $m \geq 1$, $c(\sN_m)^{G^0} = K$ and $c(\sM_m)^{\GL_n(F)} = K$.
\end{cor}

\begin{proof}
	Because the action of $\GL_n(F)$ on the connected components of $\sM_m$ is transitive and $G^0$ is normal in $\GL_n(F)$, the projection $\OO(\sM_m) \rightarrow \OO(\sN_m)$ induces an isomorphism
	\[
	c(\sM_m)^{\GL_n(F)} \xrightarrow{\sim} c(\sN_m)^{G^0},
	\]
	and so it sufficient to show the claim for $\sN_m$. Furthermore, if $r \geq 1$ is such that $rn \geq m$, then in light of the inclusion
	\[
		c(\sN_m)^{G^0} \hookrightarrow c(\sN_{rn})^{G^0},
	\]
	it is sufficient to show that $c(\sN_{rn})^{G^0} = K$, which follows directly from the proof of Theorem \ref{thm:geomconncompdescDT} above, as 
	\[
		c(\sN_{rn}) = c(\sN_{rn, L}) \otimes_L K = L_m \otimes_L K,
	\]
	and $G^0$ acts on $L_m \otimes_L K$ through the left-factor $L_m$ and the surjection
	\[
		G^0 \xrightarrow{\det} \OO_F^\times \twoheadrightarrow \Gal(L_m / L). \qedhere
	\]
\end{proof}

\subsection{Equivariant Line Bundles with Connection on the First Drinfeld Covering}\label{sect:linebundles}

Before proving Theorems A, B and C, we first give an application of Corollary \ref{mainthmabelian} to the study of equivariant line bundles with connection on the first Drinfeld covering.

In this section we shall assume that $n \geq 2$ and $K$ contains $L_1$. Let $\Sigma^1$ be a geometrically connected component of of $\sN_1$, and let $\Sigma^2$ be the preimage of $\Sigma^1$ in $\sN_2$, which, because $\lfloor\frac{0}{n} \rfloor = 0 = \lfloor\frac{1}{n} \rfloor$, is also geometrically connected by \cite[Prop.\ 3.1]{TAY} (which itself uses the result \cite[Thm.\ 0.20]{BZ} of Boutot and Zink). The covering,
\[
	f \colon \Sigma^2 \rightarrow \Sigma^1,
\]
is a finite \'{e}tale Galois covering of rigid spaces over $K$, with abelian Galois group
\[
	\Gamma \coloneqq \frac{1 + \Pi \OO_D}{1 + \Pi^2 \OO_D} \cong (\bF_{q^{n}}, +)
\]
of exponent $p$. Furthermore, by \cite[Prop.\ 3.1]{TAY}, the spaces $\Sigma^1, \Sigma^2$ are stable under the action of $\SL_{n}(F)$. This action commutes with the Galois action, and $f \colon \Sigma^2 \rightarrow \Sigma^1$ is $\SL_{n}(F)$-equivariant.

As a consequence of the results of Appendix \ref{sect:generalities}, we can deduce the following statement about torsion $\SL_n(F)$-equivariant line bundles with connection on $\Sigma^1$. We direct the reader to Definition \ref{def:piccts} for the definitions of the relevant groups involved.

\begin{cor}\label{congroupzero}
	$\Con^{\SL_{n}(F)}_{\cts}(\Sigma^1)[p] = 0$.
\end{cor}

\begin{proof}
Because $\Sigma^1$ is geometrically connected, we can apply Proposition \ref{AWProp} and thus the group $\Con^{\SL_{n}(F)}_{\cts}(\Sigma^1)[p]$ fits into the exact sequence,
\[
	0 \rightarrow \Hom(\SL_{n}(F), \mu_p(K)) \rightarrow \Con^{\SL_{n}(F)}_{\cts}(\Sigma^1)[p] \rightarrow (\OO(\Sigma^1)^\times / K^\times \OO(\Sigma^1)^{\times p})^{\SL_{n}(F)} .
\]
The first term $\Hom(\SL_{n}(F), \mu_p(K)) = 0$ because $\SL_{n}(F)$ is perfect. The final term, which we would like to show is zero, sits in the short exact sequence
\[
1 \rightarrow K^\times / K^{\times p} \rightarrow \OO(\Sigma^1)^{\times} / \OO(\Sigma^1)^{\times p} \rightarrow \OO(\Sigma^1)^{\times} / K^\times \OO(\Sigma^1)^{\times p} \rightarrow 1.
\]
To see that this is exact on the left, suppose that $\lambda \in K^\times$ is the $p$th power of some $x \in \OO(\Sigma^1)^\times$, $\lambda = x^p$. Then for any $U \in \sB$ and $\partial \in \sT_{\Sigma^1}(U)$, $0 = \partial(\lambda) = p x^{p-1}\partial(x)$, and therefore $\partial(x) = 0$ as $p x^{p-1}$ is a unit. By Lemma \ref{affinedesccX}, $x|_U \in c_{\Sigma^1}(U)$, and as this holds for any $U \in \sB$, $x \in c_{\Sigma^1}(\Sigma^1)$, and $c_{\Sigma^1}(\Sigma^1) = k$ by Corollary \ref{cor:cXuequalsk1} because $\Sigma^1$ is geometrically connected.

Taking the $\SL_n(F)$-invariants of this short exact sequence we obtain the exact sequence
\begin{align*}
1 &\rightarrow (K^\times / K^{\times p})^{\SL_n(F)} \rightarrow \left(\OO(\Sigma^1)^{\times} / \OO(\Sigma^1)^{\times p}\right)^{\SL_n(F)} \rightarrow (\OO(\Sigma^1)^{\times} / K^\times \OO(\Sigma^1)^{\times p})^{\SL_n(F)} \\ &\rightarrow \HH^1(\SL_n(F), K^\times / K^{\times p}).
\end{align*}
Noting that $\SL_n(F)$ acts trivially on $K$, the first non-trivial map is an isomorphism by \cite[Cor.\ 4.5]{TAY}, and $\HH^1(\SL_n(F), K^\times / K^{\times p}) = \Hom(\SL_n(F), K^\times / K^{\times p}) = 0$ because $\SL_n(F)$ is perfect.
\end{proof}

\begin{cor}\label{secondmapinjective}
The forgetful map $\PicCon^{\SL_{n}(F)}_{\cts}(\Sigma^1)[p] \rightarrow \Pic(\Sigma^1)[p]$ is injective.
\end{cor}

As an immediate consequence of Corollary \ref{secondmapinjective} and Corollary \ref{mainthmabelian}, we can deduce the main result of \cite{TAY}.

\begin{cor}[{\cite[Thm.\ 4.6]{TAY}}]\label{cor:earlierwork}
The group homomorphism
\[
	\widehat{\Gamma} \rightarrow \Pic(\Sigma^1)[p]
\]
is injective.
\end{cor}

\begin{proof}
The action of $G^0$ on $\sN$ continuous by Corollary \ref{cor:ctsactionG0N} and therefore so too is the action of the subgroup $\SL_n(F) \subset G^0$ on $\sN$. Consequently, the action of $\SL_n(F)$ on $\Sigma^1$ is continuous by Proposition \ref{prop:continuousactionlifts}, and thus the homomorphism of interest factors as the composition
	\[
		\widehat{\Gamma} \rightarrow \PicCon^{\SL_{n}(F)}_{\cts}(\Sigma^1)[p] \rightarrow \Pic(\Sigma^1)[p]
	\]
by Lemma \ref{lem:ctsimpliesfactors}. The first is injective by Corollary \ref{mainthmabelian}, as $\Sigma^2$ is geometrically connected so $c(\Sigma^2) = k$, and the second is injective by Corollary \ref{secondmapinjective}.
\end{proof}

\subsection{Equivariant Vector Bundles with Connection on Drinfeld Symmetric Spaces}\label{sect:equivariantDmodulesonDrinfeldSpaces}

From now on until the end of the paper, we will assume that $K$ contains $L$, the completion of the maximal unramified extension of $F$. In the remaining sections we will define and establish properties of the labelled functors of the diagram below, show it is commutative up to natural isomorphism, and use this to prove Theorem A, Theorem B and Theorem C. Vertical arrows below indicate the canonical forgetful maps.
\[\begin{tikzcd}[sep=4em]
	&& {\VectCon^{H \times \GL_n(F)}(\sM)} \\
	{\Rep_{\sm}^{\fd}(D^\times)} & {\VectCon^{\GL_n(F)}(\Omega)} & {\VectCon^{\GL_n(F)}(\sM)} \\
	{\Rep_{\sm}^{\fd}(\OO_D^\times)} & {\VectCon^{G^0}(\Omega)} \\
	{\Rep_{\sm}^{\fd}(\SL_1(D))} & {\VectCon(\Omega)}
	\arrow["\sim", from=1-3, to=2-2]
	\arrow[from=1-3, to=2-3]
	\arrow["{\Hom_{\OO_D^\times}^{D^\times}(-, \phi_* \OO_{\sM_{\infty}})}"{yshift = -0.4em}, from=2-1, to=1-3]
	\arrow["{\Hom_{D^\times}(-, f_* \OO_{\sM_{\infty}})}"'{yshift = -0.4em}, from=2-1, to=2-2]
	\arrow[from=2-1, to=3-1]
	\arrow["\sim", from=2-3, to=3-2]
	\arrow["{\Hom_{\OO_D^\times}(-, \phi_* \OO_{\sM_{\infty}})}"{xshift = -0.5em, yshift = -0.6em}, from=3-1, to=2-3]
	\arrow[from=2-2, to=3-2, crossing over]
	\arrow["{\Hom_{\OO_D^\times}(-, f_* \OO_{\sN_{\infty}})}"'{yshift = -0.5em}, from=3-1, to=3-2]
	\arrow[from=3-1, to=4-1]
	\arrow[from=3-2, to=4-2]
	\arrow["{\Hom_{\SL_1(D)}(-, f_* \OO_{\Sigma^{\infty}})}"'{yshift = -0.5em}, from=4-1, to=4-2]
\end{tikzcd}\]

\subsection{Notation}
Recall that we write $H = D^\times / \OO_D^\times$. For $m \geq 1$ we also write
\begin{itemize}
	\item $\OO_{D}^{(m)} \coloneqq \OO_D^\times / (1 + \Pi^m \OO_D)$,
	\item $D^{(m)} \coloneqq D^\times / (1 + \Pi^m \OO_D)$,
	\item $\phi_m \colon \sM_m \rightarrow \sM$ for the finite Galois covering map,
	\item $f_m \colon \sM_m \rightarrow \Omega$ for the composition of $\phi_m$ with $\pi_{\text{GM}}$.
\end{itemize}

\subsection{\texorpdfstring{Representations of $\OO_D^{\times}$}{Representations of O_D^x}}

The functors
\begin{align*}
	\Hom_{\OO_D^\times}(-, f_* \OO_{\sN_{\infty}}) &\colon \Rep_{\sm}^{\fd}(\OO_D^\times) \rightarrow \VectCon^{G^0}(\Omega), \\
	\Hom_{\OO_D^\times}(-, \phi_* \OO_{\sM_{\infty}}) &\colon \Rep_{\sm}^{\fd}(\OO_D^\times) \rightarrow \VectCon^{\GL_n(F)}(\sM)
\end{align*}
are defined as the direct limits of the respective functors 
\begin{align*}
	\Hom_{\OO_{D}^{(m)}}(-, f_{m,*} \OO_{\sN_{m}}) &\colon \Rep^{\fd}(\OO_{D}^{(m)}) \rightarrow \VectCon^{G^0}(\Omega), \\
	\Hom_{\OO_{D}^{(m)}}(-, \phi_{m,*} \OO_{\sM_{m}}) &\colon \Rep^{\fd}(\OO_{D}^{(m)}) \rightarrow \VectCon^{\GL_n(F)}(\sM),
\end{align*}
which are well-defined by Lemma \ref{compatiblelemma}. These are compatible in the following sense.

\begin{lemma}
	For $V \in \Rep_{\sm}^{\fd}(\OO_D^\times)$, there is a natural isomorphism
	\[
		\Hom_{\OO_D^\times}(V, f_* \OO_{\sN_{\infty}}) \cong \pi_{\GM,*} \left( \Hom_{\OO_D^\times}(V, \phi_* \OO_{\sM_{\infty}})|_{\sN} \right)
	\]
	in $\VectCon^{G^0}(\Omega)$.
\end{lemma}

\begin{proof}
	This follows essentially by definition, as $f_m$ is the composition of $\phi_m$ with $\pi_{\GM}$.
\end{proof}

\subsection{\texorpdfstring{Representations of $D^\times$}{Representations of D^x}}

The functors
\begin{align*}
	\Hom_{D^\times}(-, f_* \OO_{\sM_{\infty}}) &\colon \Rep_{\sm}^{\fd}(D^\times) \rightarrow \VectCon^{\GL_n(F)}(\Omega), \\
	\Hom_{\OO_D^\times}^{D^\times}(-, \phi_* \OO_{\sM_{\infty}}) &\colon \Rep_{\sm}^{\fd}(D^\times) \rightarrow \VectCon^{H \times \GL_n(F)}(\sM),
\end{align*}
are defined as the direct limits of the respective functors 
\begin{align*}
	\Hom_{D^{(m)}}(-, f_{m,*} \OO_{\sM_{m}}) &\colon \Rep^{\fd}(D^{(m)}) \rightarrow \VectCon^{\GL_n(F)}(\Omega), \\
	\Hom_{\OO_D^{(m)}}^{D^\times}(-, \phi_{m,*} \OO_{\sM_{m}}) &\colon \Rep^{\fd}(D^{(m)}) \rightarrow \VectCon^{H \times \GL_n(F)}(\sM),
\end{align*}
which are well-defined by Lemma \ref{compatiblelemma}. Here each finite level functor is defined just as described at the start of Section \ref{sect:DmodGalExt}, where the superscript $D^\times$ on the second functor is there purely to differentiate it from the similarly denoted functor of the previous section. Denoting the equivalence of Example \ref{eg:quotients} by
\[
(-)^H \colon \VectCon^{H \times \GL_n(F)}(\sM) \xrightarrow{\sim} \VectCon^{\GL_n(F)}(\sM / H),
\]
and writing $\pi_{\text{GM}} \colon \sM / H \xrightarrow{\sim} \Omega$, these are compatible in the following sense.

\begin{lemma}\label{lem:Dtimescompatible}
	For $V \in \Rep_{\sm}^{\fd}(D^\times)$, there is a natural isomorphism 
	\[
		\Hom_{D^\times}(V, f_* \OO_{\sM_{\infty}}) \cong \pi_{\GM,*}\Hom_{\OO_D^\times}^{D^\times}(V, \phi_* \OO_{\sM_{\infty}})^H
	\]
	in $\VectCon^{\GL_n(F)}(\Omega)$.
\end{lemma}

\begin{proof}
	The natural isomorphism is defined at each finite level $m \geq 1$ by
	\begin{align*}
		\pi_{\text{GM}, *}\Hom_{\OO_D^\times}^{D^\times} (V, \phi_{m,*}\OO_{\sM_m})^{H} &= \pi_{\text{GM}, *}(p_*(\Hom_{\OO_D^\times}^{D^\times} (V, \phi_{m,*}\OO_{\sM_m}))^{H}),\\
		&=\pi_{\text{GM}, *}(\Hom_{\OO_D^\times}^{D^\times}(V, p_*\phi_{m,*}\OO_{\sM_m})^{H}), \\
		&= \pi_{\text{GM}, *}\Hom_{D^\times}(V, p_*\phi_{m,*}\OO_{\sM_m}), \\
		&= \Hom_{D^\times}(V, f_{m,*}\OO_{\sM_m}). \qedhere
	\end{align*}
\end{proof}

In order to see the commutativity of the upper half of the diagram of Section \ref{sect:equivariantDmodulesonDrinfeldSpaces}, note that
\[\begin{tikzcd}
	{\Hom_{\OO_D^\times}(-, \phi_*\OO_{\sM_{\infty}})} \hspace{-1.1em} & \hspace{-1em} {\Rep_{\sm}^{\fd}(D^\times)} & {\VectCon^{H \times \GL_n(F)}(\sM)} \\
	{\Hom_{\OO_D^\times}(-, \phi_*\OO_{\sM_{\infty}})} \hspace{-1.1em} & \hspace{-1em} {\Rep_{\sm}^{\fd}(\OO_D^\times)} & {\VectCon^{\GL_n(F)}(\sM)}
	\arrow["\colon"{description}, draw=none, from=1-1, to=1-2]
	\arrow[from=1-2, to=1-3]
	\arrow[from=1-2, to=2-2]
	\arrow[from=1-3, to=2-3]
	\arrow["\colon"{description}, draw=none, from=2-1, to=2-2]
	\arrow[from=2-2, to=2-3]
\end{tikzcd}\]
commutes by construction. We can use this to see that the square
\[\begin{tikzcd}
	{\Hom_{D^\times}(-, f_*\OO_{\sM_{\infty}})} \hspace{-1em} & \hspace{-1em} {\Rep_{\sm}^{\fd}(D^\times)} & {\VectCon^{\GL_n(F)}(\Omega)} \\
	{\Hom_{\OO_D^\times}(-, f_*\OO_{\sN_{\infty}})} \hspace{-1.1em} & \hspace{-1em} {\Rep_{\sm}^{\fd}(\OO_D^\times)} & {\VectCon^{G^0}(\Omega)}
	\arrow["\colon"{description}, draw=none, from=1-1, to=1-2]
	\arrow[from=1-2, to=1-3]
	\arrow[from=1-2, to=2-2]
	\arrow[from=1-3, to=2-3]
	\arrow["\colon"{description}, draw=none, from=2-1, to=2-2]
	\arrow[from=2-2, to=2-3]
\end{tikzcd}\]
commutes, as a consequence of Lemma \ref{lem:Dtimescompatible} and the following lemma. 

\begin{lemma}
	The diagram
\[\begin{tikzcd}
	{\VectCon^{\GL_n(F)}(\Omega)} & {\VectCon^{H \times \GL_n(F)}(\sM)} \\
	{\VectCon^{G^0}(\Omega)} & {\VectCon^{\GL_n(F)}(\sM)}
	\arrow[from=1-1, to=2-1]
	\arrow["\sim"', from=1-2, to=1-1]
	\arrow[from=1-2, to=2-2]
	\arrow["\sim"', from=2-2, to=2-1]
\end{tikzcd}\]
commutes.
\end{lemma}
\begin{proof}
The lower equivalence is defined as the composition of equivalences,
\[
\VectCon^{\GL_n(F)}(\sM) \xrightarrow{\sim} \VectCon^{G^0}(\sN) \xrightarrow{\sim} \VectCon^{G^0}(\Omega)
\]
where the first is restriction (cf. Example \ref{eg:conncompequiv}), and the second is induced by the $G^0$-equivariant isomorphism
\[
\sN \xhookrightarrow{\iota} \sM \xrightarrow{p} \sM / H \xrightarrow{\pi_{\text{GM}}}\Omega, 
\] 
and therefore it is sufficient to show the commutativity of
\[\begin{tikzcd}
	{\VectCon^{\GL_n(F)}(\sM /H )} & {\VectCon^{H \times \GL_n(F)}(\sM)} \\
	{\VectCon^{G^0}(\sM / H)} & {\VectCon^{\GL_n(F)}(\sM)} \\
	{\VectCon^{G^0}(\sN)}
	\arrow[from=1-1, to=2-1]
	\arrow["\sim"', from=1-2, to=1-1]
	\arrow[from=1-2, to=2-2]
	\arrow["\sim"', from=2-1, to=3-1]
	\arrow["\sim", from=2-2, to=3-1]
\end{tikzcd}\]
But this follows directly from the fact that any $\sF \in \VectCon^{H \times \GL_n(F)}(\sM)$ satisfies
\[
\sF|_{\sN} \cong (p \circ \iota)^* \sF^{H}
\]
as elements of $\VectCon^{G^0}(\sN)$. Indeed, the isomorphism is given on any admissible open subset $U \subset \sN$ by
\begin{align*}
((p \circ \iota)^*\sF^{H})(U) &= \sF(p^{-1}(p \circ \iota(U)))^{H}, \\
&= \sF (H \cdot U)^{H}, \\
&\xrightarrow{\sim} \sF(U),
\end{align*}
with the last isomorphism that induced by sheaf restriction.
\end{proof}

\subsection{\texorpdfstring{Representations of $\SL_1(D)$}{Representations of SL_1(D)}}\label{sect:forgettingequivstructure}

Suppose in this section (Section \ref{sect:forgettingequivstructure}) that $K$ contains $L_m$ for all $m \geq 1$, or equivalently (as $K$ is already assumed to contain the maximal unramified extension of $F$) that $K$ contains $F^{\ab}$, the maximal abelian extension of $F$. Under this assumption, by Theorem \ref{thm:geomconncompdescDT} each space $\sN_m$ is the disjoint union of copies of isomorphic geometrically connected components $\Sigma^m$. We can choose these components in such a way that the image of $\Sigma^{m+1}$ is contained inside $\Sigma^{m}$, and we thus have a tower of geometrically connected rigid spaces,
\[
	\sN \leftarrow \Sigma^1 \leftarrow \Sigma^2 \leftarrow \cdots .
\]
\begin{defn}
Let $m \geq 1$. We write
\begin{itemize}
	\item $\SL_1(D) \coloneqq \ker(\Nrd \colon D^\times \rightarrow F^\times)$,
	\item $\SL_1^m(D) \coloneqq \SL_1(D) \cap (1 + \Pi^m \OO_D)$ for any $m \geq 1$,
	\item $\OO_{F}^{[m]} \coloneqq 1 + \pi^{\lceil\frac{m}{n} \rceil}\OO_F$.
\end{itemize}
\end{defn}

By Lemma \ref{lemma:conncompgaloiscovering}, each morphism $f_m \colon \Sigma^m \rightarrow \sN$ is Galois, with Galois group $\Stab(\Sigma^m) \subset \Gal(\sN_m / \sN)$. By \cite[Prop.\ 3.1]{TAY} this stabiliser is equal to the kernel of the reduced norm map
\[
	\Nrd_m \colon \OO_D^{(m)} = \OO_D^{\times} / (1 + \Pi^m \OO_D) \rightarrow \OO_F^\times / \OO_{F}^{[m]}.
\]
Therefore, using Lemma \ref{lem:grouptheorylemma} below together with the fact that $\Nrd(1+\Pi^m \OO_D) = \OO_{F}^{[m]}$ \cite[Lem.\ 5]{RIEHM}, we obtain an isomorphism
\[
	\SL_1(D) / \SL_1^m(D) \xrightarrow{\sim} \Gal(\Sigma^m / \sN) \hookrightarrow \Gal(\sN_m / \sN) = \OO_D^\times / (1 + \Pi^m \OO_D).
\]
\begin{lemma}\label{lem:grouptheorylemma}
Suppose that $\phi \colon H_1 \rightarrow H_2$ is a surjective group homomorphism and
\begin{align*}
H_1 &\geq H_{1,1} \geq H_{1,2} \geq \cdots, \\
H_2 &\geq H_{2,1} \geq H_{2,2} \geq \cdots
\end{align*}
are chains of normal subgroups with $\phi(H_{1,m}) = H_{2,m}$ for all $m \geq 1$. Then for all $m \geq 1$, $\phi$ induces an isomorphism
\[
	\ker(\phi) / (\ker(\phi) \cap H_{1,m}) \xrightarrow{\sim} \ker(\phi_m)
\]
where $\phi_m$ is the induced map
\[
	\phi_m \colon H_1 / H_{1,m} \rightarrow H_2 / H_{2,m}.
\]
\end{lemma}

\begin{proof}
Consider the commutative diagram,
\[\begin{tikzcd}
	& 1 & 1 & 1 \\
	& {\ker(\phi) \cap H_{1,m}} & {\ker(\phi)} & {\ker(\phi_m)} \\
	1 & {H_{1,m}} & {H_1} & {H_1 / H_{1,m}} & 1 \\
	1 & {H_{2,m}} & {H_2} & {H_2 / H_{2,m}} & 1 \\
	& 1 & 1 & 1
	\arrow[from=3-2, to=3-3]
	\arrow[from=3-3, to=3-4]
	\arrow[from=3-4, to=3-5]
	\arrow[from=3-1, to=3-2]
	\arrow[from=4-4, to=4-5]
	\arrow[from=4-3, to=4-4]
	\arrow[from=4-1, to=4-2]
	\arrow[from=4-2, to=4-3]
	\arrow[from=4-2, to=5-2]
	\arrow[from=4-3, to=5-3]
	\arrow[from=4-4, to=5-4]
	\arrow["\phi", from=3-3, to=4-3]
	\arrow["{\phi_m}", from=3-4, to=4-4]
	\arrow["{\phi|_{H_{1,m}}}", from=3-2, to=4-2]
	\arrow[from=2-3, to=3-3]
	\arrow[from=2-4, to=3-4]
	\arrow[from=1-4, to=2-4]
	\arrow[from=1-3, to=2-3]
	\arrow[from=1-2, to=2-2]
	\arrow[from=2-2, to=3-2]
	\arrow[from=2-3, to=2-4]
	\arrow[from=2-2, to=2-3]
\end{tikzcd}\]
The right two columns are exact, as $\phi$ is surjective. The first column is surjective by the assumption that $\phi(H_{1,m}) = H_{2,m}$. Then the result follows by the Nine Lemma.
\end{proof}

The functor
\[
\Hom_{\SL_1(D)}(-, f_{*}\OO_{\Sigma^{\infty}}) \colon \Rep_{\sm}^{\fd}(\SL_1(D)) \rightarrow \VectCon(\Omega)
\]
is defined as the direct limit of the functors 
\[
	\Hom_{\SL_1(D)}(-, f_{m,*}\OO_{\Sigma^{m}}) \colon \Rep^{\fd}(\SL_1(D) / \SL_1^m(D) ) \rightarrow \VectCon(\Omega)
\]
which is well-defined by Lemma \ref{compatiblelemma}. This is compatible with restriction from $\Rep_{\sm}^{\fd}(\OO_D^\times)$ in the sense that the diagram
\[\begin{tikzcd}
	{\Hom_{\OO_D^\times}(-, f_{*}\OO_{\sN_{\infty}}) } \hspace{-2em} & \hspace{-1em} {\Rep_{\sm}^{\fd}(\OO_D^\times)} & {\VectCon^{G^0}(\Omega)} \\
	{\Hom_{\SL_1(D)}(-, f_{*}\OO_{\Sigma^{\infty}}) } \hspace{-1.1em} & \hspace{-1em} {\Rep_{\sm}^{\fd}(\SL_1(D))} & {\VectCon(\Omega)}
	\arrow["\colon"{description}, draw=none, from=1-1, to=1-2]
	\arrow[from=1-2, to=1-3]
	\arrow[from=1-2, to=2-2]
	\arrow[from=1-3, to=2-3]
	\arrow["\colon"{description}, draw=none, from=2-1, to=2-2]
	\arrow[from=2-2, to=2-3]
\end{tikzcd}\]
commutes, which follows directly from Lemma \ref{lem:conncompgaloisDmodule}.

\subsection{Main Theorem}

As a result of the work of the previous sections, together with Corollary \ref{cor:drinfeldinvconstantfunctions} and Theorem \ref{mainthm1} we have the following. Note that whenever we consider the functor from $\SL_1(D)$-representations there is an implicit assumption that $K$ contains $F^{\ab}$ in order for the spaces $\Sigma^m$ to be defined for all $m \geq 1$.

\begin{thm}\label{mainthmdrinfeld}
	Each labelled functor of the commutative diagram
\[\begin{tikzcd}
	{\Hom_{D^\times}(-, f_* \OO_{\sM_{\infty}})} \hspace{-2.1em} & \hspace{-1.5em} {\Rep_{\sm}^{\fd}(D^\times)} & {\VectCon^{\GL_n(F)}(\Omega)} \\
	{\Hom_{\OO_D^\times}(-, f_* \OO_{\sN_{\infty}})} \hspace{-2.3em} & \hspace{-1.5em} {\Rep_{\sm}^{\fd}(\OO_D^\times)} & {\VectCon^{G^0}(\Omega)} \\
	{\Hom_{\SL_1(D)}(-, f_* \OO_{\Sigma^{\infty}})} \hspace{-1.1em} & \hspace{-1em} {\Rep_{\sm}^{\fd}(\SL_1(D))} & {\VectCon(\Omega)}
	\arrow["\colon"{description}, draw=none, from=1-1, to=1-2]
	\arrow[from=1-2, to=1-3]
	\arrow[from=1-2, to=2-2]
	\arrow[from=1-3, to=2-3]
	\arrow["\colon"{description}, draw=none, from=2-1, to=2-2]
	\arrow[from=2-2, to=2-3]
	\arrow[from=2-2, to=3-2]
	\arrow[from=2-3, to=3-3]
	\arrow["\colon"{description}, draw=none, from=3-1, to=3-2]
	\arrow[from=3-2, to=3-3]
\end{tikzcd}\]
	is exact, monoidal, fully faithful, and has its essential image closed under sub-quotients. 
\end{thm}
This proves Theorem C, and all but the essential image parts of Theorem A and Theorem B.

\begin{remark}\label{rem:dimrank}
	For each functor of Theorem \ref{mainthmdrinfeld}, the dimension of a representation is equal to the rank of the corresponding vector bundle. This follows from part $(3)$ of Theorem \ref{mainthm1} and the construction of each functor. Similarly, each functor preserves duals, symmetric powers, exterior powers and determinants, by Theorem \ref{mainthm1}(2).
\end{remark}

\begin{remark}
	Suppose that $G$ is any subgroup of $\GL_n(F)$ which stabilises the geometrically connected sub-tower $(\Sigma_m)_{m \geq 1}$, such as $\SL_n(F)$ (in fact the subgroups $G$ which stabilise this sub-tower are exactly the subgroups of $\SL_n(F)$ by Theorem \ref{thm:geomconncompdescDT}). Then in both Theorem \ref{mainthmdrinfeld} and throughout the previous Section \ref{sect:forgettingequivstructure} we may replace $\VectCon(\Omega)$ by $\VectCon^G(\Omega)$ everywhere and all statements remain true. We have chosen to state the theorem with $\VectCon(\Omega)$, as the fully faithfulness and statement that the essential image is closed under sub-quotients for $\VectCon(\Omega)$ directly imply the corresponding statements for $\VectCon^G(\Omega)$.
\end{remark}

\begin{remark}
	Similarly, if $G$ is any group with $\GL_n(\OO_F) \leq G \leq G^0$, we may replace $G^0$ with $G$ in Theorem \ref{mainthmdrinfeld} and all statements remain true. They key point is that for such a group $G$, $c(\sN_m)^G = K$, which follows from the same proof as given in Corollary \ref{cor:drinfeldinvconstantfunctions} and was the only property we used of $G^0$. However, we will use $G^0$ in an essential way in the next section, and in our description of the essential image it is no longer true that $G^0$ can be replaced with any such $G$, which is why we choose to state the above theorem with $G^0$.
\end{remark}

\begin{eg}\label{eg:objectsinimage}
	For $m \geq 1$,
	\[
		\Hom_{\OO_D^\times}(K[\OO_D^{(m)}], f_* \OO_{\sN_{\infty}}) = f_{m,*} \OO_{\sN_{m}} \in \VectCon^{G^0}(\Omega).
	\]
	The representation $V_m \coloneqq K[\OO_D^{(m)}]$ of $\OO_D^\times$ admits many extensions to a representation of $D^\times$, and each corresponds to an extension of $f_{m,*} \OO_{\sN_{m}}$ to a $\GL_n(F)$-equivariant vector bundle with connection. For example, for any choice of uniformiser $\Pi$ of $\OO_D^\times$ (or equivalently any choice of section $s \colon H \rightarrow D^\times$), $V_m$ extends to a $D^\times$-representation $V_m^s$ where $\Pi$ acts by 
	\[
		\Pi * n = \Pi n \Pi^{-1}, \qquad n \in \OO_D^{(m)},
	\]
	as in Section \ref{sect:imregrep}. These extensions are typically non-isomorphic for different choices of $\Pi$. From Proposition \ref{prop:regrepn} we have that 
	\[
	\Hom_{\OO_D^\times}(V_m^s, \phi_* \OO_{\sM_{\infty}}) = \phi_{m,*} \OO_{\sM_{m}} \in \VectCon^{H \times \GL_n(F)}(\Omega),
	\]
	where $H$ acts on $\phi_{m,*} \OO_{\sM_{m}}$ through the section $s \colon H \rightarrow D^\times$. In particular, by Lemma \ref{lem:Dtimescompatible},
	\[
		\Hom_{D^\times}(V_m^s, f_* \OO_{\sM_{\infty}}) = \pi_{\GM, *}(\phi_{m,*} \OO_{\sM_{m}})^{H} = \pi_{\GM, *}(\phi_{m,*} \OO_{\sM_{m}})^{\Pi} \in \VectCon^{\GL_n(F)}(\Omega).
	\]
	For each choice of $\Pi$, this can directly be seen to extend $f_{m,*} \OO_{\sN_{m}}$, as there is an isomorphism in $\VectCon^{G^0}(\Omega)$, given on an admissible open subset $U \subset \Omega$ by
	\[
		\pi_{\GM, *}(\phi_{m,*} \OO_{\sM_{m}})^{\Pi}(U) = \OO_{\sM_m}(f_m^{-1}(U))^{\Pi} \xrightarrow{\sim} \OO_{\sM_m}(f_m^{-1}(U) \cap \sN_m) = (f_*\OO_{\sN_m})(U).
	\]
\end{eg}

\subsection{The Essential Image}\label{sect:essimage}

It is possible to give a description of the essential image of each of the functors we have defined, by using part (6) of Theorem \ref{mainthm1}. However this is not an intrinsic description that depends only on the space $\Omega$.

In this section we give an intrinsic description of the essential image of the functor
\[
\Hom_{\OO_D^\times}(-, f_{*}\OO_{\sN_{\infty}}) \colon \Rep_{\sm}^{\fd}(\OO_D^\times) \rightarrow \VectCon^{G^0}(\Omega).
\]

In the following we view $\VectCon^{G^0}(\Omega)$ as a neutral Tannakian category, as in Section \ref{sect:finiteVB}. To verify the assumptions of that section, recall that the set of $K$-rational points of $\Omega$ is $\bP^{n-1}(K)$ with the $F$-rational hyperplanes removed, and so $\Omega$ has a $K$-rational point, $K$ being a proper extension of $F$. Furthermore $c_{\Omega}(\Omega)^{G^0} = K$, for $\Omega$ is geometrically connected \cite[Thm.\ 2.4]{KOH} and thus $c_{\Omega}(\Omega) = K$ by Corollary \ref{cor:cXuequalsk1}.
\begin{thm}\label{mainthmEI}
	The essential image of 
	\[
		\Hom_{\OO_D^\times}(-, f_{*}\OO_{\sN_{\infty}}) \colon \Rep_{\sm}^{\fd}(\OO_D^\times) \rightarrow \VectCon^{G^0}(\Omega)
	\]
	is the full subcategory $\VectCon^{G^0}(\Omega)_{\fin}$ of finite equivariant vector bundles with connection.
\end{thm}

\begin{proof}
	Suppose first that $K$ is algebraically closed, and $\sV \in \VectCon^{G^0}(\Omega)_{\fin}$ is finite. By Proposition \ref{prop:finitecomesfromcovering}, $\sV$ is a sub-object of $f_* \OO_Z$ for some $G^0$-equivariant finite \'{e}tale Galois covering $f \colon Z \rightarrow \Omega$, and thus by Theorem \ref{mainthmdrinfeld} it is sufficient for us to show that $f_* \OO_Z$ is in the essential image. If we write $Z = Z_1 \sqcup \cdots \sqcup Z_r$ as a disjoint union of $G^0$-orbits of connected components, then we have a decomposition
	\[
		f_*\OO_Z = f_* \OO_{Z_1} \oplus \cdots \oplus f_* \OO_{Z_r}
	\]
	in $\VectCon^{G^0}(\Omega)$ and thus we may assume that $G^0$ acts transitively on the connected components of $Z$. Furthermore, writing
	\[
	G^n \coloneqq \{g \in \GL_{n}(F) \mid \nu(\det(g)) \in n\bZ\} = G^0 \times \pi^{\bZ} I_n,
	\]
	then as $\pi^{\bZ} I_n$ acts trivially on $\Omega$, we may view $f \colon Z \rightarrow \Omega$ as $G^n$-equivariant by letting $\pi^{\bZ} I_n$ act trivially on $Z$. For any rigid space $X$ with an action of $G^n$ we can form the rigid space
	\[
	X \times_{G^n} \GL_n(F) \coloneqq (X \times \GL_n(F) ) / G^n,
	\]
	where $G^n$ acts diagonally through its (right) action on $X$ and its left multiplication action on $\GL_n(F)$. This naturally has a (right) action of $\GL_n(F)$ by right multiplication. If additionally $X$ has an action of $\GL_n(F)$, then there is a natural $\GL_n(F)$-equivariant morphism
	\[
	X  \times_{G^n} \GL_n(F) \rightarrow X, \qquad (x,g) \mapsto xg,
	\]
	which is finite \'{e}tale because $G^n$ is of finite index in $\GL_n(F)$. Furthermore, there is a natural $G^n$ equivariant open embedding
	\[
	\iota \colon X \hookrightarrow X \times_{G^n} \GL_n(F), \qquad x \mapsto (x,1).
	\]
	In particular, for $f \colon Z \rightarrow \Omega$, we obtain a $\GL_n(F)$-equivariant finite \'{e}tale covering
	\[
	h \colon Y \coloneqq Z \times_{G^n} \GL_n(F) \rightarrow \Omega \times_{G^n} \GL_n(F)  \rightarrow \Omega,
	\]
	for which $h \circ \iota = f$, where $\iota \colon Z \hookrightarrow Y$ as above. As $K$ is algebraically closed, we may apply the factorisation theorem of Scholze-Weinstein \cite[Thm.\ 7.3.1]{SW} to obtain a commutative diagram of $\GL_n(F)$-equivariant morphisms of rigid spaces
	\[\begin{tikzcd}
		{\sM_m} & Y \\
		& \Omega
		\arrow["\phi", from=1-1, to=1-2]
		\arrow["{f_m}"', from=1-1, to=2-2]
		\arrow["h", from=1-2, to=2-2]
	\end{tikzcd}\]
	Let $X_0$ be any connected component of $\sM_m$. As $Y$ is the disjoint union of copies of $\iota(Z)$ which are transitively permuted by $\GL_n(F)$, there is some $g \in \GL_n(F)$ with $\phi(X_0) \subset g(\iota(Z))$. Because $\phi$ is $\GL_n(F)$-equivariant, the connected component $X \coloneqq g^{-1}(X_0)$ has $\phi(X) \subset \iota(Z)$. The connected component $X$ is a connected component of $\sM_m^i$ for some $i \in \bZ$, and therefore $\phi(\sM_m^i) \subset \iota(Z)$ because the connected components of $\sM_m^i$ are permuted transitively by $G^0$ and $\iota(Z)$ is stable under the action of $G^0$. We would now like to show that $\iota_* \OO_{Z} \rightarrow \phi_* \OO_{\sM_m^i}$ is injective. To this end, first note that as $X$ is connected, $\phi(X)$ is contained in $\iota(Z_0)$ for some connected component $Z_0$ of $Z$. Because $f \colon Z_0 \rightarrow \Omega$ is finite \'{e}tale, $Z_0$ is smooth and therefore normal, and therefore applying \cite[Prop.\ A.5]{KOH} to the finite \'{e}tale $\phi \colon X \rightarrow \iota(Z_0)$ we deduce that $\iota_* \OO_{Z_0} \hookrightarrow \phi_* \OO_{X}$. As $X$ is a connected component of $\sM_m^i$, $\phi_* \OO_{X} \hookrightarrow \phi_* \OO_{\sM_m^i}$ and therefore $\iota_* \OO_{Z_0} \hookrightarrow \phi_* \OO_{\sM_m^i}$. Then because $\phi \colon \sM_m^i \rightarrow Z$ is $G^0$-equivariant and $G^0$ acts transitively on the connected components of $Z$, $\iota_* \OO_{Z} \rightarrow \phi_* \OO_{\sM_m^i}$ is injective. In particular, as $h_*$ is left-exact, 
	\[
	f_* \OO_Z = h_* \iota_* \OO_Z \hookrightarrow h_* \phi_* \OO_{\sM_{m}^i} = f_{m,*}\OO_{\sM_m^i}.
	\]
	Now, as $f_m$ is $\Pi$-equivariant with respect to the trivial action of $\Pi$ on $\Omega$ and $\Pi^i(\sM_m^i) = \sN_m$,
	\begin{align*}
		f_{m,*} \OO_{\sM_m^i} &= f_{m,*} \Pi^{i}_* \OO_{\sM_m^i}, \\
		&= f_{m,*}\OO_{\sN_m}
	\end{align*}
	and therefore, as $f_{m,*}\OO_{\sN_m}$ is in the essential image (Example \ref{eg:objectsinimage}), we are done because the essential image is closed under sub-objects by Theorem \ref{mainthmdrinfeld}.
	
	Suppose now that $K$ is general, let $C$ be a completion of an algebraic closure of $K$, and for any complete extension $L$ of $K$ with norm extending the norm on $K$ let $\sH_L$ denote the functor defined over $L$. Suppose that $\sV \in \VectCon^{G^0}(\Omega)$ is finite. We may then consider $\sV_C$, as defined in Appendix \ref{sect:basechange}, which is finite in $\VectCon^{G^0}(\Omega_C)$ and so there is some $V \in \Rep^{\fd}_{\sm, C}(\OO_D^\times)$ with $\sH_C(V) \cong \sV_C$ because $C$ is algebraically closed. Because $V$ is inflated from a representation of a finite group, there is some finite extension $L$ of $K$ and $W \in \Rep^{\fd}_{\sm, L}(\OO_D^\times)$ such that $V \cong W_C$ and all irreducible constituents $W_1, ... ,W_r$ of $W$ are absolutely irreducible, meaning that each $\End_{L[\OO_D^{\times}]}(W_m) = L$.
	
	We first ``descend'' the isomorphism $\sH_C(V) \cong \sV_C$ to an isomorphism $\sH_L(W) \cong \sV_L$. Write
	\[
		W = \bigoplus_{m = 1}^r W_m^{\oplus n_m}
	\]
	as an isotypical decomposition into irreducible $L$-representations. Because each $W_m$ is absolutely irreducible, setting $V_m \coloneqq W_{m, L}$,
	\[
		V = \bigoplus_{m = 1}^r V_m^{\oplus n_m}
	\]
	is an irreducible decomposition of $V$ and each $\End_{C[\OO_D^{\times}]}(V_m) = C$. From Proposition \ref{prop:basechangecommuteswithhoms},
	\begin{align*}
		\dim_L \Hom_{G^0 \text{-}\sD_{\Omega_L}}(\sH_L(W_m^{\oplus n_m}), \sV_L)  &= \dim_C \Hom_{G^0 \text{-}\sD_{\Omega_C}}(\sH_L(W_m^{\oplus n_m})_C, \sV_C), \\
		&= \dim_C \Hom_{G^0 \text{-}\sD_{\Omega_C}}(\sH_C(V_m^{\oplus n_m}), \sH_C(V)), \\
		&= \dim_C \Hom_{C[\OO_D^{\times}]}(V_m^{\oplus n_m}, V), \\
		&= n_m^2,
	\end{align*}
	using that $\sH_C$ is fully faithful and that $\sH_L(W_m)_C = \sH_C(V_m)$ by Lemma \ref{lem:basechangecompat}. Because $\sV_L$ is finite, $\sV_L$ is semisimple, which can be seen from the proof of Proposition \ref{prop:finitecomesfromcovering}. We can thus write
	\[
		\sV_L = \sV_1 \oplus \cdots \oplus \sV_s
	\]
	as a direct sum of simple objects $\sV_i \in \VectCon^{G^0}(\Omega_L)$. Because $\sH_L$ is fully faithful and preserves irreducibility, $\sH_L(W_m)$ is simple and 
	\[
		\End_{G^0 \text{-}\sD_{\Omega_L}}(\sH_L(W_m)) \cong \End_{L[\OO_D^{\times}]}(W_m) = L,
	\]
	and therefore among the $\sV_i$ there are $n_m$ copies of $\sH_L(W_m)$. In particular, there is an inclusion
	\[
		\sH_L(W) = \bigoplus_{m = 1}^r \sH_L(W_m^{\oplus n_m}) \hookrightarrow \sV_L.
	\]
	This inclusion is actually an isomorphism, as $\sH_L(W)$ is a direct summand of $\sV_L$ and $\sH_L(W)$ and $\sV_L$ have the same rank. We therefore have that $\sH_L(W) \cong \sV_L$.
	
	This shown, $\sV$ is itself a sub-object of an object in the essential image of $\sH_K$, as
	\[
		\sV \hookrightarrow (\sV_L)|_K \xrightarrow{\sim} \sH_L(W)|_K \hookrightarrow \sH_L((W|_K)_L)|_K \xrightarrow{\sim} (\sH_K(W|_K)_L)|_K = \bigoplus_{i = 1}^{[L\colon K]} \sH_K(W|_K),
	\]	
	using that $W \hookrightarrow (W|_K)_L$ \cite[Prop.\ 9.18(b)]{ISA} and $\sH_L((W|_K)_L) \cong \sH_K(W|_K)_L$ by Lemma \ref{lem:basechangecompat}. Then $\sV$ is itself in the essential image of $\sH_K$ by Theorem \ref{mainthmdrinfeld}.
\end{proof}
This finishes the proof of Theorem B. We now use this to finish the proof of Theorem A.
\begin{cor}\label{cor:thmAessim}
	The essential image of 
	\[
	\Hom_{D^\times}(-, f_* \OO_{\sM_{\infty}}) \colon \Rep_{\sm}^{\fd}(D^\times) \rightarrow \VectCon^{\GL_n(F)}(\Omega)
	\]
	is the full subcategory $\VectCon^{\GL_n(F)}(\Omega)_{G^0\text{-}\fin}$ with objects those that are finite when viewed as $G^0$-equivariant vector bundles with connection.
\end{cor}

\begin{proof}
	First, from Definition \ref{defn:finiteobject} we see that the monoidal equivalence
	\[
		\VectCon^{\GL_n(F)}(\sM) \xrightarrow{\sim} \VectCon^{G^0}(\Omega)
	\]
	restricts to an equivalence
	\[
	\VectCon^{\GL_n(F)}(\sM)_{\fin} \xrightarrow{\sim} \VectCon^{G^0}(\Omega)_{\fin}
	\]
	between the full subcategories of finite objects. In particular, by Theorem \ref{mainthmdrinfeld}, Theorem \ref{mainthmEI}, and the commutativity of the diagram of Section \ref{sect:equivariantDmodulesonDrinfeldSpaces}, we have that the functor
	\[
	\Hom_{\OO_D^\times}(-, \phi_*\OO_{\sM_{\infty}}) \colon \Rep_{\sm}^{\fd}(\OO_D^\times) \rightarrow \VectCon^{\GL_n(F)}(\sM)_{\fin}
	\]
	is an equivalence. Let us fix a uniformiser $\Pi$ of $\OO_D^\times$, and identify $H$ with $\Pi^{\bZ}$. By the commutativity of the diagram of Section \ref{sect:equivariantDmodulesonDrinfeldSpaces}, to show the claim it is equivalent to show that the top functor of
\[\begin{tikzcd}
	{\Hom_{\OO_D^\times}^{D^\times}(-, \phi_*\OO_{\sM_{\infty}})} \hspace{-1.2em} & \hspace{-1em} {\Rep_{\sm}^{\fd}(D^\times)} & {\VectCon^{\Pi^{\bZ} \times \GL_n(F)}(\sM)} \\
	{\Hom_{\OO_D^\times}(-, \phi_*\OO_{\sM_{\infty}})} \hspace{-1.2em} & \hspace{-1em} {\Rep_{\sm}^{\fd}(\OO_D^\times)} & {\VectCon^{\GL_n(F)}(\sM)}
	\arrow["\colon"{description}, draw=none, from=1-1, to=1-2]
	\arrow[from=1-2, to=1-3]
	\arrow[from=1-2, to=2-2]
	\arrow[from=1-3, to=2-3]
	\arrow["\colon"{description}, draw=none, from=2-1, to=2-2]
	\arrow[from=2-2, to=2-3]
\end{tikzcd}\]
	has essential image equal to the full subcategory
	\[
	\VectCon^{\Pi^{\bZ} \times \GL_n(F)}(\sM)_{\GL_n(F)\text{-}\fin}
	\]
	whose objects are those objects which restrict to finite objects of $\VectCon^{\GL_n(F)}(\sM)$. By commutativity of the above square, the essential image is contained in this full subcategory, so let us suppose conversely that
	\[
	\sV \in \VectCon^{\Pi^{\bZ} \times \GL_n(F)}(\sM)
	\]
	is finite when viewed as a $\GL_n(F)$-equivariant vector bundle with connection. Then there is some $V \in \Rep_{\sm}^{\fd}(\OO_D^\times)$ of level $m$ and an isomorphism
	\[
		\Phi \colon \sV \xrightarrow{\sim} \Hom_{\OO_D^\times}(V, \phi_* \OO_{\sM_{m}})
	\]
	in $\VectCon^{\GL_n(F)}(\sM)$. Because the action of $\Pi$ commutes with the action of $\GL_n(F)$, the inverse image functor extends to an endofunctor
	\[
		\Pi^{-1} \colon \VectCon^{\GL_n(F)}(\sM) \rightarrow \VectCon^{\GL_n(F)}(\sM),
	\]
	where for $\sW \in \VectCon^{\GL_n(F)}(\sM)$, $\Pi^{-1} \sW$ has action of $\sD_{\sM}$ via $\Pi^{\sD_{\sM}} \colon \sD_{\sM} \rightarrow \Pi^{-1}(\sD_{\sM})$ and $\GL_n(F)$-equivariant structure
	\[
	g^{\Pi^{-1}\sW} \coloneqq \Pi^{-1}g^{\sW} \colon \Pi^{-1}\sW \rightarrow \Pi^{-1}(g^{-1}\sW) = g^{-1}(\Pi^{-1}\sW),
	\]
	for any $g \in \GL_n(F)$. For example, with respect to this structure 
	\[
		\Pi^{\sV} \colon \sV \rightarrow \Pi^{-1} \sV,
	\]
	is a morphism in $\VectCon^{\GL_n(F)}(\sM)$. Furthermore,
	\[
		\Pi^{\phi_*\OO_{\sM_m}} \circ (-) \colon \Hom_{\OO_D^\times}({}^{\Pi} V, \phi_* \OO_{\sM_{m}}) \xrightarrow{\sim} \Pi^{-1} \Hom_{\OO_D^\times}(V, \phi_* \OO_{\sM_{m}})
	\]
	is an isomorphism in $\VectCon^{\GL_n(F)}(\sM)$, where ${}^{\Pi}V$ is $V$ viewed with action of $x \in \OO_D^\times$ by $\Pi x \Pi^{-1}$. We therefore have an isomorphism
	\[
	(\Pi^{\phi_*\OO_{\sM_m}} \circ (-))^{-1} \circ \Pi^{-1}\Phi \circ \Pi^{\sV} \circ \Phi^{-1} \colon \Hom_{\OO_D^\times}(V, \phi_* \OO_{\sM_{m}}) \xrightarrow{\sim} \Hom_{\OO_D^\times}({}^{\Pi} V, \phi_* \OO_{\sM_{m}})
	\]
	in $\VectCon^{\GL_n(F)}(\sM)$ and by the fully faithfulness of Theorem \ref{mainthmdrinfeld}, there is a unique $\OO_D^\times$-linear isomorphism $\varphi \colon {}^{\Pi} V \rightarrow V$ for which this equals the functorially induced morphism
	\[
		((-) \circ \varphi) \colon \Hom_{\OO_D^\times}(V, \phi_* \OO_{\sM_{m}}) \xrightarrow{\sim} \Hom_{\OO_D^\times}({}^{\Pi} V, \phi_* \OO_{\sM_{m}}).
	\]
	Because $\varphi \colon {}^{\Pi} V \rightarrow V$ is $\OO_D^\times$-linear we may extend $V$ to a representation of $D^\times$ by setting $\Pi^{-1}$ to act by $\varphi$, and obtain
	\[
	\sF \coloneqq \Hom_{\OO_D^\times}^{D^\times}(V, \phi_* \OO_{\sM_m}) \in \VectCon^{\Pi^{\bZ} \times \GL_n(F)}(\sM)
	\]
	with action (by definition of the functor) of $\Pi$ by
	\[
		\Pi^{\sF} \coloneqq \Pi^{\phi_* \OO_{\sM_m}} \circ - \circ \Pi^{-1} \colon \sF \xrightarrow{\sim} \Pi^{-1} \sF.
	\]
	Therefore, as $\Pi^{-1}$ acts on $V$ by $\varphi$,
	\begin{align*}
		\Pi^{\sF} &= \Pi^{f_* \OO_X} \circ (-) \circ \varphi, \\
		&= (\Pi^{\phi_* \OO_{\sM_m}} \circ (-) ) \circ  ((-) \circ \varphi), \\
		&= (\Pi^{\phi_* \OO_{\sM_m}} \circ (-) )  \circ (\Pi^{\phi_* \OO_{\sM_m}} \circ (-))^{-1} \circ \Pi^{-1}\Phi \circ \Pi^{\sV} \circ \Phi^{-1}, \\
		&= \Pi^{-1}\Phi \circ \Pi^{\sV} \circ \Phi^{-1},
	\end{align*}
	and therefore $\Phi$ defines an isomorphism $\Phi \colon \sV \rightarrow \sF$ in $\VectCon^{\Pi^{\bZ} \times \GL_n(F)}(\sM)$.
\end{proof}

\begin{remark}\label{rem:GvG0}
	It is natural to ask if we can replace $G^0$ with $G$ in Theorem A, i.e.\ if
	\[
		\VectCon^{\GL_n(F)}(\Omega)_{\fin} \subset \VectCon^{\GL_n(F)}(\Omega)_{G^0\text{-}\fin} 
	\]
	is actually an equality. In the case when $n = 1$, the functor of Theorem A is
	\[
		\Hom_{F^{\times}}(-, K \otimes_L L_{\infty}) \colon \Rep^{\fd}_{\sm, K}(F^\times) \xrightarrow{\sim} \Rep^{\fd}_{K}(F^\times)_{\OO_F^\times\text{-}\fin}
	\]
	where $L_{\infty}$ is the union of the Lubin-Tate extensions $L_m / L$, and the inclusion above is
	\[
		\Rep^{\fd}_{K}(F^\times)_{\fin} \subset \Rep^{\fd}_{K}(F^\times)_{\OO_F^\times\text{-}\fin}.
	\]
	That this is strict follows from Corollary \ref{cor:inflationfinite}, exhibited by any representation of $F^\times$ inflated along $F^\times \rightarrow F^\times / \OO_F^\times \cong \bZ$ that isn't inflated from a finite quotient of $\bZ$.
\end{remark}

\appendix

\section{Continuous Actions and Finite \'{E}tale Covers}\label{sect:generalities}

In this appendix we show that the property that a group action on a rigid space is continuous lifts along finite \'{e}tale covers. Throughout, let $K$ be a characteristic $0$ complete non-archimedean field, and write $\sR$ for the ring of integers of $K$. 

We first recall the notion of a continuous group action on a rigid space $X$, as defined in \cite[Def. 3.1.8]{ARD}. Let $X$ be a qcqs rigid space over $K$. Then using a formal model for $X$, \cite[Thm.\ 3.1.5]{ARD} defines a hausdorff topology $\tau_X$ on $\Aut_K(X,\OO_X)$, which is independent of the formal model chosen for $X$.

\begin{defn}\label{ctsactiondef}
Let $G$ be a topological group, $X$ a rigid analytic variety over $K$. Then \emph{$G$ acts continuously on $X$}, if there is a group homomorphism $\rho : G \rightarrow \Aut_K(X ,\OO_X)$, such that for all qcqs admissible open subsets $U \subset X$,
\begin{itemize}
\item
$G_U = \rho^{-1}(\Stab(U))$ is open in $G$,
\item
The induced group homomorphism $\rho_U : G_U \rightarrow \Aut_K(U,\OO_U)$ is continuous with respect to the topology $\tau_U$ on $\Aut_K(U,\OO_U)$ described above, and the subspace topology of $G_U$.
\end{itemize}
\end{defn}

Here $\Stab(U) \leq \Aut_K(X ,\OO_X)$ are those automorphisms for which the underlying map of sets $f$ has $f(U) = U$.

\begin{prop}\label{prop:continuousactionlifts}
	Suppose that $f \colon X \rightarrow Y$ is a finite \'{e}tale Galois cover of quasi-Stein rigid spaces with Galois group $H$. Let $G$ be a topological group with an open topologically finitely generated profinite subgroup acting on $X$ and $Y$ such that $f \colon X \rightarrow Y$ is $G$-equivariant, and $G$ acts continuously on $Y$. Then $G$ acts continuously on $X$.
\end{prop}
\begin{proof}
Let $(U_m)_{m \geq 0}$ be a quasi-Stein rigid cover of $Y$, and let $(V_m := f^{-1}(U_m))_{m \geq 0}$ be the corresponding quasi-Stein cover of $X$. In order to show that $G$ acts continuously on $X$, it is sufficient to show that $G_{m} \coloneqq G_{V_m}$ is open in $G$, and acts continuously on $V_m$ for all $m \geq 0$. Indeed, if so, then for an arbitrary qcqs admissible open $V \subset X$, $V = \cup_{m \geq 0}(V \cap V_m)$, hence because $V$ is quasi-compact, $V \subset V_m$ for some $m \geq 0$. Then $(G_m)_V \subset G_m$ is open (hence open in $G$), and thus $G_V$ is open in $G$. Furthermore, $(G_m)_V \rightarrow \Aut_K(V, \OO_V)$ is continuous by \cite[Thm.\ 3.1.10]{ARD}, hence $G_V \rightarrow \Aut_K(V, \OO_V)$ is continuous.

Fix $m \geq 0$. In order to show that $G_m$ is open and $G_m$ acts continuously on $V_m$, because $V_m$ is affinoid, hence qcqs, by \cite[Thm.\ 3.1.10]{ARD} it is sufficient to show that $G_m$ is open and $\rho : G_m \rightarrow \Aut_K(V_m, \OO_{V_m})$ is continuous. First, because the action of $G$ on $Y$ is continuous, $G_{U_m}$ is open, and so because $G_{U_m} \subset G_{m}$, $G_m$ is open in $G$. The morphism $V_m \rightarrow U_m$ is Galois with Galois group $H$, and thus the group $\Aut_{\OO(U_m)}(\OO(V_m))$ is finite, which follows from \cite[Thm.\ 3.5]{CHR} and Lemma \ref{lemma:conncompgaloiscovering}. We are then done after applying the following Lemma \ref{lem:affinoidctsaction}.
\end{proof}

\begin{lemma}\label{lem:affinoidctsaction}
	Suppose that $A \rightarrow B$ is an \'{e}tale morphism of affinoid algebras over $k$, which is $G$-equivariant with respect to actions $\sigma \colon G \rightarrow \Aut_k(A)$, $\rho \colon G \rightarrow \Aut_k(B)$ of some topological group $G$. Suppose that $\Aut_A(B)$ is finite, $G$ has an open topologically finitely generated profinite subgroup, and $\sigma$ is continuous. Then $\rho$ is continuous.
\end{lemma}
	
\begin{proof}
	Let $\sA$ be a formal model for $A$. Because $\sigma$ is continuous, by \cite[Lem.\ 3.2.4]{ARD}, we may after possibly enlarging $\sA$ assume that $\sA$ is $G$-equivariant. Any \'{e}tale morphism of affinoid algebras is \emph{standard \'{e}tale} \cite[Observation 3.1.2]{PJ}, which means that there is a presentation of the morphism $A \rightarrow B$ by 
	\[
		\iota \colon A \rightarrow B = A\langle T_1, ... , T_n \rangle / (f_1, ... , f_n), \text{ where } \det \Delta \in B^\times, \text{ for } \Delta \coloneqq \left( \frac{\partial f_i}{\partial T_j} \right)_{ij}.
	\]
	We take our formal model of $B$ to be
	\[
		\sB \coloneqq \iota(\sA \langle T_1, ... , T_n \rangle) \subset B.
	\]
	First consider an $\sR$-linear derivation $\tilde{\partial} \colon \sA \rightarrow \sA$. This has a unique extension to a $K$-linear derivation $\tilde{\partial} \colon A \rightarrow A$. Then because $A \rightarrow B$ is \'{e}tale, by Lemma \ref{extendderivations} there is a unique $\partial \colon B \rightarrow B$ such that for any $a \in A$,
	\[
		\partial(\iota(a)) = \iota(\tilde{\partial}(a)).
	\]
	\emph{We first show the following claim:} for any $N > 0$, there is some $M > 0$ such that for any $\tilde{\partial} \colon \sA \rightarrow \sA$ with $\tilde{\partial}(\sA) \subset p^M \sA$, the extension $\partial$ satisfies
	\[
		\partial(\sB) \subset p^N \sB.
	\]
	Fix $N > 0$. We first consider the elements $T_1, ... , T_n \in \sB$. Write each
	\[
		f_i = \sum_{\alpha} c^i_{\alpha} T^\alpha,
	\]
	for some $c^i_{\alpha} \in \iota(\sA)$. Suppose now that $\tilde{\partial} \colon \sA \rightarrow \sA$ with $\tilde{\partial}(\sA) \subset p^M \sA$, for some $M > 0$ to be chosen later. Then because any $k$-derivation of $B$ is continuous \cite[Satz. 2.1.5]{BKKN},
	\begin{align*}
		0 &= \partial(f_i), \\
		&= \partial \left( \sum_{\alpha} c^i_{\alpha} T^\alpha \right), \\
		&= \sum_{\alpha} \partial(c^i_{\alpha} T^\alpha), \\
		&= \sum_{\alpha} \partial(c^i_{\alpha}) T^\alpha + c^i_{\alpha} \partial(T^\alpha).
	\end{align*}
	We can simplify the right-hand term to
	\begin{align*}
		\sum_{\alpha} c^i_{\alpha} \partial(T^\alpha) &= \sum_{\alpha} \sum_{j = 1}^n c^i_{\alpha} T^{\alpha_1} \cdots \partial(T^{\alpha_j}) \cdots T^{\alpha_n}, \\
		&= \sum_{i = 1}^n \frac{\partial f_i}{\partial T_j} \partial(T_j).
	\end{align*}
	For the left-hand term, each $c^i_{\alpha} \in \iota(\sA)$ and thus $c^i_{\alpha} = \iota(a^i_{\alpha})$ for some $a^i_{\alpha} \in \sA$, hence
	\begin{align*}
		\sum_{\alpha} \partial(c^i_{\alpha}) T^\alpha &= \sum_{\alpha} \iota(\tilde{\partial}(a^i_{\alpha})) T^\alpha \in p^M \sB,
	\end{align*}
	because $\tilde{\partial}(a^i_{\alpha}) \in p^M \sA$, so $\iota(\tilde{\partial}(a^i_{\alpha}))T^{\alpha} \in p^M \sB$, and $p^M \sB$ is closed. Combined, this shows that,
	\[
		\Delta \begin{pmatrix} \partial(T_1) \\ \vdots \\ \partial(T_n) \end{pmatrix} \in (p^M \sB)^n.
	\]
	Now because $\det(\Delta) \in B^\times$, $\Delta$ is invertible, and we set $M_0 \coloneqq \min \{r \geq 0 \mid p^r \Delta^{-1} \in M_n(\sB)\}$. Then for $k = 1, ... ,n$,
	\[
		\partial(T_k) \in p^{M-M_0} \sB.
	\]
	Let us now fix $M \coloneqq N + M_0$, and consider a general element of $\sB$, which will be of the form,
	\[
		b = \sum_{\alpha} \iota(a_{\alpha}) T^{\alpha}
	\]
	for some $a_{\alpha} \in \sA$. When we apply $\partial$ we obtain, similarly to above,
	\begin{align*}
		\partial(b) &= \sum_{\alpha} \iota(\tilde{\partial}(a_{\alpha})) T^{\alpha} + \iota(a_{\alpha}) \partial(T^{\alpha}) \in p^N \sB.
	\end{align*}
	For each $\alpha$, the left-hand term is in $p^N \sB$, and the right-hand term is too, as 
	\begin{align}
		\partial(T^{\alpha}) = \sum_{k = 1}^{n} \alpha_k T_1^{\alpha_1} \cdots T_k^{\alpha_k - 1} \cdots T_n^{\alpha_n} \cdot \partial(T_k).
	\end{align}
	Therefore, as $p^N \sB$ is closed, $\partial(b) \in p^N \sB$ and we have shown the claim.

	Now we want to show that the morphism $\rho \colon G \rightarrow \End_k(B)$ is continuous. For any $g \in G$ with $(\sigma(g) - 1)(\sA) \subset p^{2}\sA$ we can consider the logarithm $u_g \coloneqq \log(\sigma(g)) \in p^2 \Der_R(\sA)$ of $\sigma(g)$ (see \cite[Lem.\ 3.2.5]{ARD} for background on the logarithm and exponential in this context). As described above, this $\sR$-linear derivation $u_g$ has a unique extension to a $K$-linear derivation of $A$ (which we also denote by $u_g$) and by Lemma \ref{extendderivations} there is a unique $v_g \in \Der_k(B)$ which extends $u_g$. 
	
	Because $\Aut_A(B)$ is finite, by \cite[Lem.\ 3.2.4]{ARD} we may enlarge $\sB$ to obtain an $\Aut_A(B)$-stable formal model $\sB'$ of $B$ which contains $\sB$. To show that $\rho \colon G \rightarrow \End_k(B)$ is continuous we need to show that for any $n \geq 0$, $\rho^{-1}(\Aut_k^n(B))$ is open in $G$, where
	\[
		\Aut_k^n(B) \coloneqq \{ \phi \in \Aut_k(B) \mid (\phi - 1)(\sB') \subset p^n \sB'\}.
	\]
	To this end, fix $n \geq 2$. Because $\sB$ and $\sB'$ are both formal models of $B$, the argument of the proof of \cite[Thm.\ 3.1.5 (b)]{ARD}, the claim established above and the fact that $\sigma$ is continuous together show the existence of an open subgroup $G_0$ of $G$ for which $v_g(\sB') \subset p^n \sB'$ for any $g \in G_0$.
	
	Define $\tau \colon G_0 \rightarrow \Aut_k^n(B)$ by $\tau(g) = \exp(v_g)$. We would like to show that, like $\tau$, $\rho(g) \in \Aut_k^n(B)$ for any $g \in G_0$. Firstly, we note that $\tau$ is a group homomorphism. Indeed, $\sigma(g)$ is a group homomorphism, and $\sigma(g) = \exp(\log(\sigma(g))) = \exp(u_g)$, so for any $g, h \in G_0$
	\[
		\exp(u_g) \circ \exp(u_h) = \exp(u_{gh}).
	\]
	Taking the logarithm,
	\[
		\log(\exp(u_g) \circ \exp(u_h)) = u_{gh}.
	\]
	The	derivation $v_{gh}$ is an extension of $u_{gh}$ to an $\sR$-derivation of $\sB'$, and	$\log(\exp(v_g) \circ \exp(v_h))$ is an extension of $\log(\exp(u_g) \circ \exp(u_h))$ to $\sB'$, and therefore by the uniqueness of Lemma \ref{extendderivations}
	\[
		\log(\exp(v_g) \circ \exp(v_h)) = v_{gh}.
	\]
	Applying the exponential, we see that $\tau(g) \circ \tau(h) = \tau(gh)$.
	
	Consider the difference of $\rho$ and $\tau$, denoted by $\lambda \colon G_0 \rightarrow \Aut_A(B)$, defined by
	\[
		\lambda (g) = \rho(g) \tau(g)^{-1} \in \Aut_A(B).
	\]
	This lies in $\Aut_A(B)$ because both $\rho(g)$ and $\tau(g)$ restrict to $\sigma(g)$ on $A$. Consider also the restriction homomorphism
	\[
		r_n \colon \Aut_{\sR}(\sB') \rightarrow \Aut_{\sR / p^n \sR} (\sB' / p^n \sB'),
	\]
	the kernel of which is $\Aut_k^n(B)$. Because $\tau(g)$ and $\lambda(g)$ both stabilise $\sB'$, $\rho(g)$ does too. Therefore, we may apply $r_n$ to $\rho(g)$, and note that as $\tau(g) \in \Aut_k^n(B) = \ker(r_n)$,
	\[
		r_n(\rho(g)) = r_n(\lambda(g)) \in r_n(\Aut_A(B)).
	\]
	Now $r_n \circ \rho \colon G_0 \rightarrow r_n(\Aut_A(B))$ is a homomorphism from $G_0$ to the finite group $r_n(\Aut_A(B))$, and thus because finite index subgroups of $G$ are open \cite{NIK}, $\ker(r_n \circ \rho) \leq G_0$ is open. Therefore, $\ker(r_n \circ \rho) \subset \rho^{-1}(\Aut_k^n(B))$ is an open subgroup, and we're done.
\end{proof}

We also will make use of the following lemma which allows us to extend Proposition \ref{prop:continuousactionlifts} to spaces which are disjoint unions of quasi-Stein spaces. For example, we use it together with Proposition \ref{prop:continuousactionlifts} to show that the action of $\GL_n(F)$ on the Drinfeld tower is continuous (Corollary \ref{cor:ctsactionGLnM}).

\begin{lemma}\label{lem:ctsactionconncomp}
Suppose that $X$ is a rigid space with an action of a topological group $G$, and $H$ is an open subgroup of $G$ which stabilises and acts continuously on each connected component of $X$. Then the action of $G$ on $X$ is continuous.
\end{lemma}

\begin{proof}
	Suppose that $U \subset X$ is a qcqs admissible open subset of $X$. Then the stabiliser $G_U \subset G$ is open, as $G_U$ contains the open subgroup $H_U$ of $G$. Furthermore, the morphism of topological groups $G_U \rightarrow \Aut(U)$ is continuous, as the restriction to the subgroup $H_U$ is continuous, which follows from the commutative diagram
\[\begin{tikzcd}
	{G_U} & {\Aut(U)} \\
	{H_U} & {\prod_i \Aut(U_i)}
	\arrow[from=1-1, to=1-2]
	\arrow[hook, from=2-1, to=1-1]
	\arrow[from=2-1, to=2-2]
	\arrow[hook, from=2-2, to=1-2]
\end{tikzcd}\]
where the product is over the connected components $X_i$, and $U_i \coloneqq X_i \cap U$.
\end{proof}

\subsection{Continuous Equivariant Line Bundles with Connection}

The above results also allow us to show that any torsion equivariant line bundle with connection is continuous in the sense of \cite{AW2}, which we make use of in Section \ref{sect:linebundles}.

Suppose henceforth that $G$ is a topological group which acts on $Y$ continuously, where $Y$ is a rigid space over $K$. We now want to define a subgroup of $\PicCon^G(Y)$ of those $G$-equivariant line bundles with connection which satisfy an appropriate continuity condition.

For any $K$-Banach space $V$, it is shown in \cite{AW2} during the preamble to Definition 3.2.3 that the unit group $\sB(V)^\times$ of $\sB(V)$, the $K$-algebra of bounded $K$-linear endomorphisms of $V$, forms a topological group such that the congruence subgroups
\[
	\Gamma_n(\sV) \coloneqq \{\gamma \in \sB(V)^\times \mid (1 - \gamma)(\sV) \subset \pi^n \sV \}
\]
form a system of open neighbourhoods of the identity, where $\sV$ is a unit ball of $V$.

Suppose now that $\sM \in \VectCon^G(Y)$. Then \emph{loc. cit.} also shows that for any affinoid subdomain $U$ of $X$ the action map induces
\[
	G_U \rightarrow \sB(\sM(U))^\times,
\]
where $\sM(U)$ has its canonical $K$-Banach space structure as a finitely generated $\OO_Y(U)$-module.

\begin{defn}\label{def:piccts}
We define $\PicCon^G_{\cts}(Y) \subset \PicCon^G(Y)$ to be the subgroup consisting of those $\sL \in \PicCon^G(Y)$ such that for any affinoid open subset $U$ of $Y$ the map
\[
	G_U \rightarrow \sB(\sL(U))^\times
\]
is continuous. We also define
\[
	\Con^G_{\cts}(Y) \coloneqq \ker(\PicCon^G_{\cts}(Y) \rightarrow \Pic(Y)).
\]
\end{defn}

\begin{remark}
The group we denote by $\PicCon_{\cts}^G(Y)$ is written as $\PicCon^G(Y)$ in \cite{AW2}, and similarly what we call $\Con_{\cts}^G(Y)$ is written as $\Con^G(Y)$.
\end{remark}
The key result we want to use, derived from \cite{AW2}, is the following.
\begin{prop}\label{AWProp}
	Suppose $Y$ is quasi-Stein and geometrically connected, and that $G$ acts continuously on $Y$. Then for any $d \geq 1$ there is an exact sequence,
	\[
		0 \rightarrow \Hom(G, \mu_d(K)) \rightarrow \Con^G_{\cts}(Y)[d] \rightarrow (\OO(Y)^\times / K^\times \OO(Y)^{\times d})^G.
	\]
\end{prop}

\begin{proof}
We have the commutative diagram,
\[\begin{tikzcd}
	0 & {\Hom(G, K^\times)} & {\Con^G_{\cts}(Y)} & {\Con_{\cts}(Y)^G} \\
	0 & {\Hom(G, K^\times)} & {\PicCon^G_{\cts}(Y)} & {\PicCon_{\cts}(Y)^G}
	\arrow[from=1-1, to=1-2]
	\arrow[from=1-2, to=1-3]
	\arrow["{=}"', from=1-2, to=2-2]
	\arrow[from=1-3, to=1-4]
	\arrow[from=1-3, to=2-3]
	\arrow[from=1-4, to=2-4]
	\arrow[from=2-1, to=2-2]
	\arrow[from=2-2, to=2-3]
	\arrow[from=2-3, to=2-4]
\end{tikzcd}\]
where the second row is exact by \cite[Prop.\ 3.2.14]{AW2}. Because the second and third vertical maps are injective, a simple diagram chase shows that the first row is also exact. Now because $d$-torsion is left exact, this first row remains exact after applying $(-)[d]$. Finally, we obtain the required exact sequence using \cite[Prop.\ 3.1.16]{AW2}, which exhibits a $G$-equivariant isomorphism
\[
	\theta_d \colon \Con_{\cts}(Y)[d] \xrightarrow{\sim} \OO_Y(Y)^\times / K^\times \OO_Y(Y)^{\times d}. \qedhere
\]
\end{proof}

This allows us to say more about the image of the homomorphism of Corollary \ref{mainthmabelian} in the presence of a continuous group action.

\begin{lemma}\label{lem:ctsimpliesfactors}
Suppose that $f \colon X \rightarrow Y$ are as in Proposition \ref{prop:continuousactionlifts}, and that the action of $G$ commutes with the action of $H$. Then the image of the homomorphism
\[
\widehat{H} \rightarrow \PicCon^G(Y)[e]
\]
of Corollary \ref{mainthmabelian} is contained inside the subgroup $\PicCon_{\cts}^G(Y)[e]$.
\end{lemma}
 
\begin{proof}
	Let $\chi \in \widehat{H}$ and recall that this homomorphism sends $\chi$ to $\sL_{\chi} = e_{\chi} \cdot f_* \OO_X$. We want to show that for any affinoid open subset $U \subset Y$, that the natural map
	\[
		G_U \rightarrow \sB(\sL_{\chi}(U))^\times
	\]
	is continuous. By Proposition \ref{prop:continuousactionlifts} the action of $G$ on $X$ is continuous, and therefore by \cite[Lemma 3.2.4]{AW2} we have that $\OO_X \in \PicCon_{\cts}^G(X)$, and thus
	\[
		G_{V} \rightarrow \sB(\OO_X(V))^\times
	\] 
	is continuous, where $V = f^{-1}(U)$. Now because $f \colon X \rightarrow Y$ is $G$-equivariant, $G_U \subset G_V$, and thus
	\[
		G_U \rightarrow \sB(\OO_X(V))^\times = \sB((f_*\OO_X)(U))^\times 
	\]
	is continuous. By Corollary \ref{maincor2}, there is a decomposition
	\[
		(f_*\OO_X)(U) = \bigoplus_{\rho \in \Irr(H)} e_{\rho} \cdot (f_* \OO_X)(U),
	\]
	which is preserved $G_U$ because the actions of $H$ and $G$ commute. The image of $G_U$ is thus contained in the diagonal subgroup
	\[
		\prod_{\rho\in \Irr(H)} \sB(e_{\rho} \cdot (f_* \OO_X)(U))^\times \hookrightarrow \sB((f_*\OO_X)(U))^\times,
	\]
	and this inclusion is a homeomorphism onto its image, which can be seen directly from the topology on the groups $\sB(-)^\times$. In particular, 
	\[
		G_U \rightarrow \prod_{\rho\in \Irr(H)} \sB(e_{\rho} \cdot (f_* \OO_X)(U))^\times
	\]
	is continuous, and therefore the projection to $\sB(\sL_{\chi}(U))^\times$ is too.
\end{proof}

\section{Base Change for $\VectCon^G(X)$ on Quasi-Stein Spaces}\label{sect:basechange}

Let $K$ be a complete non-archimedean field, and let $L$ be a complete non-archimedean field extension of $K$ such that the norm on $L$ restricts to the norm on $K$. Let $X$ be a smooth quasi-separated rigid space over $K$, with an action of an abstract group $G$.

As $X$ is quasi-separated, one may consider the base change $X_L$ of $X$ to $L$ (see Remark \ref{rem:geomconnsubtle}). In this appendix we consider the base change functor
\[
	(-)_L \colon \VectCon^G(X) \rightarrow \VectCon^G(X_L).
\]
The main result is a compatibility statement that base change commutes with taking homomorphisms of objects, Proposition \ref{prop:basechangecommuteswithhoms}, which we use in the proof of Theorem \ref{mainthmEI} to extend Theorem \ref{mainthmEI} to fields which may not be algebraically closed.

\subsection{Base Change Functors}

In order to define base change for an arbitrary complete field extension $L$ of $K$, we need the following lemma.

\begin{lemma}\label{lem:dercommfieldext}
	Suppose that $A$ is an affinoid algebra over $K$ such that $\Omega_{A/K}$ is a projective $A$-module. Then the natural map
	\[
		L \hspace{0.1em} \widehat{\otimes}_K \Der_K(A) \rightarrow \Der_L(L \hspace{0.1em} \widehat{\otimes}_K A)
	\]
	is an isomorphism.
\end{lemma}

\begin{proof}
	Let $K\langle z_1, ... ,z_r \rangle / (f_1, ... , f_s)$ be a presentation of the $K$-affinoid algebra $A$, and let $A_L \coloneqq L \hspace{0.1em} \widehat{\otimes}_K A$. By \cite[Prop.\ 6.1.1(12)]{BGR}, $A_L = L \langle z_1, ... , z_r \rangle / (f_1, ... , f_s)$, and
	\begin{align*}
	\Omega_{A_L / L} = (A_L \cdot dz_1 \oplus \cdots \oplus A_L \cdot d_{z_r} ) / \langle df_1, ... ,df_s \rangle_{A_L}.
	\end{align*}
	Similarly, there is an exact sequence
	\[
		0 \rightarrow \langle df_1, ... ,df_s \rangle_{A} \rightarrow A \cdot dz_1 \oplus \cdots \oplus A \cdot d_{z_r} \rightarrow \Omega_{A/K} \rightarrow 0,
	\]
	and hence by the exactness of $L \hspace{0.1em} \widehat{\otimes}_K (-)$ (see \cite[\S 2.1.3 Fact (3)]{CDN2}) there is an exact sequence
	\[
		0 \rightarrow \langle df_1, ... ,df_s \rangle_{A_L} \rightarrow A_L \cdot dz_1 \oplus \cdots \oplus A_L \cdot d_{z_r} \rightarrow L \hspace{0.1em} \widehat{\otimes}_K \Omega_{A/K} \rightarrow 0,
	\]
	and thus a canonical isomorphism $\Omega_{A_L / L} \xrightarrow{\sim} L \hspace{0.1em} \widehat{\otimes}_K \Omega_{A/K}$. Then the natural map $L \hspace{0.1em} \widehat{\otimes}_K \Der_K(A) \rightarrow \Der_L(A_L)$ factorises as the composition of isomorphisms
	\begin{align*}
		L \hspace{0.1em} \widehat{\otimes}_K \Der_K(A) &\xrightarrow{\sim} A_L \otimes_A \Hom_A(\Omega_{A/K}, A), \\
		&\xrightarrow{\sim} \Hom_{A_L}(A_L \otimes_A \Omega_{A/K}, A_L \otimes_A A), \\
		&\xrightarrow{\sim} \Hom_{A_L}(\Omega_{A_L/L}, A_L), \\
		&\xrightarrow{\sim} \Der_L(A_L).
	\end{align*}
	Here we have used that $\Omega_{A/K}$ is projective to apply \cite[Chapter II, \S 5, Prop. 7]{BOUR}, and used the basic properties \cite[Prop.\ 2.1.7(7)]{BGR} and \cite[Prop.\ 3.7.2(6)]{BGR} of the completed tensor product.
\end{proof}

Suppose now that $\sV \in \VectCon^G(X)$. The base change $X_L$ is defined by gluing together the affinoid spaces $U_L = \Sp(L \hspace{0.1em} \widehat{\otimes}_K \hspace{0.1em} \OO_X(U))$ for $U \in \sB$, and as an $\OO_{X_L}$-module $\sV_L$ is defined above $U_{L}$ to correspond to the $L \hspace{0.1em} \widehat{\otimes}_K \hspace{0.1em} \OO_X(U)$-module
\[
	\sV_L(U_{L}) \coloneqq L \hspace{0.1em} \widehat{\otimes}_K \sV(U).
\]
For each $U \in \sB$ there is a natural map
\begin{align*}
L \hspace{0.1em} \widehat{\otimes}_K \hspace{0.1em} \sT_X(U) &\rightarrow \End_K(L \hspace{0.1em} \widehat{\otimes}_K \sV(U))
\end{align*}
which defines an action of $\sD_{X_L}(U_L)$ using Lemma \ref{lem:dercommfieldext}. These glue to give $\sV_L$ the structure of a $\sD_{X_L}$-module, which extends to a $G\text{-}\sD_{X_L}$-module structure, where the $G$-equivariant structure on $\sV_L$ is that induced from the $G$-equivariant structure of $\sV$.

\subsection{Compatibility with Galois Covering Functors}\label{sect:basechangecompatible}

For a group $H$, we can consider the base change functor
\[
	(-)_L \colon \Mod_{K[H]}^{\fd} \rightarrow \Mod_{L[H]}^{\fd}.
\]
The following lemma shows that base change is compatible with the functors of Section \ref{sect:DmodGalExt}.
\begin{lemma}\label{lem:basechangecompat}
	Suppose that $X$ and $Y$ are rigid spaces over $K$, and $(f \colon X \rightarrow Y, G, H)$ are as described in Section \ref{sect:DmodGalExt} with $H = N$. Then for $V \in \Mod_{K[H]}^{\fd}$ the natural map
	\begin{equation}\label{eqn:basechangeisomforGaloisfunctors}
		\Hom_{K[H]}(V, f_*\OO_X)_L \rightarrow \Hom_{L[H]}(V_L, f_*\OO_{X_L})
	\end{equation}
	is an isomorphism.
\end{lemma}

\begin{proof}
	To see that (\ref{eqn:affinebasechangelocal}) is an isomorphism, it suffices to show the same is true above any affinoid open subset $U \in \sB$. To ease notation, set $A \coloneqq \OO_Y(U)$ and $B \coloneqq \OO_X(f^{-1}(U))$. The map (\ref{eqn:basechangeisomforGaloisfunctors}) on sections above $U$ is the map
	\begin{equation}\label{eqn:affinebasechangelocal}
	L \hspace{0.1em} \widehat{\otimes}_K \Hom_{K[H]}(V, B)_L \rightarrow \Hom_{L[H]}(L \otimes_K V, L \hspace{0.1em} \widehat{\otimes}_K B).
	\end{equation}
	Because $H = N$, $H$ is a finite group, and therefore any $V \in \Mod_{K[H]}^{\fd}$ can be written as a direct sum of certain simple modules $V_1, ... , V_r$. In particular, we can write $V = \bigoplus_m V_m^{\oplus n_m}$ for some $n_m \geq 1$, and therefore, because $K[H]_L = L[H]$, the natural inclusion  
	\begin{equation}\label{eqn:denseimage}
		L \otimes_K B \rightarrow L \hspace{0.1em} \widehat{\otimes}_K B
	\end{equation}
	factors as the sum of maps
	\[
	\bigoplus_m \left(L \otimes_K \Hom_{K[H]}(V_m, B)\right)^{\oplus n_m} \rightarrow \bigoplus_m \Hom_{L[H]}(L \otimes_K V_m, L \hspace{0.1em} \widehat{\otimes}_K B)^{\oplus n_m}. 
	\]
	Because (\ref{eqn:denseimage}) extends uniquely to an isomorphism from the completion $L \hspace{0.1em} \widehat{\otimes}_K B$, the same is true of each map
	\[
	L \otimes_K \Hom_{K[H]}(V_m, B) \rightarrow \Hom_{L[H]}(L \otimes_K V_m, L \hspace{0.1em} \widehat{\otimes}_K B)
	\]
	and therefore each extension
	\[
		L \hspace{0.1em} \widehat{\otimes}_K \Hom_{K[H]}(V_m, B) \rightarrow \Hom_{L[H]}(L \otimes_K V_m, L \hspace{0.1em} \widehat{\otimes}_K B)
	\]
	is an isomorphism. As $V$ is a direct sum of the $V_m$, we see that (\ref{eqn:affinebasechangelocal}) is also an isomorphism.
\end{proof}

\subsection{Compatibility with Homomorphisms}

We first need some basic lemmas.

\begin{lemma}\label{lem:quasiSteincover}
	Suppose that $X$ is a connected quasi-Stein rigid space over $K$. Then there exists a quasi-Stein covering
	\[
		U_1 \subset U_2 \subset \cdots
	\]
	of $X$ by connected admissible open affinoid subsets. If $X$ is smooth and geometrically connected, then each $U_m$ can also be taken to also be smooth and geometrically connected.
\end{lemma}

\begin{proof}
	Let $V_1 \subset V_2 \subset \cdots$ be any fixed quasi-Stein open covering of $X$. Let $U_1$ be a connected component of $V_1$, and for $m \geq 1$ inductively define $U_{m+1}$ to be the connected component of $V_{m+1}$ which contains $U_m$. We first show that 
	\[
		X = \bigcup_{m \geq 1} U_m.
	\]
	Using property ($G_1$) of the rigid space $X$ \cite[Def. 9.3.1(4)]{BGR}, both
	\[
	U \coloneqq \bigcup_{m \geq 1} U_{m} \ \text{ and } \ U' \coloneqq \bigcup_{m \geq 1} V_m \setminus U_m
	\]
	are admissible open subsets of $X$ which form a disjoint admissible open covering of $X$ by property ($G_2$) of $X$ \cite[Def. 9.3.1(4)]{BGR}. In particular, because $U$ is non-empty and $X$ is connected, $U = X$. Secondly, to verify the quasi-Stein property we need to show that for any $m \geq 1$, $\OO(U_{m+1}) \rightarrow \OO(U_m)$ has dense image. We have a commutative diagram,
\[\begin{tikzcd}
	{\OO(V_m)} & {\OO(V_{m+1})} \\
	{\OO(U_m)} & {\OO(U_{m+1})}
	\arrow["{p_m}", two heads, from=1-1, to=2-1]
	\arrow["f"', from=1-2, to=1-1]
	\arrow["{p_{m+1}}", two heads, from=1-2, to=2-2]
	\arrow["h", from=2-2, to=2-1]
\end{tikzcd}\]
	and thus
	\[
		h(\OO(U_{m+1})) = h(p_{m+1}(\OO(V_{m+1}))) = p_m(f(\OO(V_{m+1}))).
	\]
	When we take the closure in $\OO(U_m)$, because $p_m$ is continuous,
	\begin{align*}
		\overline{h(\OO(U_{m+1}))} &= \overline{p_m(f(\OO(V_{m+1})))}, \\
		&= \overline{p_m(\overline{f(\OO(V_{m+1}))})}, \\
		&= \overline{p_m(\OO(V_m))}, \\
		&= \OO(U_m).
	\end{align*}
	Now suppose that $X$ is smooth and geometrically connected. Because $X$ is smooth, we may consider the sheaf $c_X$, which because $X$ is geometrically connected has $c_X(X) = K$ by Corollary \ref{cor:cXuequalsk1}. Therefore, because $c_X$ is a sheaf, $K = c_X(X)$ is the inverse limit of the inverse system
	\[
		c_X(U_1) \leftarrow c_X(U_2) \leftarrow \cdots .
	\]
	Each $c_X(U_m)$ is a finite field extension of $K$, which follows from Lemma \ref{lem:finiteextension} and the fact that each $U_m$ is connected. In particular, this inverse system has injective transition maps and the inverse limit $K$ is the intersection of the images $c_X(U_m)$ in $c_X(U_1)$. Because $c_X(U_1)$ is finite dimensional, there must be some $m_0 \geq 1$ such that $c_X(U_m) = K$ for $m \geq m_0$. In particular, $U_m$ is geometrically connected for all $m \geq m_0$ by Corollary \ref{cor:cXuequalsk2}, and we obtain the required quasi-Stein covering by including only those $U_m$ with $m \geq m_0$.
\end{proof}

\begin{lemma}\label{lem:Dmodhatotimescomm}
	Suppose that $A$ is an affinoid algebra over $K$, $\Omega_{A/K}$ is a projective $A$-module, and $M$ and $N$ are $\sD(A)$-modules which are finitely generated projective as $A$-modules. Then the natural map
	\begin{equation}\label{eqn:localbasechangeDmodules}
	L \hspace{0.1em} \widehat{\otimes}_K \Hom_{\sD(A)}(M,N) \rightarrow \Hom_{\sD(A_L)} (L \hspace{0.1em} \widehat{\otimes}_K M, L \hspace{0.1em} \widehat{\otimes}_K N)
	\end{equation}
	is injective, and an isomorphism whenever $L$ is a finite extension of $K$.
\end{lemma}

\begin{proof}
	First note that the natural map 
	\begin{equation}\label{eqn:basechangeaffineversion}
		L \hspace{0.1em} \widehat{\otimes}_K \Hom_{A}(M,N) \rightarrow \Hom_{A_L} (L \hspace{0.1em} \widehat{\otimes}_K M, L \hspace{0.1em} \widehat{\otimes}_K N)
	\end{equation}
	is canonically identified with the natural map
	\[
		A_L \otimes_A \Hom_A(M,N) \rightarrow \Hom_{A_L}(A_L \otimes_A M, A_L \otimes_A N)
	\]
	by \cite[Prop.\ 2.1.7(7)]{BGR} and \cite[Prop.\ 3.7.2(6)]{BGR}, which is an isomorphism by \cite[Chapter II, \S 5, Prop. 7]{BOUR} because $M$ and $N$ are finitely generated projective as $A$-modules. To ease notation, write $\sH \coloneqq \Hom_A(M,N)$ and $\sH^L \coloneqq \Hom_{A_L} (L \hspace{0.1em} \widehat{\otimes}_K M, L \hspace{0.1em} \widehat{\otimes}_K N)$. There is a natural $\sD(A)$-module structure on $\sH$ as described at the end of Section \ref{sect:generaleqDmodules}, and therefore $L \hspace{0.1em} \widehat{\otimes}_K \sH$ has a $\sD(A_L)$-module structure through the functor $(-)_L$. Similarly, $\sH^L$ has a natural structure as a $\sD(A_L)$-module, and the map (\ref{eqn:basechangeaffineversion}) is an isomorphism of $\sD(A_L)$-modules with respect to these structures.
	
	From the definition of the action of $\Der_L(A_L)$ on $\sH^L$, the subset $\Hom_{\sD(A_L)} (L \hspace{0.1em} \widehat{\otimes}_K M, L \hspace{0.1em} \widehat{\otimes}_K N)$ consists of those elements which are sent to $0$ by all $\partial \in \Der_L(A_L)$. Similarly for $\sH$ and $\Der_K(A)$.
	
	Let $\partial_1, ... ,\partial_r$ be generators of $\Der_K(A)$ as an $A$-module, which also under the isomorphism of Lemma \ref{lem:dercommfieldext} correspond to generators $\partial_1^L, ... ,\partial_r^L$ of $\Der_L(A_L)$ as an $A_L$-module. We therefore have exact sequences
	\begin{align*}
		0 \rightarrow \Hom_{\sD(A)}(M,N) \rightarrow \sH \xrightarrow{[\partial_1, ... , \partial_r]} \prod_{i = 1}^r \sH
	\end{align*}
	and
	\begin{align*}
		0 \rightarrow \Hom_{\sD(A_L)} (L \hspace{0.1em} \widehat{\otimes}_K M, L \hspace{0.1em} \widehat{\otimes}_K N) \rightarrow 
		\sH^L \xrightarrow{[\partial_1^L, ... , \partial_r^L]}
		\prod_{i = 1}^r \sH^L.
	\end{align*}
	By construction, these have the property that $L \hspace{0.1em} \widehat{\otimes}_K [\partial_1, ... , \partial_r]$ becomes identified under the map (\ref{eqn:basechangeaffineversion}) with $[\partial_1^L, ... , \partial_r^L]$, and so the map (\ref{eqn:basechangeaffineversion}) restricts to an identification
	\[
		\ker(L \hspace{0.1em} \widehat{\otimes}_K [\partial_1, ... , \partial_r]) \xrightarrow{\sim} \Hom_{\sD(A_L)} (L \hspace{0.1em} \widehat{\otimes}_K M, L \hspace{0.1em} \widehat{\otimes}_K N).
	\]
	Therefore the map (\ref{eqn:localbasechangeDmodules}) is injective, because the natural map
	\[
		L \hspace{0.1em} \widehat{\otimes}_K \Hom_{\sD(A)}(M,N) \rightarrow \ker(L \hspace{0.1em} \widehat{\otimes}_K [\partial_1, ... , \partial_r])
	\]
	is injective, and (\ref{eqn:localbasechangeDmodules}) is an isomorphism if and only if this map is an isomorphism. This is where one needs the hypothesis that $L$ is finite over $K$, for in this case there is a canonical identification $L \otimes_K (-) \xrightarrow{\sim} L \hspace{0.1em} \widehat{\otimes}_K (-)$ and one can use the exactness of $L \otimes_K (-)$. In general one cannot use the exactness of $L \hspace{0.1em} \widehat{\otimes}_K (-)$ here, because $\Der_K(A)$ need not act by strict morphisms in general.
\end{proof}

We can now prove the main result of this section. We expect that not all assumptions are necessary (cf.\ Remark \ref{rem:removeassumptions}), however the statement we give here is sufficient for our purposes.

\begin{prop}\label{prop:basechangecommuteswithhoms}
Suppose that $X$ is a smooth quasi-Stein rigid space, $G$ is a group that acts on $X$, and $L$ is a complete field extension of $K$ with norm extending the norm on $K$. Suppose that $c_X(X) = K$ and $X$ has a $K$-rational point.

Then for $\sV, \sW \in \VectCon^G(X)$, the natural map
\begin{equation}\label{eqn:mainbasechangemap}
L \hspace{0.1em} \widehat{\otimes}_K \Hom_{G\text{-}\sD_X}(\sV, \sW) \rightarrow \Hom_{G\text{-}\sD_{X_L}}(\sV_L, \sW_L)
\end{equation}
is an isomorphism if either $L$ if finite over $K$ or $\sV$ and $\sW$ are finite as objects of $\VectCon(X)$.
\end{prop}

\begin{proof}
	We first claim that it is sufficient to show that the natural map 
	\begin{equation}\label{eqn:naturalmapforDX}
		L \hspace{0.1em} \widehat{\otimes}_K \Hom_{\sD_X}(\sV, \sW) \rightarrow \Hom_{\sD_{X_L}}(\sV_L, \sW_L)
	\end{equation}
	is an isomorphism. Suppose then that this is the case. Because $c_X(X) = K$ and $X$ contains a $K$-rational point we may consider $\VectCon(X)$ as a neutral Tannakian category as in Section \ref{sect:finiteVB}, and therefore $\Hom_{\sD_X}(\sV, \sW)$ is finite dimensional over $K$. In particular, the natural map
	\[
	L \otimes_K \Hom_{\sD_X}(\sV, \sW) \rightarrow L \hspace{0.1em} \widehat{\otimes}_K \Hom_{\sD_X}(\sV, \sW) \rightarrow \Hom_{\sD_{X_L}}(\sV_L, \sW_L)
	\]
	is an isomorphism \cite[Prop.\ 3.7.2(6)]{BGR}. This is furthermore $G$-equivariant, where $G$ acts trivially on $L$ and through the natural conjugation action on $\Hom_{\sD_X}(\sV, \sW)$ and $\Hom_{\sD_{X_L}}(\sV_L, \sW_L)$ (cf.\ Remark \ref{rem:homspaceequivsheaf}). Then because the action is trivial on $L$, the inclusion induces an isomorphism
	\[
		L \otimes_K \Hom_{\sD_X}(\sV, \sW)^G \xrightarrow{\sim} (L \otimes_K \Hom_{\sD_X}(\sV, \sW))^G \xrightarrow{\sim} \Hom_{\sD_{X_L}}(\sV_L, \sW_L)^G
	\]
	which is exactly the map
	\[
	L \otimes_K \Hom_{G\text{-}\sD_X}(\sV, \sW) \xrightarrow{\sim} L \hspace{0.1em} \widehat{\otimes}_K \Hom_{G\text{-}\sD_X}(\sV, \sW) \rightarrow \Hom_{G\text{-}\sD_{X_L}}(\sV_L, \sW_L)
	\]
	and so the claim is shown.

	Towards showing that (\ref{eqn:naturalmapforDX}) is an isomorphism, let $U_1 \subset U_2 \subset \cdots$ be a quasi-Stein presentation of $X$ by geometrically connected $U_m$, which exists because $c_X(X) = K$ by Lemma \ref{lem:quasiSteincover}. The $K$-rational point of $X$ is contained in some $U_m$, and so after removing some elements from the covering we may assume that $U_1$, and hence each $U_m$, contains a $K$-rational point. Set 
	\[
	\sF \coloneqq \underline{\Hom}_{\sD_X}(\sV, \sW) \hookrightarrow \sG \coloneqq \underline{\Hom}_{\OO_X}(\sV, \sW),
	\]
	and
	\[
	\sF^L \coloneqq \underline{\Hom}_{\sD_{X_L}}(\sV_L, \sW_L) \hookrightarrow \sG^L \coloneqq \underline{\Hom}_{\OO_{X_L}}(\sV_L, \sW_L).
	\]
	Both $\sG$ and $\sG^L$ are vector bundles with connection on $X$ and $X_L$ respectively. The restriction map $\sG(U_{m+1}) \rightarrow \sG(U_m)$ factors as a composition of isomorphisms
	\[
		\sG(U_{m+1}) \xrightarrow{\sim} \OO(U_{m+1}) \otimes_{\OO(U_{m+1})} \sG(U_{m+1}) \xrightarrow{\sim} \OO(U_{m}) \otimes_{\OO(U_{m+1})} \sG(U_{m+1}) \xrightarrow{\sim} \sG(U_m)
	\]
	because $\sG(U_m)$ is flat over $\OO(U_m)$ (as $\sG$ is a vector bundle) and each restriction map $\OO(U_{m+1}) \rightarrow \OO(U_m)$ is injective \cite[Prop.\ 4.2]{ABB}. In particular, the system 
	\[
		\sG(U_1) \leftarrow \sG(U_2) \leftarrow \sG(U_3) \leftarrow \cdots	
	\]
	has injective transition maps, and therefore the same is true of the system
	\[
		\sF(U_1) \leftarrow \sF(U_2) \leftarrow \sF(U_3) \leftarrow \cdots
	\]
	For each $m \geq 1$, because $c_X(U_m) = K$ and $U_m$ contains a $K$-rational point we may consider $\VectCon(U_i)$ as a neutral Tannakian category as in Section \ref{sect:finiteVB}, and in particular $\sF(U_m) = \Hom_{\sD_{U_m}}(\sV|_{U_i}, \sW|_{U_i})$ is finite dimensional over $K$. In particular, from this and the injectivity of the transition maps, $\sF(U_{m+1}) \rightarrow \sF(U_m)$ are isomorphisms for sufficiently large $m$.
	
	Now suppose that $L$ is finite over $K$. Then we may apply Lemma \ref{lem:Dmodhatotimescomm} to see that each 
	\[
		L \hspace{0.1em} \widehat{\otimes}_K \sF(U_m) \rightarrow \sF^L(U_{m, L})
	\]
	is an isomorphism, and therefore we have that the map (\ref{eqn:naturalmapforDX}) factorises as the composition
	\begin{align*}
		L \hspace{0.1em} \widehat{\otimes}_K \sF(X) &\xrightarrow{\sim} \varprojlim_{m \geq 1} L \hspace{0.1em} \widehat{\otimes}_K \sF(U_m)
		\xrightarrow{\sim} \varprojlim_{m \geq 1} \sF^L(U_{m, L})
		\xrightarrow{\sim} \sF^L(X_L)
	\end{align*}
	of isomorphisms, and is therefore an isomorphism. Here we have used that the transition maps of $(\sF(U_m))_m$ are eventually isomorphisms to see that $L \hspace{0.1em} \widehat{\otimes}_K (-)$ commutes with the inverse limit.

	Now suppose $L$ is arbitrary, but that $\sV$ and $\sW$ are finite as objects of $\VectCon(X)$. Then $\sV \oplus \sW$ is finite by Corollary \ref{cor:finiteclosure}, and hence by Proposition \ref{prop:finitecomesfromcovering}, $\sV \oplus \sW$ is a direct summand of $f_*\OO_Z$ for some finite \'{e}tale Galois covering $f \colon Z \rightarrow X$ with Galois group $H$. There is some finite extension $F$ of $K$ such that each connected component of $Z_F$ is geometrically connected by Corollary \ref{cor:cXuequalsk2}, and therefore, as we have established (\ref{eqn:naturalmapforDX}) is an isomorphism for finite extensions, we may assume without loss of generality that each connected component of $Z$ is geometrically connected. From the pullback equivalence
	\[
		f^* \colon \VectCon(X) \xrightarrow{\sim} \VectCon^{H}(Z),
	\]
	of Proposition \ref{FunctorEquivalence}, as $f^{*}\OO_X = \OO_Z$ we see that
	\[
		c_Z(Z)^{H} \cong \Hom_{H\text{-}\sD_Z}(\OO_Z, \OO_Z) \cong \Hom_{\sD_X}(\OO_X, \OO_X) \cong c_X(X) = K 
	\]
	by Lemma \ref{altdesccX}. Fix a connected component $Z_0$ of $Z$, which by our assumption that $Z_0$ is geometrically connected has $c_Z(Z_0) = K$ by Corollary \ref{cor:cXuequalsk1}. We may argue just as in the proof of Lemma \ref{lem:OXirrGDX} to see that $H$ must act transitively on the connected components of $Z$, and therefore by Example \ref{eg:conncompequiv} restriction defines an equivalence
	\[
		\VectCon^{H}(Z) \rightarrow \VectCon^{H_0}(Z_0).
	\]
	We may therefore apply Theorem \ref{mainthm1} to the extension $f \colon Z_0 \rightarrow X$ because $c_Z(Z_0) = K$, and therefore using Lemma \ref{lem:conncompgaloisDmodule} we have a commutative diagram
\[\begin{tikzcd}
	{\Mod_{K[H]}^{\fd}} && {\Mod_{K[H_0]}^{\fd}} \\
	{\VectCon^H(Z)} && {\VectCon^{H_0}(Z_0)} \\
	& {\VectCon(X)}
	\arrow[from=1-1, to=1-3]
	\arrow[from=1-1, to=2-1]
	\arrow["\sim", from=1-3, to=2-3]
	\arrow["\sim", from=2-1, to=2-3]
	\arrow["\sim"', from=2-1, to=3-2]
	\arrow["\sim", from=2-3, to=3-2]
\end{tikzcd}\]
	In particular, $\sV$ and $\sW$ are in the essential image of the functor 
	\[
		\Phi \coloneqq \Hom_{K[H_0]}(-, f_* \OO_{Z_0}) \colon \Mod_{K[H_0]}^{\fd} \rightarrow \VectCon(X)
	\]
	so there are $V, W \in \Mod_{K[H_0]}^{\fd}$ and isomorphisms $\Phi(V) \xrightarrow{\sim} \sV$ and $\Psi(W) \xrightarrow{\sim} \sW$. We may also consider the functor
	\[
		\Phi_L \coloneqq \Hom_{L[H_0]}(-, f_* \OO_{Z_{0,L}}) \colon \Mod_{L[H_0]}^{\fd} \rightarrow \VectCon(X_L)
	\]
	which satisfies $\sV_L \xrightarrow{\sim} \Phi(V)_L \xrightarrow{\sim} \Phi_L(V_L)$ and $\sW_L \xrightarrow{\sim} \Phi(W)_L \xrightarrow{\sim} \Phi_L(W_L)$ by Lemma \ref{lem:basechangecompat}. We therefore have a commutative diagram
\[\begin{tikzcd}
	{L \otimes_K \Hom_{K[H_0]}(V,W)} && {L \otimes_K \Hom_{\sD_X}(\sV,\sW)} \\
	{\Hom_{L[H_0]}(V_L, W_L)} & {\Hom_{\sD_{X_L}}(\Phi_L(V_L), \Phi_L(W_L))} & {\Hom_{\sD_{X_L}}(\sV_L, \sW_L)}
	\arrow["\sim", from=1-1, to=1-3]
	\arrow[from=1-1, to=2-1]
	\arrow[from=1-3, to=2-3]
	\arrow["\sim", from=2-1, to=2-2]
	\arrow["\sim", from=2-2, to=2-3]
\end{tikzcd}\]
Because $V$ and $W$ are finite dimensional, the map
\[
L \otimes_K \Hom_{K[H_0]}(V,W) \rightarrow \Hom_{L[H_0]}(V_L, W_L)
\]
is an isomorphism, and therefore the right-hand map is too. Then (\ref{eqn:naturalmapforDX}) is an isomorphism, for
\[
L \otimes_K \Hom_{\sD_X}(\sV,\sW) \rightarrow L \hspace{0.1em} \widehat{\otimes}_K \Hom_{\sD_X}(\sV,\sW)
\]
is an isomorphism, $\Hom_{\sD_X}(\sV,\sW)$ being finite dimensional over $K$.
\end{proof}

\begin{remark}\label{rem:removeassumptions}
	Taking $G = 1$, $\sV = \OO_X$, and $\sW = \OO_X$ in Proposition \ref{prop:basechangecommuteswithhoms}, then using Lemma \ref{altdesccX} the conclusion of Proposition \ref{prop:basechangecommuteswithhoms} is equivalent to the statement that the natural map
	\[
	L \hspace{0.1em} \widehat{\otimes}_K c_X(X) \rightarrow c_{X_L}(X_L)
	\]
	is an isomorphism. In this case the assumptions of Proposition \ref{prop:basechangecommuteswithhoms} can be relaxed: one can show that this also holds for any quasi-separated rigid space $K$ using the results of Section \ref{sect:constantsheaf} and \cite[Thm.\ 3.2.1]{CON}.
\end{remark}

\bibliography{biblio}{}
\bibliographystyle{plain}

\vspace{2em}
\end{document}